\documentclass[11pt]{article}

\usepackage{geometry}
\usepackage{latexsym}
\usepackage{mathrsfs}
\usepackage{pifont}
\usepackage{makecell}
\usepackage{amsfonts}
\usepackage{amssymb}
\usepackage{microtype}
\usepackage{subfigure}    
\usepackage{graphicx}
\usepackage{epstopdf}
\usepackage{float}
\usepackage{multirow}
\usepackage{bm}
\usepackage{amsmath}
\usepackage{amsthm}
\usepackage{color}
\usepackage[subnum]{cases}
\usepackage{rotating}
\usepackage{enumitem}
\usepackage{enumerate}
\usepackage{accents}
\usepackage{algorithm}
\usepackage{algorithmic}
\usepackage[title]{appendix}
\usepackage{mathtools}
\usepackage{caption}
\usepackage{url}
\usepackage{booktabs}
\usepackage{tcolorbox}
\usepackage{verbatim}

\makeatletter
\@addtoreset{equation}{section}
\makeatother

\definecolor{blue}{rgb}{0,0,0.9}
\definecolor{red}{rgb}{0.9,0,0}
\definecolor{green}{rgb}{0,0.9,0}
\definecolor{brown}{rgb}{0.6,0.1,0.1}
\definecolor{lightgreen}{rgb}{0.1,0.5,0.1}

\usepackage[colorlinks=true,
breaklinks=true,
bookmarks=true,
urlcolor=blue,
citecolor=lightgreen,
linkcolor=lightgreen,
bookmarksopen=false,
draft=false]{hyperref}

\graphicspath{{images/}}

\begin{document}

\newtheorem{property}{Property}[section]
\newtheorem{proposition}{Proposition}[section]
\newtheorem{append}{Appendix}[section]
\newtheorem{definition}{Definition}[section]
\newtheorem{lemma}{Lemma}[section]
\newtheorem{corollary}{Corollary}[section]
\newtheorem{theorem}{Theorem}[section]
\newtheorem{remark}{Remark}[section]
\newtheorem{problem}{Problem}[section]
\newtheorem{example}{Example}[section]
\newtheorem{assumption}{Assumption}
\renewcommand*{\theassumption}{\Alph{assumption}}


\title{Fast and Effective Computation of Generalized Symmetric Matrix Factorization}

\author{Lei Yang\thanks{School of Computer Science and Engineering, and Guangdong Province Key Laboratory of Computational Science, Sun Yat-sen University ({\tt yanglei39@mail.sysu.edu.cn}). 
},~~
Han Wan\thanks{School of Computer Science and Engineering, Sun Yat-sen University ({\tt wanh23@mail2.sysu.edu.cn}).},~~
Min Zhang\thanks{School of Mathematics and Information Science, Guangzhou University ({\tt zhangmin1206@gzhu.edu.cn}). 
},~~
Ling Liang\thanks{(Corresponding author) Department of Mathematics, The University of Tennessee, Knoxville ({\tt liang.ling@u.nus.edu}).}
}

\date{}

\maketitle

\begin{abstract}
In this paper, we study a nonconvex, nonsmooth, and non-Lipschitz generalized symmetric matrix factorization model that unifies a broad class of matrix factorization formulations arising in machine learning, image science, engineering, and related areas. We first establish two exactness properties. On the modeling side, we prove an exact penalty property showing that, under suitable conditions, the symmetry-inducing quadratic penalty enforces symmetry whenever the penalty parameter is sufficiently large but finite, thereby exactly recovering the associated symmetric formulation. On the algorithmic side, we introduce an auxiliary-variable splitting formulation and establish an exact relaxation relationship that rigorously links stationary points of the original objective function to those of a relaxed potential function. Building on these exactness properties, we propose an average-type nonmonotone alternating updating method (A-NAUM) based on the relaxed potential function. At each iteration, A-NAUM alternately updates the two factor blocks by (approximately) minimizing the potential function, while the auxiliary block is updated in closed form. To ensure the convergence and enhance practical performance, we further incorporate an average-type nonmonotone line search and show that it is well-defined under mild conditions. Moreover, based on the Kurdyka-{\L}ojasiewicz property and its associated exponent, we establish global convergence of the entire sequence to a stationary point and derive convergence rate results. Finally, numerical experiments on real datasets demonstrate the efficiency of A-NAUM.


\vspace{5mm}
\end{abstract}

\section{Introduction}\label{sec-intro}

Matrix factorization (MF) and its symmetric variants are central tools for learning low-dimensional and reduced-order latent representations from high-dimensional data. By combining low-rank structure with application-driven regularizers (e.g., nonnegativity, (group) sparsity, structured constraints), these models can encode prior structural information and promote interpretable latent factors \cite{cza2006new,g2011nonnegative,ls1999learning,sdbssmp2004nonnegative,uhzb2016generalized,zyzy2014low}. From an optimization perspective, factorized parameterizations are also essential in large-scale models, as they reduce computational and storage costs and enable scalable algorithmic implementations \cite{hswalwwc2022lora,jvz2014speeding,lltt2024accelerating,xxgcczcct2023qa}. Despite these advantages, algorithmic design and convergence analysis for MF models are typically tailored to specific formulations. In practice, however, it is often necessary to study families of MF models or to compare multiple regularization strategies within a single application. This consideration motivates the development of \textit{a unified model and a flexible optimization framework} that accommodates broad regularization structures while preserving scalability and admitting rigorous convergence guarantees. 

In this paper, we propose the following generalized symmetric matrix factorization problem:
\begin{equation}\label{quadpenaltypro}
\min_{X,\,Y}~~\mathcal{F}_{\lambda}(X,\,Y)
:=\Psi(X)+\Phi(Y)+\frac{1}{2}\|\mathcal{A}(XY^\top)-\bm{b}\|^2
+\frac{\lambda}{2}\|X-Y\|_F^2, \tag{GSMF}
\end{equation}
where $X=[\bm{x}_1,\ldots,\bm{x}_r]\in\mathbb{R}^{n\times r}$ and $Y=[\bm{y}_1,\ldots,\bm{y}_r]\in\mathbb{R}^{n\times r}$ are decision matrices with $r\leq n$; $\Psi,\Phi:\mathbb{R}^{n\times r}\to\mathbb{R}\cup\{\infty\}$ are proper closed functions that may be \emph{nonconvex, nonsmooth, and even non-Lipschitz}; $\bm{b}\in\mathbb{R}^q$ is a given vector with $q\leq n^2$; and $\mathcal{A}:\mathbb{R}^{n\times n}\to\mathbb{R}^q$ is a linear mapping with $\mathcal{A}\mathcal{A}^*$ being the identity map on $\mathbb{R}^q$. 
The penalty parameter $\lambda\geq0$ controls the extent to which symmetry is promoted by penalizing the discrepancy between $X$ and $Y$. Note that when $\lambda=0$, the two factor matrices need not have the same number of rows; however, for a unified presentation and without loss of generality, we assume throughout that $X$ and $Y$ share the same number of rows.

The model \eqref{quadpenaltypro} serves as a flexible formulation that subsumes a broad class of well-studied matrix factorization models. When $\lambda=0$, it reduces to regularized matrix factorization (RMF); with appropriate choices of $(\Psi,\Phi,\mathcal{A})$, this setting encompasses, for example, nonnegative matrix factorization (NMF) and matrix completion (MC) \cite{g2011nonnegative,ls1999learning,ls2000algorithms,nhd2012low,pt1994positive,slc2016scalable,slc2016tractable}. When $\lambda>0$ and $\Psi=\Phi$, the term $\tfrac{\lambda}{2}\|X-Y\|_F^2$ acts as a quadratic penalty that promotes $X\approx Y$, thereby encouraging symmetry in the factorization. In particular, as $\lambda\to\infty$, the classical penalty theory \cite[Section 17.1]{nw2006numerical} suggests that \eqref{quadpenaltypro} approaches the following symmetric regularized matrix factorization problem
\begin{equation}\label{SMFpro}
\min_{X\in\mathbb{R}^{n \times r}}
\; 2\Phi(X)+\frac{1}{2}\|\mathcal{A}(XX^\top)-\bm{b}\|^2, \tag{SRMF}
\end{equation}
in which the factor matrices are explicitly tied and the decision variable reduces to a single matrix $X$. Beyond its role as a modeling engine, symmetric matrix factorization also serves as an important computational primitive in matrix optimization. In particular, it arises naturally in low-rank approaches to semidefinite programming, most notably through the Burer-Monteiro factorization, where a positive semidefinite matrix variable $Z\succeq 0$ is parameterized as $Z=XX^\top$ and optimization is then carried out directly over the factor $X$; see, e.g., \cite{bm2003nonlinear,bm2005local,lltt2024accelerating,ylct2023inexact}.

Although the quadratic penalty term $\frac{\lambda}{2}\|X-Y\|_F^2$ in the model \eqref{quadpenaltypro} is intended to enforce $X=Y$ and thereby approximate the symmetric model \eqref{SMFpro}, it is not evident a priori that a \emph{finite} penalty parameter $\lambda$ suffices. To better justify \eqref{quadpenaltypro} as a practically meaningful surrogate for \eqref{SMFpro}, it is necessary to establish an \textit{exact penalty} property that goes beyond the classical quadratic penalty theory, namely, that for all sufficiently large but finite $\lambda$, every stationary point of \eqref{quadpenaltypro} satisfies $X=Y$ and corresponds to a stationary point of \eqref{SMFpro}. Establishing such an exactness property is essential for validating the practical relevance of the proposed general model, and constitutes the first main objective of this work.

On the algorithmic side, a dominant computational paradigm for various matrix factorization models is alternating optimization, which exploits the fact that the problem becomes significantly simpler when a certain part of the decision variables is fixed. This paradigm underlies many practical algorithms for NMF, MC and their variants, including multiplicative updates, alternating nonnegative least squares (ANLS), and hierarchical alternating least squares (HALS); see, e.g., \cite{cza2007hierarchical,gg2012accelerated,kp2008nonnegative,kp2008toward}. 
For the model \eqref{quadpenaltypro} with $\lambda=0$, Yang et al.~\cite{ypc2018nonmonotone} proposed a unified alternating framework based on a potential function combined with a nonmonotone line search. However, the convergence guarantees in~\cite{ypc2018nonmonotone} are limited to the subsequential convergence; \textit{neither} convergence of the entire sequence \textit{nor} convergence rates were established. Moreover recently, Li et al.~\cite{lzll2023provable} studied a penalized formulation of symmetric NMF, which corresponds to a special case of \eqref{quadpenaltypro} with $\mathcal{A}$ being the vectorization operator and $\Psi,\Phi$ being the indicator functions of $\mathbb{R}^{n\times r}_+$. By extending classical ANLS and HALS schemes for NMF, they proposed the SymANLS and SymHALS methods. While effective in practice, both the algorithmic design and theoretical analysis in~\cite{lzll2023provable} are tailored to this specific setting and do not directly extend to the more general formulation~\eqref{quadpenaltypro} with broader choices of $(\Psi,\Phi,\mathcal{A})$. 

Given the broad modeling scope of \eqref{quadpenaltypro}, there remains a clear need for an efficient, scalable, and globally convergent alternating optimization framework that applies to general choices of $(\Psi,\Phi,\mathcal{A},\lambda)$. Addressing this need constitutes the second main objective of this work. Despite the appeal of alternating updates in $(X,\,Y)$, the composite term $\mathcal{A}(XY^\top)$ can complicate the resulting subproblems, as the bilinear product $XY^\top$ is tightly coupled with the linear map $\mathcal{A}$. In particular, this coupling precludes the use of widely adopted and efficient hierarchical (column-wise) updating schemes (see \eqref{supro:U-hier} and \eqref{supro:V-hier}) in NMF. These difficulties motivate the development of splitting and relaxation strategies beyond ``direct" alternating minimization on $\mathcal{F}_\lambda$. To this end, we introduce an auxiliary variable $Z$ and consider the potential function
\begin{equation}\label{thetapofun}
\Theta_{\alpha,\beta,\lambda}(X,\,Y,\,Z)
:= \Psi(X)+\Phi(Y)+\frac{\alpha}{2}\|XY^\top-Z\|_F^2
+\frac{\beta}{2}\|\mathcal{A}(Z)-\bm{b}\|^2+\frac{\lambda}{2}\|X-Y\|_F^2,
\end{equation}
where $\alpha$ and $\beta$ are relaxation parameters, one of which is allowed to be \textit{negative}. This formulation can be interpreted as introducing the auxiliary constraint $Z=XY^\top$ and relaxing it via a quadratic penalty. By decoupling the bilinear product $XY^\top$ from the linear map $\mathcal{A}$, this relaxation yields much simpler subproblems in alternating schemes and, in particular, admits an explicit update for $Z$. Our algorithmic framework is therefore developed based on $\Theta_{\alpha,\beta,\lambda}$. Moreover, we incorporate an average-type nonmonotone line search to ensure convergence of the entire sequence and to further improve numerical performance.

The main contributions of this paper are summarized as follows:
\begin{itemize}
\item \textbf{Exact penalty and exact relaxation.} On the modeling side, we identify conditions under which the quadratic penalty in \eqref{quadpenaltypro} is \emph{exact}: for sufficiently large but finite $\lambda$, stationary points of \eqref{quadpenaltypro} satisfy $X=Y$ and thus recover stationary points of the symmetric model \eqref{SMFpro}; see Section~\ref{exact-penal}. Moreover, to facilitate the algorithmic development, we introduce the potential function $\Theta_{\alpha,\beta,\lambda}$ in \eqref{thetapofun} by decoupling $XY^\top$ from $\mathcal{A}$ through an auxiliary variable $Z$, and establish an \emph{exact relaxation} relationship between $\Theta_{\alpha,\beta,\lambda}$ and the original objective $\mathcal{F}_\lambda$ under mild conditions; see Section~\ref{exact-relax}.

\item \textbf{A flexible and globally convergent alternating optimization framework.} On the algorithmic side, we develop an average-type nonmonotone alternating updating method (A-NAUM) based on $\Theta_{\alpha,\beta,\lambda}$, which accommodates a variety of efficient block updating strategies for $X$ and $Y$ together with an explicit update for $Z$. This yields a flexible algorithmic framework applicable to broad choices of $(\Psi,\Phi,\mathcal{A},\lambda)$; see Section~\ref{sec-algo}. We further provide a comprehensive convergence analysis of A-NAUM in Section~\ref{sec-convanal}. In particular, despite the nonmonotone nature of the algorithm, we establish convergence of the entire sequence as well as convergence rate results under the Kurdyka--{\L}ojasiewicz property and its associated exponent.

\item \textbf{Numerical validation.} We conduct numerical experiments on real datasets for approximate symmetric nonnegative matrix factorization to demonstrate the efficiency of A-NAUM and its competitive performance against representative baselines; see Section~\ref{sec-numer}.

\end{itemize}

The rest of this paper is organized as follows. Section~\ref{sec-nota} introduces notation and preliminaries. Section~\ref{sec-exresult} establishes the exact penalty and exact relaxation properties. In Section~\ref{sec-algo}, we present A-NAUM for solving \eqref{quadpenaltypro}, followed by a comprehensive convergence analysis in Section~\ref{sec-convanal}. Numerical experiments are reported in Section~\ref{sec-numer}, and concluding remarks are given in Section~\ref{sec-conclu}.

\section{Notation and preliminaries}\label{sec-nota}

In this paper, scalars, vectors, and matrices are denoted by lowercase letters, bold lowercase letters, and uppercase letters, respectively. We use $\mathbb{N}$, $\mathbb{R}$ ($\mathbb{R}_+$), $\mathbb{R}^n$ ($\mathbb{R}^n_+$), and $\mathbb{R}^{m\times n}$ ($\mathbb{R}^{m\times n}_+$) to denote the sets of natural numbers, real (nonnegative) numbers, $n$-dimensional real (nonnegative) vectors, and $m\times n$ real (nonnegative) matrices, respectively. For a vector $\bm{x}\in\mathbb{R}^n$, $\bm{x}_i$ denotes its $i$th entry and $\|\bm{x}\|$ denotes its Euclidean norm. For a matrix $X\in\mathbb{R}^{m\times n}$, $X_{ij}$ denotes its $(i,j)$th entry and $\bm{x}_j$ denotes its $j$th column. We write $\mathrm{tr}(X)$ for the trace of $X$, $\|X\|_F:=\sqrt{\sum_{i=1}^m\sum_{j=1}^n X_{ij}^2}$ for the Frobenius norm, and $\|X\|$ for the spectral norm (i.e., the largest singular value). For matrices $X$ and $Y$ of the same size, $\langle X,\,Y\rangle:=\sum_{i=1}^m\sum_{j=1}^n X_{ij}Y_{ij}$ denotes the inner product, and $X\leq Y$ (resp., $X\geq Y$) means $X_{ij}\leq Y_{ij}$ (resp., $X_{ij}\geq Y_{ij}$) for all $(i,j)$. For a nonempty closed set $\mathcal{C}\subseteq\mathbb{R}^{m\times n}$, $\delta_{\mathcal{C}}$ denotes its indicator function, i.e., $\delta_{\mathcal{C}}(X)=0$ if $X\in\mathcal{C}$ and $\delta_{\mathcal{C}}(X)=+\infty$ otherwise. The distance from $X$ to $\mathcal{C}$ is denoted by $\operatorname{dist}(X,\,\mathcal{C}):=\inf_{Y\in\mathcal{C}}\|X-Y\|_F$.

For a linear map $\mathcal{A}:\mathbb{R}^{m\times n}\to\mathbb{R}^q$, $\mathcal{A}^*$ denotes its adjoint. We use $\texttt{vec}:\mathbb{R}^{m\times n}\to\mathbb{R}^{mn}$ to denote the vectorization map defined by $[\texttt{vec}(X)]_{i+(j-1)m}=X_{ij}$ for all $(i,j)$,
and $\texttt{mat}:\mathbb{R}^{mn}\to\mathbb{R}^{m\times n}$ to denote the corresponding matrixization map defined by $[\texttt{mat}(\bm{b})]_{ij}=\bm{b}_{i+(j-1)m}$ for all $(i,j)$.
It is straightforward to verify that $\texttt{vec}$ and $\texttt{mat}$ are inverse maps and adjoint to each other. Moreover, we use $\mathcal{I}_q$ and $\mathcal{I}$ to denote the identity maps on $\mathbb{R}^q$ and $\mathbb{R}^{n\times n}$, respectively. 

For an extended-real-valued function $h: \mathbb{R}^{m \times n} \to \mathbb{R} \cup \{ + \infty \}$, we say that it is \textit{proper} if $h(X) > - \infty$ for all $X \in \mathbb{R}^{m \times n}$ and its effective domain, denoted as ${\rm dom}\,h := \{ X \in \mathbb{R}^{m \times n} : h(X) <  \infty \}$, is nonempty. Moreover, $h$ is said to be closed if it is lower semicontinuous, i.e., $h(X) \leq \liminf\limits_{k\to\infty}h(X^k)$ for every sequence $X^k\to X$ as $k\to \infty$. We also use the notation $Y \xrightarrow{h} X$ to denote $Y \rightarrow X$ (i.e., $\|Y-X\|_F\to0$) and $h(Y) \rightarrow h(X)$. The (limiting) subdifferential \cite[Definition~8.3]{rw1998variational} of $h$ at $X\in \mathrm{dom}\,h$, denoted by $\partial h(X)$, is defined as
\begin{equation*}
\partial h(X):=\left\{ D \in \mathbb{R}^{m\times n}: \exists\,X^k \xrightarrow{h} X~\mathrm{and}~D^k \rightarrow D ~\mathrm{with}~D^k \in \widehat{\partial} h(X^k) ~\mathrm{for~all}~k\right\},
\end{equation*}
where $\widehat{\partial} h(\widetilde{Y})$ denotes the Fr\'{e}chet subdifferential of $h$ at $\widetilde{Y}\in \mathrm{dom}\,h$, which is the set of all $D \in \mathbb{R}^{m\times n}$ satisfying
\begin{equation*}
\liminf\limits_{Y \neq \widetilde{Y}, \,Y \rightarrow \widetilde{Y}} \frac{h(Y)-h(\widetilde{Y})-\langle D, \,Y-\widetilde{Y}\rangle}{\|Y-\widetilde{Y}\|_F} \geq 0.
\end{equation*}
From the above definition, \cite[Proposition~8.7]{rw1998variational} indicates that
\begin{equation}\label{robust}
\left\{ D\in\mathbb{R}^{m\times n}: \exists\, X^k \xrightarrow{h} X,~ D^k \rightarrow D~\mathrm{with}~D^k \in \partial h(X^k)~\mathrm{for~all}~k \right\} \subseteq \partial h(X).
\end{equation}
When $h$ is continuously differentiable or convex, the above subdifferential coincides with the classical concept of gradient or convex subdifferential of $h$, respectively; see, e.g., \cite[Exercise~8.8]{rw1998variational} and \cite[Proposition~8.12]{rw1998variational}. For $\kappa \geq 0$, a proper closed function $h$ is said to be $\kappa$-weakly convex if $h+\frac{\kappa}{2}\|\cdot\|_F^2$ is convex.

For a proper closed function $h: \mathbb{R}^{m \times n}\to\mathbb{R}\cup\{+\infty\}$,  the generalized Fermat's rule \cite[Theorem 10.1]{rw1998variational} states that, if $h$ attains a local minimum at $\widetilde{X}$, then $0 \in \partial h(\widetilde{X})$. Accordingly, throughout this paper, we say that $\widetilde{X}$ is a \textit{stationary point} of the function $h: \mathbb{R}^{m \times n} \to \mathbb{R}\cup\{+\infty\}$, or its associated minimization problem $\min\limits_{X\in\mathbb{R}^{m\times n}}\left\{h(X)\right\}$ if $0 \in \partial h(\widetilde{X})$.

We next recall the Kurdyka-{\L}ojasiewicz (KL) property (see \cite{abs2013convergence,bdl2007the,bst2014proximal,lp2017calculus} for more details), which is a widely used condition for establishing the convergence of the whole sequence.
For simplicity, let $\Phi_{\nu}$ ($\nu>0$) denote a class of concave functions $\varphi:[0,\nu) \rightarrow \mathbb{R}_{+}$ satisfying: (i) $\varphi(0)=0$; (ii) $\varphi$ is continuously differentiable on $(0,\nu)$ and continuous at $0$; (iii) $\varphi'(t)>0$ for all $t\in(0,\nu)$. Then, the KL property and its associated exponent can be described as follows.

\begin{definition}[\textbf{KL property and exponent}]\label{property-KL}
A proper closed function $h: \mathbb{R}^{m\times n} \rightarrow \mathbb{R} \cup \{+\infty\}$ is said to satisfy the KL property at $\widetilde{X}\in{\rm dom}\,\partial h:=\{X \in \mathbb{R}^{m\times n}: \partial h(X) \neq \emptyset\}$, if there exist a $\nu\in(0, +\infty]$, a neighborhood $\mathcal{V}$ of $\widetilde{X}$, and a function $\varphi \in \Phi_{\nu}$ such that for all $X \in \mathcal{V} \cap \{X\in\mathbb{R}^{m\times n} : h(\widetilde{X})<h(X)<h(\widetilde{X})+\nu\}$, it holds that
\begin{equation*}
\varphi'(h(X)-h(\widetilde{X}))\,\operatorname{dist}(0, \,\partial h(X)) \geq 1.
\end{equation*}
The function $h$ is called a KL function, if $h$ satisfies the KL property at each point of ${\rm dom}\,\partial h$. Furthermore, the function $h$ is said to be a KL function with an exponent $\theta$ if $\varphi$ can be chosen as $\varphi(t)=\tilde{a}t^{1-\theta}$ for some $\tilde{a} > 0$ and $\theta\in[0,1)$.
\end{definition}

The uniformized KL property, which was established in \cite[Lemma 6]{bst2014proximal} is stated as follows.

\begin{proposition}[\textbf{Uniformized KL property}]\label{uniKL}
Suppose that $h: \mathbb{R}^{m\times n} \rightarrow \mathbb{R} \cup \{+\infty\}$ is a proper closed function and $\Gamma$ is a compact set. If $h \equiv \zeta$ on $\Gamma$ for some constant $\zeta$ and satisfies the KL property at each point of $\Gamma$, then there exist $\varepsilon>0$, $\nu>0$, and $\varphi \in \Phi_{\nu}$ such that
\begin{equation*}
\varphi'(h(X) - \zeta)\operatorname{dist}(0, \,\partial h(X)) \geq 1
\end{equation*}
for all $X \in \{X\in\mathbb{R}^{m\times n}: \operatorname{dist}(X,\,\Gamma)<\varepsilon\} \cap \{X\in \mathbb{R}^{m\times n} : \zeta < h(X) < \zeta + \nu\}$.
\end{proposition}

\section{Exact penalty and exact relaxation}\label{sec-exresult}

In this section, we establish two key connections: (i) the relationship between the penalized formulation \eqref{quadpenaltypro} and the symmetric model \eqref{SMFpro}; and (ii) the relationship between the original objective $\mathcal{F}_{\lambda}$ and the relaxed potential function $\Theta_{\alpha,\beta,\lambda}$ in \eqref{thetapofun}. The first result formalizes how the quadratic penalty in \eqref{quadpenaltypro} can enforce $X=Y$ and thereby recover a symmetric factorization, while the second provides the theoretical justification for optimizing the relaxed formulation and serves as the foundation for our subsequent algorithmic development.

\subsection{Exact penalty results}\label{exact-penal}

Motivated by \cite[Section 2.2]{lzll2023provable}, we show that, although \eqref{quadpenaltypro} is formulated as a \textit{quadratic} penalty relaxation of \eqref{SMFpro}, it nevertheless admits an exact penalty property under suitable conditions, which goes beyond the scope of classical results in, e.g., \cite[Section 17.1]{nw2006numerical}.

\begin{theorem}\label{thm:2equal1}
Suppose that $\Psi=\Phi$ and $\Phi$ is $\kappa$-weakly convex, and that the linear map $\mathcal{A}:\mathbb{R}^{n\times n}\to\mathbb{R}^q$ and the vector $\bm b\in\mathbb{R}^q$ satisfy: {\rm(i)} $\mathcal{A}^*\bm b$ is symmetric; and {\rm(ii)} for any $U\in\mathbb{R}^{n\times n}$,
\begin{equation*}
\mathcal{A}^*\mathcal{A}(U)-\big(\mathcal{A}^*\mathcal{A}(U)\big)^\top
~=~\mathcal{A}^*\mathcal{A}\big(U-U^\top\big).
\end{equation*}
Let $(\widetilde X,\widetilde Y)$ be an arbitrary stationary point of \eqref{quadpenaltypro}. If the penalty parameter $\lambda$ satisfies
\begin{equation*}
2\lambda>\big\|\mathcal{A}^*\mathcal{A}(\widetilde X\widetilde Y^\top)\big\|+\kappa
-\rho_{\min}(\mathcal{A}^*\bm b),
\end{equation*}
where $\rho_{\min}(\mathcal{A}^*\bm b)$ denotes the smallest eigenvalue of $\mathcal{A}^*\bm b$, then $\widetilde X=\widetilde Y$, and $\widetilde X$ is a stationary point of \eqref{SMFpro}.
\end{theorem}
\begin{proof}
See Appendix \ref{proof:2equal1}.
\end{proof}

From Theorem \ref{thm:2equal1}, we see that any stationary point of \eqref{quadpenaltypro} is also a stationary point of \eqref{SMFpro}, provided that the penalty parameter $\lambda > 0$ is sufficiently large but finite. This shows that the quadratic penalty model \eqref{quadpenaltypro} can enforce $\widetilde X=\widetilde Y$ at stationarity, thereby effectively inducing symmetry between the factor matrices. Moreover, this result extends \cite[Theorem~1]{lzll2023provable}, which was established only for symmetric nonnegative matrix factorization, to a substantially broader class of regularized matrix factorization models. Finally, the conditions imposed on $(\mathcal{A},\bm{b})$ can be satisfied, for example, when $\mathcal{A}$ is the vectorization map from $\mathbb{R}^{n\times n}$ to $\mathbb{R}^{n^2}$ and $\texttt{mat}(\bm{b})$ (i.e., the matrixization of $\bm{b}$) is symmetric. In addition, as shown in the following proposition, these conditions can also be satisfied when $\mathcal{A}=\mathcal{P}_\Omega$ is a symmetric sampling map and $\mathcal{P}_\Omega^*(\bm{b})$ is symmetric, where $\mathcal{P}_\Omega^*$ denotes the adjoint map of $\mathcal{P}_\Omega$.


\begin{proposition}\label{prop-samplingmap}
Let $q\leq n^2$ and let $\Omega:=\{(i_1,j_1),(i_2,j_2),\ldots,(i_q,j_q)\}\subseteq \{1,\ldots,n\}\times\{1,\ldots,n\} $ satisfy: {\rm(i)} the elements of $\Omega$ are ordered lexicographically, with priority given first to the column index and then to the row index; and {\rm(ii)} $\Omega$ is symmetric in the sense that $(i,j)\in\Omega$ if and only if $(j,i)\in\Omega$. Define the sampling map $\mathcal{P}_\Omega:\mathbb{R}^{n\times n}\to\mathbb{R}^q$ by
\begin{equation*}
\mathcal{P}_\Omega(U):=(U_{i_1j_1},\ldots,U_{i_qj_q})^\top,
\qquad\forall\,U\in\mathbb{R}^{n\times n}.
\end{equation*}
Then, for any $U\in\mathbb{R}^{n\times n}$, $\mathcal{P}_\Omega^*\mathcal{P}_\Omega(U)-\big(\mathcal{P}_\Omega^*\mathcal{P}_\Omega(U)\big)^\top
~=~\mathcal{P}_\Omega^*\mathcal{P}_\Omega(U-U^\top)$.  
\end{proposition}
\begin{proof}
See Appendix \ref{proof:prop-samplingmap}.
\end{proof}

\subsection{Exact relaxation results}\label{exact-relax}

In this subsection, we establish precise relationships between the objective function $\mathcal{F}_{\lambda}$ and its relaxed potential function $\Theta_{\alpha,\beta,\lambda}$. Specifically, under suitable conditions, we show that problem \eqref{quadpenaltypro} is exactly equivalent to minimizing $\Theta_{\alpha,\beta,\lambda}$. When these conditions are further relaxed, this equivalence weakens to a correspondence between stationary points of $\mathcal{F}_{\lambda}$ and $\Theta_{\alpha,\beta,\lambda}$. Our analysis follows arguments similar to those in \cite[Section 3]{ypc2018nonmonotone}. For completeness, detailed proofs of the main results are provided in Appendices \ref{proof:equal-min} and \ref{proof:equal-sta}. 

\begin{theorem}\label{thm:equal-min}
Suppose that $\mathcal{A}\mathcal{A}^* = \mathcal{I}_q$. If $\alpha$ and $\beta$ are chosen such that $\alpha \mathcal{I} + \beta \mathcal{A}^* \mathcal{A} \succ 0$ and $\frac{1}{\alpha} + \frac{1}{\beta} = 1$, then problem $\min_{X,Y,Z}\big\{\Theta_{\alpha,\beta,\lambda}(X,\,Y,\,Z)\big\}$ is equivalent to problem \eqref{quadpenaltypro}.
\end{theorem}
\begin{proof}
See Appendix \ref{proof:equal-min}.
\end{proof}

Theorem~\ref{thm:equal-min} shows that problem \eqref{quadpenaltypro} is equivalent to minimizing $\Theta_{\alpha,\beta,\lambda}$ when $\mathcal{A}\mathcal{A}^*=\mathcal{I}_q$ and $\alpha,\beta$ are chosen appropriately. Furthermore, the following theorem establishes that, under more relaxed conditions on $\alpha$ and $\beta$, a rigorous correspondence between stationary points of the original objective $\mathcal{F}_{\lambda}$ and those of the potential function $\Theta_{\alpha,\beta,\lambda}$ can still be obtained.

\begin{theorem}\label{thm:equal-sta}
Suppose that $\mathcal{A}\mathcal{A}^*=\mathcal{I}_{q}$ and $\alpha$, $\beta$ are chosen such that $\frac{1}{\alpha}+\frac{1}{\beta}=1$. Then, the following statements hold.
\begin{itemize}
\item [{\rm(i)}] If $(\widetilde{X}, \,\widetilde{Y}, \,\widetilde{Z})$ is a stationary point of $\Theta_{\alpha,\beta,\lambda}$, then $(\widetilde{X}, \,\widetilde{Y})$ is a stationary point of $\mathcal{F}_{\lambda}$.

\item [{\rm(ii)}] If $(\widetilde{X}, \,\widetilde{Y})$ is a stationary point of $\mathcal{F}_{\lambda}$, then $(\widetilde{X}, \,\widetilde{Y}, \,\widetilde{Z})$ is a stationary point of $\Theta_{\alpha,\beta,\lambda}$, where $\widetilde{Z}$ is given by
    \begin{equation}\label{eq:Zbar=}
    \widetilde{Z} = \textstyle\left(\mathcal{I} - \frac{\beta}{\alpha+\beta}\mathcal{A}^*\mathcal{A}\right)(\widetilde{X}\widetilde{Y}^{\top}) + \frac{\beta}{\alpha+\beta}\mathcal{A}^*(\bm{b}).
    \end{equation}
\end{itemize}
\end{theorem}
\begin{proof}
See Appendix \ref{proof:equal-sta}.
\end{proof}

Theorem \ref{thm:equal-sta} implies that, when $\mathcal{A}\mathcal{A}^*=\mathcal{I}_q$ and $\alpha$, $\beta$ are chosen appropriately, a stationary point of $\mathcal{F}_{\lambda}$ can be recovered from a stationary point of $\Theta_{\alpha,\beta,\lambda}$. Since the linear map $\mathcal{A}$ is no longer coupled with the bilinear term $XY^{\top}$ in $\Theta_{\alpha,\beta,\lambda}$, the task of computing a stationary point of $\Theta_{\alpha,\beta,\lambda}$ becomes considerably more tractable, especially in the regime of alternating optimization. It is therefore both natural and advantageous to design algorithms based on $\Theta_{\alpha,\beta,\lambda}$, rather than directly on the original objective function $\mathcal{F}_{\lambda}$, to obtain a stationary point of $\mathcal{F}_{\lambda}$.

Compared with Theorem \ref{thm:equal-min}, the conditions on $\alpha$ and $\beta$ in Theorem \ref{thm:equal-sta} are considerably more relaxed. Indeed, the requirements $\mathcal{A}\mathcal{A}^*=\mathcal{I}_q$, $\alpha \mathcal{I} + \beta \mathcal{A}^*\mathcal{A} \succ 0$, and $\frac{1}{\alpha} + \frac{1}{\beta} = 1$ in Theorem \ref{thm:equal-min} readily imply $\alpha>1$ and $\beta=\frac{\alpha}{\alpha-1}>1$. However, as indicated by the numerical results in Section \ref{sec-numer}, such choices of $\alpha>1$ do not necessarily yield the best empirical performance. In contrast, the weaker requirements $\mathcal{A}\mathcal{A}^*=\mathcal{I}_q$ and $\frac{1}{\alpha} + \frac{1}{\beta} = 1$ in Theorem \ref{thm:equal-sta} offer much greater flexibility in choosing $\alpha$ and $\beta$, and suitable choices can lead to significantly improved numerical performance. Motivated by these observations, in the next section, we build on this exact relaxation framework and develop an efficient alternating optimization algorithm for computing a stationary point of $\Theta_{\alpha,\beta,\lambda}$ and hence of $\mathcal{F}_\lambda$, under the weaker conditions $\mathcal{A}\mathcal{A}^*=\mathcal{I}_{q}$ and $\frac{1}{\alpha}+\frac{1}{\beta}=1$.

\section{An average-type nonmonotone alternating updating method}\label{sec-algo}

In this section, we develop an average-type nonmonotone alternating updating method (A-NAUM) to find a stationary point of $\Theta_{\alpha,\beta,\lambda}$ and hence of $\mathcal{F}_\lambda$, under the conditions $\mathcal{A}\mathcal{A}^*=\mathcal{I}_{q}$ and $\frac{1}{\alpha}+\frac{1}{\beta}=1$. For notational simplicity, we define
\begin{equation}\label{eq:definerho}
\rho \;:=\;\left\|\mathcal{I}-{\textstyle\frac{\beta}{\alpha+\beta}}\mathcal{A}^*\mathcal{A}\right\|^2,
\end{equation}
and choose a nonnegative scalar $\gamma\geq 0$ such that
\begin{equation}\label{eq:definegamma}
(\alpha+\gamma)\mathcal{I}+\beta\mathcal{A}^*\mathcal{A}\succeq 0.
\end{equation}
Since $\mathcal{A}\mathcal{A}^*=\mathcal{I}_q$, the eigenvalues of $\mathcal{A}^*\mathcal{A}$ belong to $\{0,1\}$ and then the eigenvalues of $\mathcal{I}-\frac{\beta}{\alpha+\beta}\mathcal{A}^*\mathcal{A}$ are either $1$ or $\frac{\alpha}{\alpha+\beta}$, which yields $\rho=\max\!\left\{1,\;\frac{\alpha^2}{(\alpha+\beta)^2}\right\}$. Similarly, the eigenvalues of $(\alpha+\gamma)\mathcal{I}+\beta\mathcal{A}^*\mathcal{A}$ are $(\alpha+\gamma)$ or $(\alpha+\beta+\gamma)$, and thus condition \eqref{eq:definegamma} holds whenever $\gamma\geq \max\{0,\,-\alpha,\,-(\alpha+\beta)\}$. The complete A-NAUM framework is given in Algorithm~\ref{algo-A-NAUM}.

\begin{algorithm}[ht]
\caption{A-NAUM for computing a stationary point of $\mathcal{F}_\lambda$}\label{algo-A-NAUM}
\textbf{Input:} $(X^0,\,Y^0)$, $\alpha$ and $\beta$ satisfying $\frac{1}{\alpha}+\frac{1}{\beta}=1$, $\mu^{\min}>0$, $0<\sigma^{\min}<\sigma^{\max}<\infty$, $\tau>1$, $c>0$, $0<p_{\min}<1$, $\rho$ as in \eqref{eq:definerho}, and $\gamma\geq0$ satisfying \eqref{eq:definegamma}. Set $\mathcal{R}_0=\mathcal{F}_{\lambda}(X^0,\,Y^0)$ and $k=0$. \\[3pt]
\textbf{while} a termination criterion is not met, \textbf{do} \vspace{-2mm}
\begin{itemize}[leftmargin=2cm]
\item[\textbf{Step 1}.] Compute $Z^k$ by \vspace{-2mm}
     \begin{equation}\label{Zkupdate}
     Z^k = {\textstyle\left(\mathcal{I} - \frac{\beta}{\alpha+\beta}\mathcal{A}^*\mathcal{A}\right)}(X^k(Y^k)^{\top})
     + {\textstyle\frac{\beta}{\alpha+\beta}}\mathcal{A}^*(\bm{b}).
     \end{equation}

\vspace{-4mm}
\item[\textbf{Step 2}.] Choose $\mu_k^0\geq\mu^{\min}$ and $\sigma_k^0 \in [\sigma^{\min},\,\sigma^{\max}]$ arbitrarily. Set $\mu_k=\mu_k^0$, $\sigma_k=\sigma_k^0$, and $\mu_k^{\max}=(\alpha+2\gamma\rho)\|Y^k\|^2+c$. \vspace{-1mm}
    \begin{enumerate}[(2a)]
    \item Set $\mu_k\leftarrow\min\{\mu_k, \,\mu_k^{\max}\}$. Compute $U$ by either \eqref{supro:U-prox}, \eqref{supro:U-proxlinear} or \eqref{supro:U-hier}.\label{algo-AcomputeU}

    \item Compute $V$ by either \eqref{supro:V-prox}, \eqref{supro:V-proxlinear} or \eqref{supro:V-hier}.\label{algo-AcomputeV}

    \item If
    \begin{equation}\label{eq:A-linesearch}
    \mathcal{F}_\lambda(U, \,V) - \mathcal{R}_k
    \leq -\frac{c}{2}\left(\|U - X^k\|_F^2 + \|V - Y^k\|_F^2\right),
    \end{equation}
    then go to \textbf{Step 3}.\label{algo-Alinesearch}

\item If $\mu_k = \mu_k^{\max}$, set  \vspace{-1mm}
    \begin{equation*}
    \sigma_k^{\max} = (\alpha + 2\gamma\rho) \|U\|^2 + c, ~~
    \sigma_k \leftarrow \min\{\tau\sigma_k, \,\sigma_k^{\max}\},  \vspace{-1mm}
    \end{equation*}
    and go to step (2b); otherwise, set $\mu_k\leftarrow\tau\mu_k$, $\sigma_k \leftarrow \tau\sigma_k$, and go to step (2a).\label{algo-Abacktrack}
      
\end{enumerate}

\vspace{-2mm}
\item[\textbf{Step 3}.] Set $X^{k+1}\leftarrow U$, $Y^{k+1} \leftarrow V$, $\bar{\mu}_k \leftarrow \mu_k$, $\bar{\sigma}_k \leftarrow \sigma_k$. Choose $p_{k+1}\in[p_{\min},\,1]$ to update  \vspace{-1mm}
    \begin{equation}\label{eq:omegadefinition}
    \mathcal{R}_{k+1} \leftarrow (1 - p_{k+1}) \mathcal{R}_{k}
    + p_{k+1}\mathcal{F}_{\lambda}(X^{k+1},\,Y^{k+1}).  \vspace{-1mm}
    \end{equation}
    Then, set $k \leftarrow k+1$ and go to \textbf{Step 1}.
    
\vspace{-2mm}
\end{itemize}
\textbf{end while}  \\
\textbf{Output}: $(X^k, \,Y^k)$. \vspace{0.5mm}
\end{algorithm}


As shown in Algorithm~\ref{algo-A-NAUM}, at the $k$th iteration we first compute $Z^k$ in \textbf{Step~1} using the explicit formula \eqref{Zkupdate}, which is derived from the stationarity condition \eqref{eq:Zbar=} of the potential function $\Theta_{\alpha,\beta,\lambda}$ with respect to $Z$. Next, in \textbf{Step~2}, we compute $U$ and $V$ as candidates for $X^{k+1}$ and $Y^{k+1}$, respectively, by (approximately) minimizing $\Theta_{\alpha,\beta,\lambda}$ with respect to $X$ and $Y$ in a Gauss--Seidel manner, that is, by updating one block while keeping the others fixed and up to date. The detailed updating rules are summarized as follows. For notational simplicity, we define
\begin{equation*}
\mathcal{H}_{\alpha}(X,\,Y,\,Z) := \frac{\alpha}{2}\|XY^{\top} - Z\|_F^2, \quad \forall\; (X,\,Y,\,Z)\in \mathbb{R}^{n\times r}\times \mathbb{R}^{n\times r} \times \mathbb{R}^{n\times n}.
\end{equation*}

\subsection{Candidate updates for $X$ and $Y$.}

Given $Z^k$ from \textbf{Step~1}, A-NAUM alternately computes $U$ and $V$ as candidates for $X^{k+1}$ and $Y^{k+1}$, respectively, using one of the following schemes with proximal parameters $\mu_k,\sigma_k>0$.

\paragraph{Proximal update.} Fixing $Y^k$ and $Z^k$, we can compute $U$ by minimizing $\Theta_{\alpha,\beta,\lambda}$ with respect to $X$, with an added proximal term, as follows:
\begin{align}
U &\in \arg\min_{X}\;
\Psi(X)+\mathcal{H}_{\alpha}(X,\,Y^k,\,Z^k)
+\frac{\lambda}{2}\|X-Y^k\|_F^2+\frac{\mu_k}{2}\|X-X^k\|_F^2, \label{supro:U-prox} \tag{U1a}\\
\Leftrightarrow \;\;
0 &\in \partial \Psi(U) + \alpha ( U(Y^k)^\top - Z^k ) Y^k + \lambda (U - Y^k) + \mu_k (U - X^k). \label{supro:Uopt-prox} \tag{U1b}
\end{align}
Similarly, we can compute $V$ with fixed $U$ and $Z^k$ as follows:
\begin{align}
V &\in \arg\min_{Y}\;
\Phi(Y)+\mathcal{H}_{\alpha}(U,\,Y,\,Z^k)
+\frac{\lambda}{2}\|U-Y\|_F^2+\frac{\sigma_k}{2}\|Y-Y^k\|_F^2, \label{supro:V-prox} \tag{V1a}\\
\Leftrightarrow \;\;
0 &\in \partial \Phi(V) + \alpha( UV^\top - Z^k)^\top U - \lambda (U-V) + \sigma_k (V - Y^k). \label{supro:Vopt-prox} \tag{V1b}
\end{align}

\paragraph{Prox-linear update.} We can linearize the smooth function $\mathcal{H}_\alpha$ and compute $U$ as follows:
\begin{align}
&U \in \arg\min\limits_{X}~ \Psi(X) + \langle \nabla_X \mathcal{H}_{\alpha}(X^k,Y^k,Z^k), \,X - X^k\rangle + \frac{\lambda}{2}\|X - Y^k\|_F^2+ \frac{\mu_k}{2}\|X - X^k\|_F^2, \label{supro:U-proxlinear} \tag{U2a} \\
&\Leftrightarrow \;\;
0 \in \partial \Psi(U) + \alpha ( X^k (Y^k)^\top - Z^k ) Y^k + \lambda (U - Y^k) + \mu_k (U - X^k). \label{supro:Uopt-proxlinear} \tag{U2b}
\end{align}
Similarly, we can compute $V$ as follows:
\begin{align}
V &\in \arg\min_{Y}\;
\Phi(Y)+\big\langle \nabla_Y\mathcal{H}_{\alpha}(U,\,Y^k,\,Z^k),\,Y-Y^k\big\rangle
+\frac{\lambda}{2}\|U-Y\|_F^2+\frac{\sigma_k}{2}\|Y-Y^k\|_F^2, \label{supro:V-proxlinear} \tag{V2a}\\
\Leftrightarrow\;\; 
0 &\in \partial \Phi(V) + \alpha ( U(Y^k)^\top - Z^k )^\top U - \lambda (U-V)+\sigma_k (V - Y^k). \label{supro:Vopt-proxlinear} \tag{V2b}
\end{align}

\paragraph{Hierarchical-prox update.} If $\Psi$ is column-wise separable, i.e., $\Psi(X)=\sum_{i=1}^r\psi_i(\bm{x}_i)$ for $X=[\bm{x}_1,\ldots,\bm{x}_r]\in\mathbb{R}^{n\times r}$, we can update $U$ column-by-column. For $i=1,\ldots,r$, compute
\begin{align}
&\bm{u}_i \in \arg\min_{\bm{x}_i}\;
\psi_i(\bm{x}_i)
+\mathcal{H}_{\alpha}\big(\bm{u}_{j<i},\,\bm{x}_i,\,\bm{x}_{j>i}^k,\,Y^k,\,Z^k\big) +\frac{\lambda}{2}\|\bm{x}_i-\bm{y}_i^k\|^2+\frac{\mu_k}{2}\|\bm{x}_i-\bm{x}_i^k\|^2, \label{supro:U-hier} \tag{U3a} \\
&\hspace{-2mm}\Leftrightarrow \;0 \in \partial \psi_i(\bm{u}_i) + \alpha\!\left( \sum_{j=1}^i \bm{u}_j (\bm{y}_j^k)^\top \!+\! \sum_{j=i+1}^r \bm{x}_j^k (\bm{y}_j^k)^\top \!-\! Z^k \right)\!\bm{y}_i^k + \lambda (\bm{u}_i - \bm{y}_i^k) + \mu_k (\bm{u}_i - \bm{x}_i^k), \label{supro:Uopt-hier} \tag{U3b}
\end{align}
where $\bm{u}_{j<i}:=(\bm{u}_1,\ldots,\bm{u}_{i-1})$ and $\bm{x}_{j>i}^k:=(\bm{x}_{i+1}^k,\ldots,\bm{x}_r^k)$. Similarly, if $\Phi$ is column-wise separable, i.e., $\Phi(Y)=\sum_{i=1}^r\phi_i(\bm{y}_i)$ for $Y=[\bm{y}_1,\ldots,\bm{y}_r]\in\mathbb{R}^{n\times r}$, we can update $V$ column-by-column. For $i=1,\ldots,r$, compute
\begin{align}
&\bm{v}_i \in \arg\min_{\bm{y}_i}\;
\phi_i(\bm{y}_i)
+\mathcal{H}_{\alpha}\big(U,\,\bm{v}_{j<i},\,\bm{y}_i,\,\bm{y}_{j>i}^k,\,Z^k\big)
+\frac{\lambda}{2}\|\bm{u}_i-\bm{y}_i\|^2+\frac{\sigma_k}{2}\|\bm{y}_i-\bm{y}_i^k\|^2, \label{supro:V-hier} \tag{V3a} \\
\Leftrightarrow \; \; &0 \in \partial \phi_i(\bm{v}_i) + \alpha \left( \sum_{j=1}^i \bm{u}_j \bm{v}_j^\top + \sum_{j=i+1}^r \bm{u}_j (\bm{y}_j^k)^\top - Z^k\right)^\top \bm{u}_i - \lambda (\bm{u}_i - \bm{v}_i)+ \sigma_k (\bm{v}_i - \bm{y}_i^k), \label{supro:Vopt-hier} \tag{V3b}
\end{align}
where $\bm{v}_{j<i}:=(\bm{v}_1,\ldots,\bm{v}_{i-1})$ and
$\bm{y}_{j>i}^k:=(\bm{y}_{i+1}^k,\ldots,\bm{y}_r^k)$.

After computing $(U,V)$, we set $(X^{k+1},Y^{k+1})=(U,V)$ if the average-type nonmonotone line search criterion \eqref{eq:A-linesearch} is satisfied, i.e., if $\mathcal{F}_\lambda(U,V)$ is sufficiently smaller than the reference value $\mathcal{R}_k$. Otherwise, the proximal parameters $\mu_k$ and $\sigma_k$ are increased and the above procedure is repeated. Later, we will show that the criterion \eqref{eq:A-linesearch} is well-defined under mild conditions, ensuring that the A-NAUM in Algorithm \ref{algo-A-NAUM} is well-defined.

\subsection{Comparisons with existing alternating frameworks}

The iterative framework of A-NAUM is motivated by the max-type nonmonotone alternating updating method (M-NAUM) proposed by Yang et al. \cite{ypc2018nonmonotone} for solving a class of matrix factorization problems, where an exact relaxation technique is combined with a max-type nonmonotone line search to achieve strong numerical performance. Nevertheless, the proposed A-NAUM differs from M-NAUM in two key aspects. First, A-NAUM is designed for the more general formulation \eqref{quadpenaltypro} and allows for a positive penalty parameter $\lambda>0$, whereas M-NAUM focuses on the case $\lambda=0$. Second, and more importantly, these two methods employ fundamentally different nonmonotone mechanisms. 

Specifically, for the case $\lambda=0$, M-NAUM \cite{ypc2018nonmonotone} adopts the nonmonotone line search framework of Grippo, Lampariello, and Lucidi \cite{gll1986nonmonotone}, where the reference value is defined as
\begin{equation*}
\mathcal{R}_k:=\max\left\{\mathcal{F}_\lambda(X^t, \,Y^t):
\,t=k,\,k-1,\ldots,[k-N]_{+}\right\}
\end{equation*}
for some fixed $N\in\mathbb{N}$. This \textit{max-type} strategy permits temporary increases of the objective value, provided it remains below the maximum over the most recent $N+1$ iterations. Such flexibility often allows smaller proximal parameters $\mu_k$ and $\sigma_k$, which can be beneficial in practice. 

In contrast, inspired by Zhang and Hager \cite{zh2004nonmonotone}, A-NAUM adopts an \emph{average-type} nonmonotone strategy, where the reference value $\mathcal{R}_k$ is defined as a convex combination of the previous reference value $\mathcal{R}_{k-1}$ and the current function value $\mathcal{F}_\lambda(X^k, \,Y^k)$ with a weight parameter $p_k$. This mechanism could yield a smoother and more stable nonmonotone behavior.

Both max-type and average-type nonmonotone line search strategies have been widely used in proximal-gradient-type methods to enhance numerical performance; see, e.g., \cite{clp2016penalty,gzlhy2013general,wnf2009sparse,y2024proximal,ypc2018nonmonotone}. However, max-type strategies are typically associated with weaker convergence guarantees, which motivates our adoption of the average-type strategy in A-NAUM. As will be shown later, compared with M-NAUM, which only has global subsequential convergence, the proposed A-NAUM admits substantially stronger convergence properties (including the convergence of the whole sequence and the convergence rate) under weaker assumptions, while maintaining comparable or even better numerical performance, as demonstrated in Section \ref{sec-numer}.


Our A-NAUM also differs fundamentally from classical alternating minimization schemes (see, e.g., \cite{abs2013convergence,xy2013block}) in its treatment of the auxiliary variable $Z$ and, consequently, in the descent behavior it enforces. Under the stronger assumptions of Theorem~\ref{thm:equal-min}, the update of $Z$ in \eqref{Zkupdate} corresponds to an exact minimization of $\Theta_{\alpha,\beta,\lambda}$ with respect to $Z$. In this regime, A-NAUM with $p_{k}\equiv1$ can be viewed as an alternating-minimization-type method applied to the decoupled problem $\min_{X,Y,Z}\left\{\Theta_{\alpha,\beta,\lambda}(X,\,Y,\,Z)\right\}$. In contrast, when A-NAUM is applied under the weaker conditions $\mathcal{A}\mathcal{A}^*=\mathcal{I}_q$ and $\frac{1}{\alpha}+\frac{1}{\beta}=1$, the update of $Z$ in \eqref{Zkupdate} is only required to satisfy the stationarity condition and does not necessarily yield a monotone decrease of $\Theta_{\alpha,\beta,\lambda}$. This deliberate nonmonotonicity distinguishes A-NAUM from classical alternating minimization frameworks and substantially enlarges the admissible parameter regime. In particular, it allows for negative values of $\alpha$ or $\beta$, which are excluded in standard alternating minimization but are shown in Section~\ref{sec-numer} to provide improved numerical performance.

\subsection{A unified and flexible algorithmic framework}

The A-NAUM in Algorithm~\ref{algo-A-NAUM} is deliberately designed as a \textit{unified and flexible} algorithmic framework. It integrates the exact relaxation technique with an average-type nonmonotone line search, while accommodating multiple block-update strategies that can be selectively employed to exploit problem structures and to build accelerated variants. Owing to the generality of \eqref{quadpenaltypro}, A-NAUM applies to a wide range of RMF/SRMF models, serving as a unified algorithmic backbone. Moreover, A-NAUM is highly modular: both $X$ and $Y$ updates admit three alternatives, yielding nine valid combinations and allowing practitioners to tailor the algorithm to problem structures and computational constraints. 

For example, when $\Psi$ or $\Phi$ is column-wise separable, one can leverage the decomposition
\begin{equation*}
XY^\top={\textstyle\sum_{i=1}^r}\bm{x}_i\bm{y}_i^\top,
\qquad
X=[\bm{x}_1,\ldots,\bm{x}_r]\in\mathbb{R}^{m\times r},\;\;
Y=[\bm{y}_1,\ldots,\bm{y}_r]\in\mathbb{R}^{n\times r},
\end{equation*}
together with the structure of $\Theta_{\alpha,\beta,\lambda}$ to implement the hierarchical-prox updates in \eqref{supro:U-hier} or \eqref{supro:V-hier}, that is, column-wise updates of the factor matrix even when $\mathcal{A}\neq\mathcal{I}$. The principal computational advantage is that, after simple reformulations, each column subproblem reduces to \emph{evaluating the proximal mapping of $\psi_i$ (or $\phi_i$)}, which admits a closed-form solution or can be computed efficiently for many commonly used regularizers. This leads to inexpensive inner updates and improves scalability when $r$ is moderate and the proximal operators are cheap.

\section{Convergence analysis}\label{sec-convanal}

In this section, we investigate the convergence behavior of A-NAUM. We first establish the
global subsequential convergence of A-NAUM. Then, by assuming that the Kurdyka-{\L}ojasiewicz property holds for $\mathcal{F}_{\lambda}$, we prove the global convergence of the whole sequence generated by A-NAUM and analyze the rate of convergence. Before presenting these results, we impose the following standing assumptions, which will be used throughout the convergence analysis.

\begin{assumption}\label{assumA}
{\rm (i)} $\Psi$ (or $\psi_i$ for all $i$ if $\Psi(X) = \sum_{i=1}^r \psi_i(\bm{x}_i)$) and $\Phi$ (or $\phi_i$ for all $i$ if $\Phi(Y) = \sum_{i=1}^r \phi_i(\bm{y}_i)$) are proper closed and bounded from below; {\rm (ii)} $\mathcal{A}\mathcal{A}^*=\mathcal{I}_q$; {\rm (iii)} $\frac{1}{\alpha} + \frac{1}{\beta} = 1$; {\rm (iv)} When $\lambda=0$, $\mathcal{F}_{0}$ is level-bounded.
\end{assumption}

Under Assumption \ref{assumA}(i), the objective functions of the subproblems \eqref{supro:U-prox}, \eqref{supro:U-proxlinear}, \eqref{supro:U-hier}, \eqref{supro:V-prox}, \eqref{supro:V-proxlinear}, and \eqref{supro:V-hier} are level-bounded and therefore admit minimizers, although such minimizers may not be unique due to the possible nonconvexity. Moreover, unlike M-NAUM \cite{ypc2018nonmonotone}, we do not require $\Psi$ and $\Phi$ to be continuous on their respective domains. Assumption \ref{assumA}(ii) is satisfied, for example, when $\mathcal{A}$ is the vectorization map from $\mathbb{R}^{n\times n}$ to $\mathbb{R}^{n^2}$, or the sampling map $\mathcal{P}_{\Omega}$ from $\mathbb{R}^{n\times n}$ to $\mathbb{R}^{q}$.  Finally, under Assumption \ref{assumA}(i), it is straightforward to verify that $\mathcal{F}_{\lambda}$ is level-bounded for any positive $\lambda>0$. This, together with Assumption \ref{assumA}(iv), ensures that $\mathcal{F}_{\lambda}$ is level-bounded for any nonnegative $\lambda\geq0$, a property that will be used to establish the boundedness of the iterates in the subsequent analysis.

\subsection{Convergence analysis I: global subsequential convergence}

We first establish a sufficient descent property, which forms a key technical ingredient in our subsequent analysis. The proof is analogous to that of \cite[Lemma~5.1]{ypc2018nonmonotone}, and is provided in Appendix~\ref{proof:lem:Flambda-suffidescent} for completeness.

\begin{lemma}[\textbf{Sufficient descent of $\mathcal{F}_\lambda$}]\label{lem:Flambda-suffidescent}
Suppose that Assumption \ref{assumA} holds. At the $k$-th iteration, given $(X^k, \,Y^k)$, let $(U, \,V)$ be the candidate for $(X^{k+1}, \,Y^{k+1})$ generated by Steps (2\ref{algo-AcomputeU}) and (2\ref{algo-AcomputeV}). Then, for each $k\geq 0$, we have
\begin{equation}\label{eq:Flambdadescent}
\begin{aligned}
&\quad \mathcal{F}_\lambda(U, \,V) - \mathcal{F}_\lambda(X^k, \,Y^k) \\
&\leq -\frac{\mu_k-(\alpha+2\gamma\rho)\|Y^k\|^2}{2}\|U - X^k\|_F^2 
- \frac{\sigma_k-(\alpha+2\gamma\rho)\|U\|^2}{2}\|V - Y^k\|_F^2.
\end{aligned}
\end{equation}
\end{lemma}

From Lemma \ref{lem:Flambda-suffidescent}, we see that a sufficient descent of $\mathcal{F}_\lambda(X,\,Y)$ can be guaranteed provided that $\mu_k$ and $\sigma_k$ are sufficiently large. Building on this result, we show that the average-type nonmonotone line search criterion \eqref{eq:A-linesearch} is well-defined.

\begin{lemma}[\textbf{Well-definedness of criterion \eqref{eq:A-linesearch}}]\label{lem:A-welldefined}
Suppose that Assumption \ref{assumA} holds. Then, for each $k\geq 0$, the line search criterion \eqref{eq:A-linesearch} is satisfied after finitely many inner iterations.
\end{lemma}
\begin{proof}
See Appendix \ref{proof:lem:A-welldefined}.
\end{proof}

Based on the above results, we further establish the following properties.


\begin{proposition}\label{prop:Omegak-descentbehavior}
Suppose that Assumption \ref{assumA} holds. Let $\{(X^k,\,Y^k)\}$, $\{\bar{\mu}_{k}\}$ and $\{\bar{\sigma}_{k}\}$ be sequences generated by the A-NAUM in Algorithm \ref{algo-A-NAUM}. Then, the following statements hold.
\begin{itemize}
\item [{\rm(i)}] $\mathcal{R}_k\geq\mathcal{F}_\lambda(X^k, \,Y^k)$ for all $k\geq 0$;

\item [{\rm(ii)}] The sequence $\{\mathcal{R}_k\}$ is non-increasing and $\zeta:= \lim\limits_{k\to\infty}\mathcal{R}_k$ exists;

\item [{\rm(iii)}] The sequence $\{\mathcal{F}_\lambda(X^k,\,Y^k)\}$ converges to the same limit $\zeta$ as $\{\mathcal{R}_k\}$;

\item[{\rm(iv)}] $\lim\limits_{k \to \infty} \|X^{k+1} - X^{k}\|_{F}= 0$ and $\lim\limits_{k \to \infty} \|Y^{k+1} - Y^{k}\|_{F}= 0$;
    
\item[{\rm(v)}]     
    $\{(X^{k}, \,Y^{k})\}$, $\{\bar{\mu}_{k}\}$, and $\{\bar{\sigma}_{k}\}$ are bounded.

\end{itemize}
\end{proposition}
\begin{proof}
\textit{Statement (i)}. The desired result follows from \eqref{eq:omega>flambda} and the discussions that follow.

\textit{Statement (ii)}. By Lemma \ref{lem:A-welldefined}, the sequence $\{(X^k,\,Y^k)\}$ is well-defined. Moreover, from the line search criterion \eqref{eq:A-linesearch}, we have
\begin{equation*}
\mathcal{F}_{\lambda}(X^{k+1},\,Y^{k+1})-\mathcal{R}_{k}
\leq -\frac{c}{2}\big(\|X^{k+1}-X^{k}\|_F^2 + \|Y^{k+1}-Y^{k}\|_F^2\big), \quad \forall\,k\geq 0.
\end{equation*}
Combining this inequality with the updating rule of $\mathcal{R}_k$ in \eqref{eq:omegadefinition} and $p_{k+1}\in[p_{\min},1]$, we obtain
\begin{equation}\label{omegaupdate}
\begin{aligned}
\mathcal{R}_{k+1}
&=(1-p_{k+1})\mathcal{R}_{k} + p_{k+1}\mathcal{F}_{\lambda}(X^{k+1},\,Y^{k+1}) \\
&\leq(1-p_{k+1})\mathcal{R}_{k} + p_{k+1}\left(\mathcal{R}_{k}
-\frac{c}{2}\big(\|X^{k+1}-X^{k}\|_F^2+\|Y^{k+1}-Y^{k}\|_F^2\big)\right) \\
&\leq\mathcal{R}_{k}-\frac{cp_{\min}}{2}\big(\|X^{k+1}-X^{k}\|_F^2 
+ \|Y^{k+1}-Y^{k}\|_F^2\big),
\end{aligned}
\end{equation}
for all $k\geq0$. Hence, $\{\mathcal{R}_k\}$ is a non-increasing. Moreover, it follows from Assumption \ref{assumA}(i) and $\{(X^k,\,Y^k)\}\subseteq {\rm dom}\,\mathcal{F}_\lambda$ that $\{\mathcal{F}_\lambda(X^k,\,Y^k)\}$ is bounded from below. Thus, $\{\mathcal{R}_k\}$ is also bounded from below by statement (i). Therefore, $\{\mathcal{R}_k\}$ is convergent and the limit $\zeta:= \lim\limits_{k\to\infty}\mathcal{R}_k$ exists.

\textit{Statement (iii)}. Using the updating rule of $\{\mathcal{R}_k\}$ in \eqref{eq:omegadefinition} again, we see that
\begin{equation*}
\mathcal{F}_\lambda(X^{k+1},\,Y^{k+1}) = 
\mathcal{R}_{k} + {\textstyle\frac{1}{p_{k+1}}}(\mathcal{R}_{k+1}-\mathcal{R}_{k}), 
\quad \forall\,k\geq 0.
\end{equation*}
Then, by taking the limit as $k\to\infty$ and using statement (ii) together with $0<p_{\min}\leq p_{k+1}\leq1$, we conclude that $\{\mathcal{F}_\lambda(X^k,\,Y^k)\}$ converges to the same limit $\zeta$ as $\{\mathcal{R}_k\}$.

\textit{Statement (iv)}.
It follows from \eqref{omegaupdate} that
\begin{equation*}
\|X^{k+1}-X^{k}\|_{F}^2 + \|Y^{k+1}-Y^{k}\|_{F}^2
\leq {\textstyle\frac{2}{cp_{\min}}}(\mathcal{R}_k-\mathcal{R}_{k+1}), 
\quad \forall\,k\geq 0.
\end{equation*}
Since $\{\mathcal{R}_k\}$ converges, it follows that
\begin{equation*}
\lim\limits_{k\to\infty}\|X^{k+1}-X^{k}\|_{F}=0 
\quad \mathrm{and} \quad
\lim\limits_{k\to\infty}\|Y^{k+1}-Y^{k}\|_{F}=0.
\end{equation*}

\textit{Statement (iv)}. 
From statements (i) and (ii), together with $\mathcal{R}_0=\mathcal{F}_{\lambda}(X^0,\,Y^0)$, we obtain
\begin{equation}\label{eq:sublevelset}
\mathcal{F}_\lambda(X^k, \,Y^k) \leq \mathcal{R}_k 
\leq \mathcal{R}_0 = \mathcal{F}_\lambda(X^0, Y^0), 
\quad \forall\,k\geq 0.
\end{equation}
Since Assumption \ref{assumA}(iv) ensures that $\mathcal{F}_{\lambda}$ is level-bounded for any nonnegative $\lambda\geq0$, the sequence $\{(X^{k}, \,Y^{k})\}$ is bounded. Moreover, from Step 2 of Algorithm \ref{algo-A-NAUM}, we know that $\bar{\mu}_k \leq \mu_k^{\max} = (\alpha + 2\gamma\rho)\|Y^k\|^2 + c$ for all $k\geq 0$. Since $\{Y^k\}$ is bounded, both $\{\mu_k^{\max}\}$ and $\{\bar{\mu}_k\}$ are also bounded. Finally, we show that $\{\bar{\sigma}_k\}$ is also bounded. Indeed, at the $k$th iteration, three cases may occur:
\begin{itemize}
\item $\bar{\mu}_k < \mu_k^{\max}$: In this case, we have $\bar{\sigma}_k \leq \sigma_k^0 \tau^{\bar{n}_k} \leq \sigma^{\max} \tau^{\bar{n}_k}$, where $\bar{n}_k$ denotes the number of inner iterations for the line search at the $k$th iteration and $\bar{n}_k \leq \max\left\{1, \left\lfloor \frac{\log(\mu_{k}^{\max}) - \log(\mu^{\min})}{\log\tau} + 2 \right\rfloor\right\}$ (see \eqref{eq:etak-bound} and the discussions preceding it).

\item $\bar{\mu}_k = \mu_k^{\max}$ and $\bar{\sigma}_k > \sigma_k^{\max}$: In this case, we have $\bar{\sigma}_k \leq \sigma_k^0 \tau^{\bar{n}_k} \leq \sigma^{\max} \tau^{\bar{n}_k}$, where $\bar{n}_k \leq \max\left\{1, \left\lfloor\frac{\log(\mu_{k}^{\max}) - \log(\mu^{\min})}{\log\tau} + 2 \right\rfloor\right\}$.

\item Otherwise, we have $\bar{\sigma}_k \leq \sigma_k^{\max} = (\alpha + 2\gamma\rho)\|X^{k+1}\|^2 + c$.
\end{itemize}
Note that $\{\bar{n}_k\}$ is bounded as $\{\mu_k^{\max}\}$ is bounded. Thus, $\{\bar{\sigma}_k\}$ is bounded as the sequences $\{X^k\}$ and $\{\bar{n}_k\}$ are bounded. This completes the proof.
\end{proof}

We are now ready to establish the global subsquential convergence for A-NAUM.

\begin{theorem}\label{thm:A-convergence}
Suppose that Assumption \ref{assumA} holds. Let $\{(X^{k},\,Y^{k})\}$ be the sequence generated by the A-NAUM in Algorithm \ref{algo-A-NAUM}, let $\Gamma$ denote the set of all cluster points of $\{(X^{k},\,Y^{k})\}$, let $\zeta$ be the limit value defined in Proposition \ref{prop:Omegak-descentbehavior}(ii), and let $\mathcal{S}$ denote the set of all stationary points of $\mathcal{F}_{\lambda}$. Then, $\Gamma\subseteq\mathcal{S}$ and $\mathcal{F}_\lambda \equiv \zeta$ on $\Gamma$.
\end{theorem}
\begin{proof}
We first claim that $\Gamma\subseteq {\rm dom}\,\mathcal{F}_\lambda$. By Proposition \ref{prop:Omegak-descentbehavior}(v), the sequence $\{(X^k,\,Y^k)\}$ is bounded and therefore admits at least one cluster point. Thus, $\Gamma$ is nonempty and compact. Take any $(\widetilde{X},\,\widetilde{Y})\in\Gamma$ and let $\{(X^{k_j},\,Y^{k_j})\}_{j\in\mathbb{N}}$ be a subsequence converging to $(\widetilde{X},\,\widetilde{Y})$. Since $\mathcal{F}_\lambda$ is lower semi-continuous (as $\Psi$ and $\Phi$ are lower semi-continuous by Assumption \ref{assumA}(i)), it follows from \eqref{eq:sublevelset} that
\begin{align*}
\mathcal{F}_\lambda(\widetilde{X}, \,\widetilde{Y})
\leq \liminf\limits_{j\to\infty}\,\mathcal{F}_\lambda(X^{k_j}, \,Y^{k_j})
\leq \lim_{j\to\infty}\,\mathcal{F}_\lambda(X^{k_j}, \,Y^{k_j})
\leq \mathcal{F}_\lambda(X^0, \,Y^0),
\end{align*}
which implies that $\mathcal{F}_\lambda(\widetilde{X},\,\widetilde{Y})<\infty$, and hence $(\widetilde{X},\,\widetilde{Y})\in{\rm dom}\,\mathcal{F}_\lambda$. From the updating rule of $Z^k$ in \eqref{Zkupdate}, we further have 
\begin{equation*}
\textstyle
Z^{k_j}\,\to\,\widetilde{Z}:=\left(\mathcal{I} - \frac{\beta}{\alpha+\beta}\mathcal{A}^*\mathcal{A}\right)(\widetilde{X}\widetilde{Y}^{\top}) + \frac{\beta}{\alpha+\beta}\mathcal{A}^*(\bm{b}), \quad \text{as} ~~j \to \infty.
\end{equation*}

We next show that $(\widetilde{X},\,\widetilde{Y},\,\widetilde{Z})$ is a stationary point of $\Theta_{\alpha,\beta,\lambda}$, i.e., $0\in\partial\Theta_{\alpha,\beta,\lambda}(\widetilde{X},\,\widetilde{Y},\,\widetilde{Z})$. By the definition of $\widetilde{Z}$ and $\mathcal{A}\mathcal{A}^*=\mathcal{I}_q$, we can obtain
\begin{equation}\label{eq:1thetaoptc}
\alpha (\widetilde{Z} - \widetilde{X}\widetilde{Y}^\top)
+ \beta \mathcal{A}^*(\mathcal{A}(\widetilde{Z}) - \bm{b}) = 0.
\end{equation}
It therefore suffices to verify that
\begin{numcases}{}
0\in \partial\Psi(\widetilde{X}) + \alpha(\widetilde{X}\widetilde{Y}^\top - \widetilde{Z})\widetilde{Y}+\lambda(\widetilde{X}-\widetilde{Y}),  \label{eq:1thetaopta} \\[3pt]
0 \in  \partial\Phi(\widetilde{Y}) + \alpha(\widetilde{X}\widetilde{Y}^\top - \widetilde{Z})^\top\widetilde{X}-\lambda(\widetilde{X}-\widetilde{Y}).  \label{eq:1thetaoptb}
\end{numcases}
We first show that
\begin{numcases}{}
\lim\limits_{j\to\infty}~\Psi(X^{k_j+1})=\Psi(\widetilde{X}),  \label{psi_xkj} \\
\lim\limits_{j\to\infty}~\Phi(Y^{k_j+1})=\Phi(\widetilde{Y}).  \label{phi_ykj}
\end{numcases}
Since $\|X^{k_{j}+1}-X^{k_{j}}\|_F\to 0$ by Proposition \ref{prop:Omegak-descentbehavior}(iv), the sequence $\{X^{k_{j}+1}\}$ also converges to $\widetilde{X}$. We now prove \eqref{psi_xkj} for the three updating schemes:
\begin{itemize}[leftmargin=0.5cm]
\item \textbf{Proximal:} Since $X^{k_{j}+1}$ solves \eqref{supro:U-prox} with $k=k_{j}$ and $\mu_k=\bar{\mu}_{k_{j}}$, we have
    \begin{align*}
    &\quad \Psi(X^{k_{j}+1}) 
    + \mathcal{H}_{\alpha}(X^{k_{j}+1},\,Y^{k_{j}},\,Z^{k_{j}})
    + {\textstyle\frac{\lambda}{2}}\|X^{k_{j}+1}-Y^{k_{j}}\|_F^2 
    + {\textstyle\frac{\bar{\mu}_{k_j}}{2}}\|X^{k_{j}+1}-X^{k_{j}}\|_F^2 \\
    &\leq \Psi(\widetilde{X}) 
    + \mathcal{H}_{\alpha}(\widetilde{X},\,Y^{k_{j}},\,Z^{k_{j}}) 
    + {\textstyle\frac{\lambda}{2}}\|\widetilde{X}-Y^{k_{j}}\|_F^2 
    + {\textstyle\frac{\bar{\mu}_{k_{j}}}{2}}\|\widetilde{X}-X^{k_{j}}\|_F^2,
    \quad \forall\,j\in\mathbb{N},
    \end{align*}
    which implies that
    \begin{align*}
    \Psi(X^{k_{j}+1})
    &\leq\Psi(\widetilde{X}) 
    + \mathcal{H}_{\alpha}(\widetilde{X},\,Y^{k_{j}},\,Z^{k_{j}}) 
    - \mathcal{H}_{\alpha}(X^{k_{j}+1},\,Y^{k_{j}},\,Z^{k_{j}})
    + {\textstyle\frac{\lambda}{2}}\|\widetilde{X}-Y^{k_{j}}\|_F^2  \\
    &\quad - {\textstyle\frac{\lambda}{2}}\|X^{k_{j}+1}-Y^{k_{j}}\|_F^2 
    + {\textstyle\frac{\bar{\mu}_{k_j}}{2}}\|\widetilde{X}-X^{k_{j}}\|_F^2
    - {\textstyle\frac{\bar{\mu}_{k_j}}{2}}\|X^{k_{j}+1}-X^{k_{j}}\|_F^2,
    \quad \forall\,j\in\mathbb{N}.
    \end{align*}
    Passing to the limit in the above relation, and invoking $\|X^{k_{j}+1} - X^{k_{j}}\|_F\to 0$, $X^{k_{j}}\to \widetilde{X}$, $X^{k_{j}+1}\to \widetilde{X}$, $Y^{k_{j}}\to \widetilde{Y}$, $Z^{k_{j}}\to \widetilde{Z}$, the continuity of $\mathcal{H}_{\alpha}$, and the boundedness of $\{\bar{\mu}_{k_j}\}$ (by Proposition \ref{prop:Omegak-descentbehavior}(v)), we obtain that $\limsup\limits_{j\to\infty}\,\Psi(X^{k_{j}+1})\leq \Psi(\widetilde{X})$. On the other hand, since $\Psi$ is lower semi-continuous (by Assumption \ref{assumA}(i)), we have that $\Psi(\widetilde{X})\leq \liminf\limits_{j\to\infty}\,\Psi(X^{k_{j}+1})$. Therefore, we obtain \eqref{psi_xkj}.

\item \textbf{Proximal-linear:} Since $X^{k_{j}+1}$ solves \eqref{supro:U-proxlinear} with $k=k_{j}$ and $\mu_k=\bar{\mu}_{k_{j}}$, we have
    \begin{align*}
    &\Psi(X^{k_{j}+1}) 
    + \langle\nabla_X\mathcal{H}_{\alpha}(X^{k_{j}}, Y^{k_{j}}, Z^{k_{j}}), X^{k_{j}+1}\!-\!X^{k_{j}}\rangle
    + {\textstyle\frac{\lambda}{2}}\|X^{k_{j}+1}\!-\!Y^{k_{j}}\|_F^2 
    + {\textstyle\frac{\bar{\mu}_{k_j}}{2}}\|X^{k_{j}+1}\!-\!X^{k_{j}}\|_F^2 \\
    &\leq\Psi(\widetilde{X}) 
    + \langle\nabla_X\mathcal{H}_{\alpha}(X^{k_{j}},Y^{k_{j}},Z^{k_{j}}),
    \widetilde{X} - X^{k_{j}}\rangle
    + {\textstyle\frac{\lambda}{2}}\|\widetilde{X}-Y^{k_{j}}\|_F^2 
    + {\textstyle\frac{\bar{\mu}_{k_j}}{2}}\|\widetilde{X}-X^{k_{j}}\|_F^2,
    \quad \forall\,j\in\mathbb{N},
    \end{align*}
    which implies that
    \begin{align*}
    \Psi(X^{k_{j}+1})
    &\leq \Psi(\widetilde{X}) 
    + \langle\nabla_X\mathcal{H}_{\alpha}(X^{k_{j}},Y^{k_{j}},Z^{k_{j}}),
    \,\widetilde{X}-X^{k_{j}+1}\rangle 
    + {\textstyle\frac{\lambda}{2}}\|\widetilde{X}\!-\!Y^{k_{j}}\|_F^2 
    - {\textstyle\frac{\lambda}{2}}\|X^{k_{j}+1}\!-\!Y^{k_{j}}\|_F^2 \\
    &\quad + {\textstyle\frac{\bar{\mu}_{k_j}}{2}}\|\widetilde{X}-X^{k_{j}}\|_F^2
    - {\textstyle\frac{\bar{\mu}_{k_j}}{2}}\|X^{k_{j}+1}-X^{k_{j}}\|_F^2,
    \quad \forall\,j\in\mathbb{N}.
    \end{align*}
    Thus, using similar arguments as above, together with the continuity of $\nabla_X\mathcal{H}_{\alpha}$, gives \eqref{psi_xkj}. 

\item \textbf{Hierarchical-prox:} Since $\bm{x}_i^{k_{j}+1}$ solves the subproblem \eqref{supro:U-hier} for $i=1,2,\ldots,r$ with $k=k_{j}$ and $\mu_k=\bar{\mu}_{k_{j}}$, we have
    \begin{align*}
    &\quad \psi_i(\bm{x}_i^{k_{j}+1}) 
    + \mathcal{H}_{\alpha}(\bm{x}^{k_{j}+1}_{t<i},\bm{x}_i^{k_{j}+1}, \bm{x}_{t>i}^{k_{j}},Y^{k_{j}},Z^{k_{j}}) 
    + {\textstyle\frac{\lambda}{2}}\|\bm{x}_i^{k_{j}+1}-\bm{y}_i^{k_{j}}\|_F^2 
    + {\textstyle\frac{\bar{\mu}_{k_{j}}}{2}}\|\bm{x}_i^{k_{j}+1} - \bm{x}_i^{k_{j}}\|_F^2\\
    &\leq \psi_i(\tilde{\bm{x}}_i) 
    + \mathcal{H}_{\alpha}(\bm{x}^{k_{j}+1}_{t<i},\tilde{\bm{x}}_i, \bm{x}_{t>i}^{k_{j}}, Y^{k_{j}}, Z^{k_{j}}) 
    + {\textstyle\frac{\lambda}{2}}\|\tilde{\bm{x}}_i-\bm{y}_i^{k_{j}}\|_F^2 
    + {\textstyle\frac{\bar{\mu}_{k_{j}}}{2}}\|\tilde{\bm{x}}_i-\bm{x}_i^{k_{j}}\|_F^2,
    \end{align*}
    which implies that
    \begin{align*}
    \psi_i(\bm{x}_i^{k_{j}+1})
    &\leq\psi_i(\tilde{\bm{x}}_i) 
    + \mathcal{H}_{\alpha}(\bm{x}^{k_{j}+1}_{t<i}, \tilde{\bm{x}}_i, \bm{x}_{t>i}^{k_{j}}, Y^{k_{j}}, Z^{k_{j}})
    - \mathcal{H}_{\alpha}(\bm{x}^{k_{j}+1}_{t<i}, \bm{x}_i^{k_{j}+1}, \bm{x}_{t>i}^{k_{j}}, Y^{k_{j}}, Z^{k_{j}}) \\
    &\quad 
    + {\textstyle\frac{\lambda}{2}}\|\tilde{\bm{x}}_i-\bm{y}_i^{k_{j}}\|_F^2
    - {\textstyle\frac{\lambda}{2}}\|\bm{x}_i^{k_{j}+1}-\bm{y}_i^{k_{j}}\|_F^2 
    + {\textstyle\frac{\bar{\mu}_{k_{j}}}{2}}\|\tilde{\bm{x}}_i-\bm{x}_i^{k_{j}}\|_F^2
    - {\textstyle\frac{\bar{\mu}_{k_{j}}}{2}}\|\bm{x}_i^{k_{j}+1} - \bm{x}_i^{k_{j}}\|_F^2.
    \end{align*}
    Similarly, one can verify that $\lim\limits_{j\to\infty}\psi_i(\bm{x}_i^{k_{j}+1})=\psi_i(\tilde{\bm{x}}_i)$ for $i=1,2,\ldots,r$, which implies \eqref{psi_xkj}. 
\end{itemize}
By an entirely analogous argument, we can obtain \eqref{phi_ykj}. We now prove \eqref{eq:1thetaopta} as follows.
\begin{itemize}[leftmargin=0.5cm]
\item \textbf{Proximal\,\&\,prox-linear:} Passing to the limit along $\{(X^{k_{j}},\,Y^{k_{j}})\}$ in \eqref{supro:Uopt-prox} or \eqref{supro:Uopt-proxlinear} with $X^{k_{j}+1}$ in place of $U$ and $\bar{\mu}_{k_{j}}$ in place of $\mu_k$, and invoking $\|X^{k_{j}+1} - X^{k_{j}}\|_F\to 0$ (by Proposition \ref{prop:Omegak-descentbehavior}(iv)), $X^{k_{j}}\to \widetilde{X}$, $X^{k_{j}+1}\to \widetilde{X}$, $Y^{k_{j}}\to \widetilde{Y}$, $Z^{k_{j}}\to \widetilde{Z}$, the boundedness of $\{\bar{\mu}_{k_{j}}\}$ (by Proposition \ref{prop:Omegak-descentbehavior}(v)), \eqref{psi_xkj}, \eqref{robust}, and $(\widetilde{X},\,\widetilde{Y}) \in {\rm dom}\,\mathcal{F}_\lambda$, we obtain \eqref{eq:1thetaopta}.

\item \textbf{Hierarchical-prox:} 
Similarly, we have
\begin{equation*}
0 \in \partial \psi_i(\tilde{\bm{x}}_i) + \alpha (\widetilde{X}\widetilde{Y}^\top - \widetilde{Z}) \tilde{\bm{y}}_i+\lambda(\tilde{\bm{x}}_i-\tilde{\bm{y}}_i),
\quad \forall\,i = 1, 2, \ldots, r,
\end{equation*}
which implies \eqref{eq:1thetaopta} upon stacking.
\end{itemize}
The inclusion \eqref{eq:1thetaoptb} follows similarly. Combining \eqref{eq:1thetaoptc}, \eqref{eq:1thetaopta}, and \eqref{eq:1thetaoptb}, we conclude that $(\widetilde{X},\,\widetilde{Y},\,\widetilde{Z})$ is a stationary point of $\Theta_{\alpha,\beta,\lambda}$. Then, invoking Assumption \ref{assumA}(ii)\&(iii) and the definition of $\widetilde{Z}$, it follows from Theorem \ref{thm:equal-sta} that $(\widetilde{X},\,\widetilde{Y})$ is a stationary point of $\mathcal{F}_\lambda$ (i.e., $(\widetilde{X}, \widetilde{Y})\in\mathcal{S}$).

Finally, using \eqref{psi_xkj}, \eqref{phi_ykj}, and Proposition \ref{prop:Omegak-descentbehavior}(iii)\&(iv), we obtain
\begin{equation*}
\zeta
=\lim\limits_{j\to\infty}\,\mathcal{F}_\lambda(X^{k_{j}+1},\,Y^{k_{j}+1})
=\mathcal{F}_\lambda(\widetilde{X},\,\widetilde{Y}).
\end{equation*}
Since $(\widetilde{X},\,\widetilde{Y})\in\Gamma$ is arbitrary, we have $\Gamma\subseteq\mathcal{S}$ and $\mathcal{F}_\lambda\equiv\zeta$ on $\Gamma$. This completes the proof.
\end{proof}

\subsection{Convergence analysis II: global convergence and convergence rate}\label{sec-global}

In this section, we further study the convergence property of the whole sequence $\{(X^k,\,Y^k)\}$ generated by A-NAUM and establish the convergence rate of $\{(X^k,\,Y^k)\}$ under the Kurdyka-{\L}ojasiewicz (KL) property and its associated exponent. The analysis is inspired by recent advances on average-type nonmonotone line search techniques \cite{qtpq2025convergence,yhl2025nonmonotone}, but is more involved due to the block alternating structure of the proposed algorithm.

We first present the following technical lemma, whose proof is straightforward and is provided in Appendix~\ref{proof:lem:dist-inequa} for completeness.


\begin{lemma}\label{lem:dist-inequa}
Suppose that Assumption \ref{assumA} holds. Let $\{(X^k, \,Y^k)\}$ be the sequence generated by the A-NAUM in Algorithm \ref{algo-A-NAUM}. Then, there exists some $d>0$ such that
\begin{equation}\label{eq:dist-inequa}
\operatorname{dist}\big(0, \,\partial\mathcal{F}_{\lambda}(X^k,\,Y^k)\big) 
\leq d\big( \|X^k - X^{k-1}\|_F + \|Y^k - Y^{k-1}\|_F \big),
\quad \forall\,k\geq 0.
\end{equation}
\end{lemma}

We next establish the global convergence of the whole sequence $\{(X^k,\,Y^k)\}$.

\begin{theorem}\label{thm:fullconverge}
Suppose that Assumption \ref{assumA} holds. Let $\{(X^k, \,Y^k)\}$ be the sequence generated by the A-NAUM in Algorithm \ref{algo-A-NAUM}. If $\mathcal{F}_\lambda$ is a KL function, then the sequence $\{(X^k,\,Y^k)\}$ converges to a stationary point of problem \eqref{quadpenaltypro}.
\end{theorem}
\begin{proof}
In view of Theorem \ref{thm:A-convergence}, we have that $\lim\limits_{k\to\infty}\operatorname{dist}((X^k,Y^k),\Gamma)=0$ and $\Gamma\subseteq\mathcal{S}$, where $\Gamma$ is the set of all cluster points of $\{(X^{k},\,Y^{k})\}$ and $\mathcal{S}$ is the set of all stationary points of problem \eqref{quadpenaltypro}. Thus, we only need to show that the sequence $\{(X^k, Y^k)\}$ is convergent.

First, the fact that $\mathcal{F}_\lambda \equiv \zeta$ on $\Gamma$ (by Theorem \ref{thm:A-convergence}), together with the assumption that $\mathcal{F}_\lambda$ is a KL function, and the uniformized KL property (Proposition \ref{uniKL}), implies that there exist $\varepsilon>0$, $\nu>0$ and $\varphi \in \Phi_{\nu}$ such that
\begin{equation}\label{KL-inequa}
\varphi'\big(\mathcal{F}_\lambda(X,\,Y)-\zeta\big)\operatorname{dist}\big(0,\,\partial\mathcal{F}_\lambda(X,\,Y)\big)\geq1, 
\end{equation}
for all $(X,\,Y)$ satisfying $\operatorname{dist}\big((X,\,Y),\Gamma\big)<\varepsilon$ and $\zeta<\mathcal{F}_\lambda(X,\,Y)<\zeta+\nu$. Moreover, since $\{\mathcal{R}_k\}$ converges non-increasingly to $\zeta$ (by Proposition \ref{prop:Omegak-descentbehavior}(ii)), there exists an integer $K_1$ such that $\zeta<\mathcal{R}_k<\zeta+\nu$ holds for all $k\geq K_1$.

In the following, for notational simplicity, we define
\begin{equation}\label{notation}
\begin{aligned}
M&:=\left \lceil\textstyle\frac{1+\sqrt{1-p_{\min}}}{1-\sqrt{1-p_{\min}}}  \right \rceil ^2, \quad\ell(k):=k+M-1, \quad\Xi_{k}:=\sqrt{\mathcal{R}_{k}-\mathcal{R}_{k+1}},\\
\Delta_{i,j}^{\varphi}&:=\varphi(\mathcal{R}_i-\zeta)-\varphi(\mathcal{R}_j-\zeta), \quad\pi:= \sqrt{cp_{\min}}/2,
\end{aligned}
\end{equation}
where $\left\lceil a \right\rceil$ is the smallest integer greater than or equal to $a$. Then, it follows from \eqref{omegaupdate} that
\begin{equation}\label{xk-xkinequal}
\|X^{k+1}-X^{k}\|_F+\|Y^{k+1}-Y^{k}\|_F
\leq\sqrt{2\left(\|X^{k+1}-X^{k}\|_F^2+\|Y^{k+1}-Y^{k}\|_F^2\right)}
\leq\frac{\Xi_k}{\pi}.
\end{equation}
In addition, take any $(\widetilde{X},\,\widetilde{Y})\in \Gamma$ and let $\{(X^{k_j},\,Y^{k_j})\}_{j\in\mathbb{N}}$ be a subsequence converging to $(\widetilde{X},\,\widetilde{Y})$. With these preparations, we proceed to prove the convergence of $\{(X^{k},\,Y^{k})\}$ and divide the proof into four steps.

\textit{Step 1.}
We claim that, for the $\varepsilon>0$ given above, there exists a sufficiently large index $J>0$ such that the following inequality holds:
\begin{equation}\label{epsiloninequa}
\widehat{Q}:=\sup\limits_{j\geq J}\left\{\mathcal{Q}_j:=\|X^{k_j} - \widetilde{X}\|_F + \|Y^{k_j} - \widetilde{Y}\|_F + \frac{4}{\pi}\sum_{i={k_j}-1}^{\ell({k_j})-1} \Xi_{i}
+ \frac{d}{\pi^2} \sum_{i = {k_j}}^{\ell({k_j})} \varphi(\mathcal{R}_i-\zeta)\right\}<\varepsilon.
\end{equation}
Recall from \eqref{notation} that $\ell(k) - k = M - 1$, which is a fixed constant. Thus, the number of terms in each summation within $\mathcal{Q}_j$ is fixed and independent of $k$. Moreover, since $\{\mathcal{R}_k\}$ converges monotonically to $\zeta$ (by Proposition \ref{prop:Omegak-descentbehavior}(ii)) and $\varphi$ is continuous on $[0,\nu)$ with $\varphi(0)=0$ (by the properties required on the function $\varphi$ in the KL property), it follows that $\sum_{i = k_j}^{\ell(k_j)} \Xi_{i-1}\to 0$ and $\sum_{i = k_j}^{\ell(k_j)} \varphi(\mathcal{R}_i-\zeta)\to 0$ as $j\to\infty$. These, together with $(X^{k_j}, \,Y^{k_j}) \to (\widetilde{X}, \,\widetilde{Y})$ as $j\to\infty$ imply $\mathcal{Q}_j\to0$, and therefore there exists an index $J$ such that \eqref{epsiloninequa} holds.

\vspace{1mm}
\textit{Step 2.}
We show that
\begin{equation}\label{eq:2caseinequa}
\textstyle\frac{1-\sqrt{1-p_{\min}}}{\sqrt{M}}\sum_{i=k}^{\ell(k)}\Xi_i\le
\left(\frac{1}{2}+\sqrt{1-p_{\min}}\right)\Xi_{k-1}+\frac{d}{2\pi}\Delta^\varphi_{k,k+M}
\end{equation}
holds for all $k\in \left\{j\in \mathbb{N}:(X^j,Y^j)\in \mathcal{B}_{\widehat{Q}}(\widetilde{X},\widetilde{Y}),\ j\geq K_1\right\}$ with 
\begin{equation*}
\mathcal{B}_{\widehat{Q}}(\widetilde{X},\,\widetilde{Y})
:=\left\{(X,\,Y)\in\mathbb{R}^{n \times r} \times \mathbb{R}^{n \times r}:
\|X-\widetilde{X}\|_F+\|Y-\widetilde{Y}\|_F\leq\widehat{Q}\right\}.
\end{equation*}

To prove this, consider an arbitrary index $k$ from the above index set. For such $k\geq K_1$, we have that $\zeta<\mathcal{R}_k<\zeta+\nu$, $\|X^k-\widetilde{X}\|_F+\|Y^k-\widetilde{Y}\|_F\leq\widehat{Q}<\varepsilon$, and the inequality \eqref{eq:dist-inequa} holds. Moreover, by Jensen's inequality, we have that
\begin{align*}
{\textstyle\frac{1}{\sqrt{M}}}\sqrt{\mathcal{R}_k-\mathcal{R}_{k+M}}
&=\sqrt{{\textstyle\frac{1}{M}}\left(\mathcal{R}_k-\mathcal{R}_{k+1}+\dots+\mathcal{R}_{k+M-1}-\mathcal{R}_{k+M}\right)}  \\
&\geq {\textstyle\frac{1}{M}} \left(\sqrt{\mathcal{R}_k-\mathcal{R}_{k+1}}+\dots+\sqrt{\mathcal{R}_{k+M-1}-\mathcal{R}_{k+M}}\right).
\end{align*}
This, together with $\ell(k)=k+M-1$, $\Xi_{k}=\sqrt{\mathcal{R}_{k}-\mathcal{R}_{k+1}}$, and $1-\sqrt{1-p_{\min}}>0$ (due to $p_{\min}\in(0,1)$), yields
\begin{equation}\label{eq:Jenseninequa}
\textstyle\frac{1-\sqrt{1-p_{\min}}}{\sqrt{M}}\sum_{i=k}^{\ell(k)}\Xi_i
\leq\left(1-\sqrt{1-p_{\min}}\right)\sqrt{\mathcal{R}_k-\mathcal{R}_{k+M}}.
\end{equation}
Next, we prove \eqref{eq:2caseinequa} by estimating the right-hand side of \eqref{eq:Jenseninequa} in two cases.

\textbf{Case 1:}
$\mathcal{F}_\lambda(X^{k},\,Y^{k})\leq\mathcal{R}_{k+M}$. In this case, from the updating rule of $\mathcal{R}_k$ in \eqref{eq:omegadefinition}, we have
\begin{align*}
&\mathcal{R}_k-\mathcal{R}_{k+M} 
= (1-p_k)\mathcal{R}_{k-1}+p_k\mathcal{F}_\lambda(X^{k},\,Y^{k})-\mathcal{R}_{k+M}\\
&\leq (1-p_k)\mathcal{R}_{k-1}+p_k\mathcal{R}_{k+M}-\mathcal{R}_{k+M}=(1-p_k)(\mathcal{R}_{k-1}-\mathcal{R}_{k+M})\\
&\leq (1-p_{\min})(\mathcal{R}_{k-1}-\mathcal{R}_{k+M})\quad
\quad ({\rm by}\,\,p_k\in [p_{\min},1]\,\,{\rm and}\,\,\mathcal{R}_{k-1}\geq\mathcal{R}_{k+M})\\
&= (1-p_{\min})(\mathcal{R}_{k-1}-\mathcal{R}_k+\mathcal{R}_k-\mathcal{R}_{k+M}).
\end{align*}
Taking square roots on both sides of this inequality and using the fact that $\sqrt{a+b}\leq\sqrt{a}+\sqrt{b}$ for all $a$, $b\geq 0$, we obtain by rearranging the resulting terms that
\begin{equation*}
\big(1-\sqrt{1-p_{\min}}\big)\sqrt{\mathcal{R}_k-\mathcal{R}_{k+M}}
\leq \sqrt{1-p_{\min}}\,\Xi_{k-1}.
\end{equation*}
Substituting this into \eqref{eq:Jenseninequa} leads to the desired inequality \eqref{eq:2caseinequa} for this case.

\textbf{Case 2:}
$\mathcal{F}_\lambda(X^{k},\,Y^{k})>\mathcal{R}_{k+M}$. In this case, by the choice of $k$ and Proposition \ref{prop:Omegak-descentbehavior}(i), we have that $\zeta<\mathcal{R}_{k+M}<\mathcal{F}_\lambda(X^{k},\,Y^{k})\leq\mathcal{R}_k<\zeta+\nu$ and $\operatorname{dist}\big((X^k,Y^k),\Gamma\big)<\varepsilon$. Thus, it follows from \eqref{KL-inequa} that
\begin{equation}\label{KLproper}
\varphi'\big(\mathcal{F}_\lambda(X^{k},\,Y^{k})-\zeta\big)\operatorname{dist}\big(0,\,\partial\mathcal{F}_\lambda(X^{k},\,Y^{k})\big)\geq1.
\end{equation}
Moreover, we see that
\begin{align*}
&\quad\operatorname{dist}\big(0,\,\partial\mathcal{F}_\lambda(X^{k},\,Y^{k})\big)\cdot \Delta^\varphi_{k,k+M}\\
&=\operatorname{dist}\big(0,\,\partial\mathcal{F}_\lambda(X^{k},\,Y^{k})\big)\cdot [\varphi(\mathcal{R}_k-\zeta)-\varphi(\mathcal{R}_{k+M}-\zeta)]\\
&\geq\operatorname{dist}\big(0,\,\partial\mathcal{F}_\lambda(X^{k},\,Y^{k})\big)\cdot \left[\varphi\big(\mathcal{F}_\lambda(X^{k},\,Y^{k})-\zeta\big)-\varphi(\mathcal{R}_{k+M}-\zeta)\right]\\
&\geq\operatorname{dist}\big(0,\,\partial\mathcal{F}_\lambda(X^{k},\,Y^{k})\big)\cdot \varphi'\big(\mathcal{F}_\lambda(X^{k},\,Y^{k})-\zeta\big)\cdot\big(\mathcal{F}_\lambda(X^{k},\,Y^{k})-\mathcal{R}_{k+M}\big)\\
&\geq\mathcal{F}_\lambda(X^{k},\,Y^{k})-\mathcal{R}_{k+M},
\end{align*}
where the first inequality follows from the monotonicity of $\varphi$ and $\mathcal{F}_\lambda(X^{k},\,Y^{k})\leq\mathcal{R}_k$ by Proposition \ref{prop:Omegak-descentbehavior}(i); the second inequality follows from the concavity of $\varphi$; the last inequality follows from \eqref{KLproper} and the hypothesis $\mathcal{F}_\lambda(X^{k}, \,Y^{k})>\mathcal{R}_{k+M}$. This, together with \eqref{eq:dist-inequa} and \eqref{xk-xkinequal}, yields
\begin{equation}\label{f-omega}
\mathcal{F}_\lambda(X^{k},\,Y^{k})-\mathcal{R}_{k+M}
\leq\textstyle\frac{d}{\pi}\Delta^\varphi_{k,k+M}\Xi_{k-1}.
\end{equation}
Now, using the updating rule of $\mathcal{R}_k$ in \eqref{eq:omegadefinition}, we have that
\begin{align*}
\mathcal{R}_k-\mathcal{R}_{k+M}
&=(1-p_k)\mathcal{R}_{k-1}+p_k\mathcal{F}_\lambda(X^{k},\,Y^{k})-\mathcal{R}_{k+M}\\
&=(1-p_k)(\mathcal{R}_{k-1}-\mathcal{R}_{k+M})+p_k(\mathcal{F}_\lambda(X^{k},\,Y^{k})-\mathcal{R}_{k+M})\\
&\leq (1-p_{\min})\big(\mathcal{R}_{k-1}-\mathcal{R}_k+\mathcal{R}_k-\mathcal{R}_{k+M}\big)
+ \textstyle\frac{d}{\pi}\Delta^\varphi_{k,k+M}\Xi_{k-1},
\end{align*}
where the inequality follows from $p_k\in[p_{\min},1]$, $\mathcal{F}_\lambda(X^{k}, \,Y^{k})>\mathcal{R}_{k+M}$, and \eqref{f-omega}. Taking square roots on both sides of the above inequality and using the fact $2\sqrt{ab}\leq a+b$ and $\sqrt{a+b}\leq\sqrt{a}+\sqrt{b}$ for all $a,\,b\geq0$, we have that
\begin{align*}
\sqrt{\mathcal{R}_k-\mathcal{R}_{k+M}}
&\leq \sqrt{(1-p_{\min})\big(\mathcal{R}_{k-1}-\mathcal{R}_k+\mathcal{R}_k-\mathcal{R}_{k+M}\big)+\textstyle\frac{d}{\pi}\Delta^\varphi_{k,k+M}\Xi_{k-1}}\\
&\leq\sqrt{1-p_{\min}}\,\Xi_{k-1} + \sqrt{1-p_{\min}}\,\sqrt{\mathcal{R}_k-\mathcal{R}_{k+M}}+\textstyle\frac{1}{2}\Xi_{k-1}+\frac{d}{2\pi}\Delta^\varphi_{k,k+M}\\
&=\sqrt{1-p_{\min}}\sqrt{\mathcal{R}_k-\mathcal{R}_{k+M}}+\left(\textstyle\frac{1}{2}+\sqrt{1-p_{\min}}\right)\Xi_{k-1}+\textstyle\frac{d}{2\pi}\Delta^\varphi_{k,k+M},
\end{align*}
which implies that
\begin{equation*}
\big(1-\sqrt{1-p_{\min}}\big)\sqrt{\mathcal{R}_k-\mathcal{R}_{k+M}}
\leq\left(\textstyle\frac{1}{2}+\sqrt{1-p_{\min}}\right)\Xi_{k-1}
+\textstyle\frac{d}{2\pi}\Delta^\varphi_{k,k+M}.
\end{equation*}
Substituting this into \eqref{eq:Jenseninequa} further implies inequality \eqref{eq:2caseinequa} also holds for this case.

\vspace{1mm}
\textit{Step 3.} Without loss of generality, we assume that $J$ is a sufficiently large index such that $k_J\geq K_1$. For such $k_J$, we claim that the following relations hold for all $k\geq \ell(k_J)$: 
\begin{numcases}{}
(X^k,\,Y^k)\in\mathcal{B}_{\widehat{Q}}(\widetilde{X}, \,\widetilde{Y}), \label{first-statement} \\
\textstyle\sum_{i = \ell(k_J)}^k \Xi_i\leq \left(1 + 2\sqrt{1 - p_{\min}}\right) \sum_{i = k_J-1}^{\ell(k_J)-1} \Xi_{i} + \frac{d}{\pi} \sum_{ i = k_J}^{\ell(k_J)} \varphi(\mathcal{R}_i - \zeta). \label{second-statement}
\end{numcases}

We prove by induction. For all $k \in \{k_J, \ldots, \ell(k_J)\}$, it follows from \eqref{xk-xkinequal} and \eqref{epsiloninequa} that
\begin{align*}
&\quad~\|X^k-\widetilde{X}\|_F+\|Y^k-\widetilde{Y}\|_F \\
&=\|X^k-X^{k-1}+\cdots+X^{k_J}-\widetilde{X}\|_F+\|Y^k-Y^{k-1}+\cdots+Y^{k_J}-\widetilde{Y}\|_F\\
&\leq \|X^{k_J}-\widetilde{X}\|_F + \|Y^{k_J}-\widetilde{Y}\|_F 
+ {\textstyle\sum_{i=k_J}^{k-1}}\left(\|X^{i+1}-X^{i}\|_F
+ \|Y^{i+1}-Y^{i}\|_F\right) \\
&\leq\|X^{k_J}-\widetilde{X}\|_F + \|Y^{k_J}-\widetilde{Y}\|_F 
+ {\textstyle\sum_{i=k_J-1}^{\ell(k_J)-1}}\left(\|X^{i+1}-X^{i}\|_F
+ \|Y^{i+1}-Y^{i}\|_F\right) \\
&\leq \|X^{k_J}-\widetilde{X}\|_F + \|Y^{k_J}-\widetilde{Y}\|_F
+ {\textstyle\frac{1}{\pi}\sum_{i=k_J-1}^{\ell(k_J)-1}}\Xi_{i}
\leq \widehat{Q}.
\end{align*}
This proves \eqref{first-statement} for $k = k_J, \ldots, \ell(k_J)$. Using this fact, we see that
\begin{equation}\label{sqrtM-Xi-ellkJ}
\begin{aligned}
&{\textstyle\big(1-\sqrt{1-p_{\min}}\big)\sqrt{M}\,\Xi_{\ell(k_J)} 
\leq \frac{1-\sqrt{1-p_{\min}}}{\sqrt{M}}\sum_{i=k_J}^{\ell(k_J)}\sum_{t=i}^{\ell(i)}\Xi_t} \\
\leq& {\textstyle\left(\frac{1}{2}+\sqrt{1-p_{\min}}\right)\sum_{i=k_J}^{\ell(k_J)}\Xi_{i-1}
+\frac{d}{2\pi}\sum_{i=k_J}^{\ell(k_J)}\Delta^\varphi_{i,i+M}} \\
\leq& {\textstyle\left(\frac{1}{2}+\sqrt{1-p_{\min}}\right)\sum_{i=k_J-1}^{\ell(k_J)-1}\Xi_{i}
+ \frac{d}{2\pi}\sum_{i=k_J}^{\ell(k_J)}\varphi(\mathcal{R}_i-\zeta)}\\
\leq&{\textstyle\left(\frac{1}{2}+\sqrt{1-p_{\min}}\right)\left(\sum_{i=k_J-1}^{\ell(k_J)-1}\Xi_{i}+\Xi_{\ell(k_J)} \right)
+ \frac{d}{2\pi}\sum_{i=k_J}^{\ell(k_J)}\varphi(\mathcal{R}_i-\zeta)},
\end{aligned}
\end{equation}
where the first inequality follows from the nonnegativity of $\Xi_k$ and the fact that the term $\Xi_{\ell(k_J)}$ occurs $M$ times in the double sum; the second inequality is obtained by applying \eqref{eq:2caseinequa} to each $k\in\{k_J ,k_J+1, \ldots, \ell(k_J)\}$; the third inequality follows from $\Delta^\varphi_{i, i+M}:=\varphi(\mathcal{R}_i-\zeta)-\varphi(\mathcal{R}_{i+M}-\zeta)\leq \varphi(\mathcal{R}_i-\zeta)$ for any $i$. By the definition of $M$ in \eqref{notation}, it is straightforward to verify that 
\begin{equation}\label{Minequal}
(1-\sqrt{1 - p_{\min}}) \sqrt{M} - \left(\textstyle\frac{1}{2}+ \sqrt{1 - p_{\min}}\right)\geq \textstyle\frac{1}{2}.  
\end{equation}
This, together with \eqref{sqrtM-Xi-ellkJ}, yields that
\begin{align*}
\Xi_{\ell(k_J)} \leq \left(1 + 2\sqrt{1 - p_{\min}}\right) \textstyle\sum_{i = k_J-1}^{\ell(k_J)-1} \Xi_{i}+\frac{d}{\pi} \sum_{i = k_J}^{\ell(k_J)}\varphi(\mathcal{R}_i-\zeta),
\end{align*}
which implies that \eqref{second-statement} holds for $k = \ell(k_J)$.

Next, suppose that \eqref{first-statement} and \eqref{second-statement} hold for all $k$ from $\ell(k_J)$ to some $K \geq \ell(k_J)$. It remains to show that they also hold for $k=K+1$. Indeed, since \eqref{second-statement} holds for $K$ and $1+2\sqrt{1-p_{\min }}\leq3$ (due to $p_{\min}\in(0,1)$), we have that
\begin{equation*}
\textstyle\sum_{i = \ell(k_J)}^K \Xi_i \leq 3 \sum_{i = k_J-1}^{\ell(k_J)-1} \Xi_{i} + \frac{d}{\pi} \sum_{i = k_J}^{\ell(k_J)} \varphi(\mathcal{R}_i-\zeta),
\end{equation*}
which further implies
\begin{equation*}
\textstyle\sum_{i=k_J}^K \Xi_i 
\leq \textstyle\sum_{i=k_J-1}^K \Xi_i 
= \sum_{i = k_J-1}^{\ell(k_J)-1} \Xi_{i} + \sum_{i = \ell(k_J)}^K \Xi_i 
\leq 4 \sum_{i = k_J-1}^{\ell(k_J)-1} \Xi_{i} + \frac{d}{\pi} \sum_{i = k_J}^{\ell(k_J)} \varphi(\mathcal{R}_i-\zeta).
\end{equation*}
Using this relation, together with \eqref{xk-xkinequal}, we see that
\begin{align*}
&\quad~\|X^{K+1}-\widetilde{X}\|_F+\|Y^{K+1}-\widetilde{Y}\|_F \\
&= \|X^{K+1}-X^{K}+\cdots+X^{k_J}-\widetilde{X}\|_F+\|Y^{K+1}-Y^{K}+\cdots+Y^{k_J}-\widetilde{Y}\|_F \\
&\leq \|X^{k_J} - \widetilde{X}\|_F + \|Y^{k_J} - \widetilde{Y}\|_F +\textstyle\sum_{i = k_J}^K \left(\|X^{i+1} - X^{i}\|_F+\|Y^{i+1} - Y^{i}\|_F\right) \\
&\leq \|X^{k_J} - \widetilde{X}\|_F + \|Y^{k_J} - \widetilde{Y}\|_F + \textstyle\frac{1}{\pi}\sum_{i = k_J}^{K}\Xi_{i}\\
&\leq \|X^{k_J } - \widetilde{X}\|_F + \|Y^{k_J } - \widetilde{Y}\|_F + \textstyle\frac{4}{\pi}\sum_{i = k_J-1}^{\ell(k_J)-1} \Xi_{i}+\frac{d}{\pi^2} \sum_{i = k_J}^{\ell(k_J)} \varphi(\mathcal{R}_i-\zeta).
\end{align*}
This, togehther with \eqref{epsiloninequa}, implies that \eqref{first-statement} holds for $k=K + 1$.

We now verify \eqref{second-statement} for $k=K + 1$. By the above discussion, we see that $(X^k,\,Y^k) \in \mathcal{B}_{\widehat{Q}}(\widetilde{X},\,\widetilde{Y})$ holds for all $k \in \{k_J, k_J+1,\ldots, K + 1\}$. Using this fact, we have that
\begin{equation}\label{sqrtM-sum}
\begin{aligned}
&\textstyle(1-\sqrt{1 - p_{\min}}) \sqrt{M} \sum_{i = \ell(k_J)}^{K+1} \Xi_i 
\leq \frac{1 - \sqrt{1 - p_{\min}}}{\sqrt{M}} \sum_{i = k_J}^{K+1} \sum_{t = i}^{\ell(i)} \Xi_t \\
&\leq \textstyle \left(\frac{1}{2} + \sqrt{1 - p_{\min}}\right) \sum_{i = k_J}^{K+1} \Xi_{i-1} + \frac{d}{2\pi} \sum_{i = k_J}^{K+1} \Delta^\varphi_{i, i+M} \\
&\leq\textstyle\left(\frac{1}{2} + \sqrt{1 - p_{\min}}\right) \sum_{i = k_J-1}^{K} \Xi_{i} + \frac{d}{2\pi}\sum_{i = k_J}^{\ell(k_J)} \varphi(\mathcal{R}_i-\zeta)\\
&\leq\textstyle\left(\frac{1}{2} + \sqrt{1 - p_{\min}}\right)\left( \sum_{i = k_J-1}^{\ell(k_J)-1} \Xi_{i}+\sum_{i = \ell(k_J)}^{K+1} \Xi_{i}\right)+ \frac{d}{2\pi}\sum_{i = k_J}^{\ell(k_J)} \varphi(\mathcal{R}_i-\zeta),
\end{aligned} 
\end{equation}
where the first inequality follows from the nonnegativity of $\Xi_k$ and the fact that each term $\Xi_i$ for $i \in\{ \ell(k_J),\ell(k_J)+1, \ldots, K+ 1\} $ occurs $M$ times in the double sum; the second inequality is obtained by applying \eqref{eq:2caseinequa} to each $k\in\{k_J ,k_J+1,\ldots, K+1\}$; and the third inequality holds because
\begin{align*}
\sum_{i = k_J}^{K+1} \Delta^\varphi_{i, i+M}
&=\sum_{i = k_J}^{K+1}\varphi(\mathcal{R}_i-\zeta)-\sum_{i = k_J}^{K+1}\varphi(\mathcal{R}_{i+M}-\zeta)=\sum_{i = k_J}^{K+1}\varphi(\mathcal{R}_i-\zeta)-\sum_{i = k_J+M}^{k+M+1}\varphi(\mathcal{R}_{i}-\zeta)\\
&=\sum_{i = k_J}^{k_J+M-1}\varphi(\mathcal{R}_i-\zeta)-\sum_{i = K+2}^{k+M+1}\varphi(\mathcal{R}_{i}-\zeta)
\leq\sum_{i = k_J}^{\ell(k_J)}\varphi(\mathcal{R}_i-\zeta).
\end{align*}
Using \eqref{sqrtM-sum}, together with \eqref{Minequal}, we obtain that 
\begin{align*}
\textstyle\sum_{i = \ell(k_J)}^{K+1} \Xi_i\leq \left(1 + 2\sqrt{1 - p_{\min}}\right) \sum_{i = k_J-1}^{\ell(k_J)-1} \Xi_{i} + \frac{d}{\pi}\sum_{i = k_J}^{\ell(k_J)} \varphi(\mathcal{R}_i-\zeta)
\end{align*}
which shows that \eqref{second-statement} holds for $k=K+1$ and completes the induction.

\vspace{1mm}
\textit{Step 4.}
Since \eqref{second-statement} holds for all $k\geq\ell(k_J)$, taking the limit in \eqref{second-statement} yields
\begin{equation*}
\textstyle\sum_{i = \ell(k_J)}^\infty \Xi_i \leq \left(1 + 2\sqrt{1 - p_{\min}}\right)\sum_{i = k_J-1}^{\ell(k_J)-1} \Xi_{i} + \frac{d}{\pi} \sum_{i = k_J}^{\ell(k_J)} \varphi(\mathcal{R}_i-\zeta)<\infty.
\end{equation*}
This, together with \eqref{xk-xkinequal}, yields
\begin{equation*}
{\textstyle\sum_{i=\ell(k_J)}^\infty}\left(\|X^{i+1}-X^i\|_F+\|Y^{i+1}-Y^i\|_F\right)
\leq{\textstyle\frac{1}{\pi}\sum_{i=\ell(k_J)}^\infty} \Xi_i 
< \infty,
\end{equation*}
which implies that ${\textstyle\sum_{i=0}^\infty}\left(\|X^{i+1}-X^i\|_F+\|Y^{i+1}-Y^i\|_F\right)<\infty$ and hence $\{(X^k,\,Y^k)\}$ is convergent. We then complete the proof.
\end{proof}

Based on the KL exponent, we further characterize the convergence rates of both the iterate sequence $\{(X^k,\,Y^k)\}$ and the corresponding objective function values. 

\begin{theorem}\label{thm:fun_rate}
Suppose that Assumption \ref{assumA} holds and $\mathcal{F}_\lambda$ is a KL function with an exponent $\theta \in[0,1)$. Let $\{(X^k, \,Y^k)\}$ be the sequence generated by the A-NAUM in Algorithm \ref{algo-A-NAUM}, and let $\zeta$ be given in Proposition \ref{prop:Omegak-descentbehavior}(ii). Then, the following statements hold for all sufficiently large $k$. 
\begin{itemize}
\item[{\rm(i)}] If $\theta=0$, there exist $c_1>0$ and $\eta_1\in(0,1)$ such that 
$\zeta-c_1\eta_1^k\leq \mathcal{F}_\lambda(X^k,\,Y^k)\leq\zeta$.

\item[{\rm(ii)}] If $\theta\in\left(0,\frac{1}{2}\right]$, there exist $c_2>0$ and $\eta_2 \in(0,1)$ such that 
$|\mathcal{F}_\lambda(X^k,\,Y^k)-\zeta|\leq c_2 \eta_2^k$.

\item[{\rm(iii)}] If $\theta\in\left(\frac{1}{2}, 1\right)$, there exists $c_3>0$ such that 
$|\mathcal{F}_\lambda(X^k,\,Y^k)-\zeta|\leq c_3k^{-\frac{1}{2\theta-1}}$.
\end{itemize}
\end{theorem}
\begin{proof}
See Appendix \ref{proof:thm:fun_rate}.
\end{proof}

\begin{theorem}\label{thm:iter_rate}
Under the same conditions of Theorem \ref{thm:fun_rate}, 
and let $(\widetilde{X},\,\widetilde{Y})$ be the limit point, the following statements hold for all sufficiently large $k$.
\begin{itemize}
\item[{\rm(i)}] If $\theta\in\left[0,\frac{1}{2}\right]$, there exist $d_1>0$ and $\varrho \in(0,1)$ such that $\|X^k-\widetilde{X}\|_F + \|Y^k-\widetilde{Y}\|_F\leq d_1\varrho^k$.

\item[{\rm(ii)}] If $\theta\in\left(\frac{1}{2},1\right)$, there exists $d_2>0$ such that $\|X^k-\widetilde{X}\|_F + \|Y^k-\widetilde{Y}\|_F\leq d_2k^{-\frac{1-\theta}{2\theta-1}}$.
\end{itemize}
\end{theorem}
\begin{proof}
See Appendix \ref{proof:thm:iter_rate}. 
\end{proof}

\section{Numerical experiments}\label{sec-numer}

In this section, we evaluate the proposed A-NAUM for solving the following \emph{approximate symmetric nonnegative matrix factorization} problem:
\begin{equation}\label{eq:GSNMF}
\min_{X,Y\in\mathbb{R}^{n\times r}}
~~\mathcal{F}_\lambda(X,\,Y):=\frac{1}{2}\|XY^\top-M\|_F^2+\frac{\lambda}{2}\|X-Y\|_F^2
\quad\text{s.t.}\quad X\geq 0, ~Y\geq 0,
\end{equation}
which is a representative and practically important instance of \eqref{quadpenaltypro}. Here, $M\in\mathbb{R}^{n\times n}$ is a given symmetric target matrix, $r>0$ is the factorization rank, and $\lambda\geq 0$ controls the strength of symmetry enforcement. Problem \eqref{eq:GSNMF} corresponds to \eqref{quadpenaltypro} with $\Psi(X)=\delta_{\mathbb{R}_+^{n\times r}}(X)$, $\Phi(Y)=\delta_{\mathbb{R}_+^{n\times r}}(Y)$, and $\mathcal{A}=\mathcal{I}$. Consequently, A-NAUM applies directly, and the hierarchical-prox updates \eqref{supro:U-hier}--\eqref{supro:V-hier} yield closed-form column updates. When $\lambda>0$, the model promotes symmetry, while the limiting case $\lambda=0$ reduces to the standard NMF formulation. All experiments are implemented in MATLAB R2022a and executed on the same machine.

\vspace{2mm}
\noindent\textbf{Implementation of A-NAUM.} 
For solving problem \eqref{eq:GSNMF}, the explicit $Z$-update reduces to
\begin{equation*}
\textstyle Z^k=\frac{\alpha}{\alpha+\beta}X^k(Y^k)^\top+\frac{\beta}{\alpha+\beta}M,\qquad \forall k\geq 0,
\end{equation*}
and the hierarchical column updates take the following explicit forms
\begin{align*}
\bm{u}_i
&=\max\left\{0,\;
\frac{\alpha P_i^k\bm{y}_i^k+\lambda\bm{y}_i^k+\mu_k\bm{x}_i^k}{\alpha\|\bm{y}_i^k\|^2+\lambda+\mu_k}\right\},
\;
\bm{v}_i
=\max\left\{0,\;
\frac{\alpha(Q_i^k)^\top\bm{u}_i+\lambda\bm{u}_i+\sigma_k\bm{y}_i^k}{\alpha\|\bm{u}_i\|^2+\lambda+\sigma_k}\right\},
\quad \forall 1\leq i\leq n,
\end{align*}
where $P_i^k$ and $Q_i^k$ are defined through
\begin{equation}\label{eq:PQ2} \begin{aligned} P_i^k \bm{y}_i^k &=\textstyle \frac{\alpha}{\alpha + \beta} X^k (Y^k)^\top \bm{y}_i^k + \frac{\beta}{\alpha + \beta} M \bm{y}_i^k - \sum_{j=1}^{i-1} \bm{u}_j (\bm{y}_j^k)^\top \bm{y}_i^k - \sum_{j=i+1}^{r} \bm{x}_j^k (\bm{y}_j^k) ^\top\bm{y}_i^k,\\ (Q_i^k)^\top \bm{u}_i &=\textstyle \frac{\alpha}{\alpha + \beta} Y^k (X^k)^\top \bm{u}_i + \frac{\beta}{\alpha + \beta} M^\top \bm{u}_i - \sum_{j=1}^{i-1} \bm{v}_j\bm{u_j}^\top \bm{u}_i - \sum_{j=i+1}^{r} \bm{y}_j^k\bm{u_j}^\top \bm{u}_i. \end{aligned} \end{equation}
When evaluating terms such as $X^k (Y^k)^\top \bm{y}_i^k$ and $Y^k (X^k)^\top \bm{u}_i$ above, we avoid explicitly forming the potentially massive $(n\times n)$ matrix $X^k (Y^k)^\top$. Instead, we compute the smaller intermediate products $(Y^k)^\top \bm{y}_i^k$ and $(X^k)^\top \bm{u}_i$ first, and then multiply by $X^k$ or $Y^k$, respectively. Moreover, the Gram-type quantities $(X^k)^\top U$, $U^\top U$, $(Y^k)^\top Y^k$ and $M^\top U$ computed in \eqref{eq:PQ2} can be reused to efficiently evaluate successive changes and the objective value via
\begin{align*}
\|U - X^k\|_F^2 
&= \mathrm{tr}(U^\top U) - 2\,\mathrm{tr}\big((X^k)^\top U\big) + \mathrm{tr}\big((X^k)^\top X^k\big),\\[5pt]
\|V - Y^k\|_F^2 
&= \mathrm{tr}(V^\top V) - 2\,\mathrm{tr}\big((Y^k)^\top V\big) + \mathrm{tr}\big((Y^k)^\top Y^k\big),\\[5pt]
\|UV^\top - M\|_F^2 
&= \mathrm{tr}\big((U^\top U)(V^\top V)\big) - 2\,\mathrm{tr}\big((M^\top U)V\big) + \|M\|_F^2.
\end{align*}
Here, $\mathrm{tr}((X^k)^\top X^k)$ and $\mathrm{tr}((Y^k)^\top Y^k)$ can be obtained directly from the previously computed $U^\top U$ and $V^\top V$ (from the preceding iteration), and $\|M\|_F^2$ can be precomputed only once.

\vspace{2mm}
\noindent\textbf{Baselines and evaluation metrics.}
We compare A-NAUM with M-NAUM \cite{ypc2018nonmonotone} to evaluate the performance of the average-type nonmonotone line search against that of the max-type nonmonotone line search. We also compare A-NAUM with SymHALS~\cite{lzll2023provable}, a competitive column-wise updating method tailored for solving problem \eqref{eq:GSNMF}. Extensive comparisons in~\cite{lzll2023provable} indicate that SymHALS is the most efficient among existing algorithms, and thus it serves as a strong baseline for performance evaluation. For each run, we record running time, the number of iterations denoted by \texttt{iter}, and the relative objective value defined by $\texttt{relobj}:=\frac{\sqrt{2\mathcal{F}_\lambda}}{\|M\|_F}$. In addition, we plot the relative objective value against time to assess the practical convergence behavior.

\vspace{2mm}
\noindent\textbf{Initialization and termination.} 
All methods are initialized from the same starting point $(X^0,Y^0)$, whose entries are independently drawn from the standard uniform distribution on $(0,1)$. They are terminated when 
\begin{equation*}
\frac{|\mathcal{F}^{k}_{\lambda}-\mathcal{F}^{k-1}_{\lambda}|}{\mathcal{F}^{k}_{\lambda}+1}
\leq \texttt{tol}
\end{equation*}
for 3 consecutive iterations, where $\mathcal{F}^k_{\lambda}$ is the objective value at $(X^k,Y^k)$ and $\texttt{tol}$ is a given positive tolerance; or, the maximum allowed runtime is exceeded.


\vspace{2mm}
\noindent\textbf{Parameter settings.}
A-NAUM uses
$\mu^{\min}=\bar{\mu}_{-1}=1$,
$\sigma^{\min}=\bar{\sigma}_{-1}=1$,
$\sigma^{\max}=10^6$,
$\tau=4$,
$c=10^{-4}$,
$\mu_k^0=\max\{0.1\bar{\mu}_{k-1},\mu^{\min}\}$, and
$\sigma_k^0=\min\{\max\{0.1\bar{\sigma}_{k-1},\sigma^{\min}\},\sigma^{\max}\}$. We set $p_k\equiv0.2$ for the average-type reference value and $N=3$ for the max-type reference value. For a given parameter $\alpha$, set $\beta=\frac{\alpha}{\alpha-1}$, $\gamma=\max\big\{0,\,-\alpha,\,-(\alpha+\beta)\big\}$, and $\rho=\max\left\{1,\,\frac{\alpha^2}{(\alpha+\beta)^2}\right\}$.

\vspace{2mm}
\noindent\textbf{Real datasets.}
We conduct experiments on two face datasets: ORL\footnote{\url{https://www.cl.cam.ac.uk/research/dtg/attarchive/facedatabase.html}} and CBCL\footnote{\url{http://cbcl.mit.edu/software-datasets/FaceData2.html}}. 
ORL contains 400 images of faces with $112\times92$ pixels and CBCL contains 2429 images of faces with $19\times19$ pixels. For each dataset, we randomly select $n$ images and vectorize them to form a matrix $N$ with $n$ columns. Then, given a noise level parameter $t\geq0$, we construct the matrix $M\in\mathbb{R}^{n\times n}_+$ in~\eqref{eq:GSNMF} via (in MATLAB syntax): \texttt{M = N'*N; M = M/max(M(:))~+~t*abs(randn(n,n))}.

\subsection{Comparisons under different $\alpha$ and $\beta$}\label{sec:alpha_synth}

We first compare A-NAUM and M-NAUM by examining the effect of $\alpha$, and report the average results over 5 independent runs with $\texttt{tol}=10^{-12}$ in Table~\ref{table:diffe-alpha}. The results show that the choice of $\alpha$ has a noticeable impact on performance. In particular, $\alpha\in\{0.6,0.8\}$ consistently provides a favorable balance among iteration count, computational time, and the final relative objective value for both methods. By contrast, choices with $\alpha>1$ tend to be slower, which highlights the practical advantage of the relaxed relation $\tfrac{1}{\alpha}+\tfrac{1}{\beta}=1$, under which $\beta<0$ is allowed when $\alpha<1$. We also observe that A-NAUM performs comparably to M-NAUM in most cases, and is slightly better for the empirically best choices $\alpha\in\{0.6,0.8\}$. This is indeed unsurprising, since the two methods share essentially the same algorithmic framework and differ only in the nonmonotone line search criterion. However, A-NAUM admits substantially stronger convergence properties (including convergence of the entire sequence and convergence rate results) under weaker assumptions, which may make it more reliable in practice. Based on these observations, we fix $\alpha=0.6$ and report only the results of A-NAUM in the remaining experiments.


\begin{table}[htb!]
\caption{Comparisons between M-NAUM and A-NAUM with different $\alpha$, $t=0.001$, $r=5$, and $\lambda\in \{0,1\}$ on ORL. 
}\label{table:diffe-alpha}
\renewcommand\arraystretch{1}
\centering \tabcolsep 4.5pt
\scalebox{0.9}{
\begin{tabular}{|c|ccc|ccc|ccc|ccc|}
\hline
&\multicolumn{6}{c|}{$\lambda=0$}
&\multicolumn{6}{c|}{$\lambda=1$}\\
\hline
& \multicolumn{3}{c|}{M-NAUM} & \multicolumn{3}{c|}{A-NAUM}
& \multicolumn{3}{c|}{M-NAUM} & \multicolumn{3}{c|}{A-NAUM} \\
\hline
$\alpha$ & \texttt{iter} & \texttt{relobj} & \texttt{time} & \texttt{iter} 
& \texttt{relobj} & \texttt{time} & \texttt{iter} & \texttt{relobj} 
& \texttt{time} & \texttt{iter} & \texttt{relobj} & \texttt{time}  \\
\hline
 0.2 & 181308 & 8.38e-03 & 112.86 & 319155 & 8.38e-03 & 203.13 & 38628 & 8.38e-03 & 24.15 & 31660 & 8.38e-03 & 18.09 \\  
 0.4 & 69086 & 8.38e-03 & 37.38 & 120227 & 8.38e-03 & 69.39    & 16440 & 8.38e-03 & 9.23 & 32363 & 8.38e-03 & 18.40 \\  
 0.6 & 41026 & 8.38e-03 & 11.44 & 39884 & 8.38e-03 & 11.14     & 8114  & 8.38e-03 & 2.25 &  8255 & 8.38e-03 & 2.36 \\  
 0.8 & 51993 & 8.38e-03 & 14.55 & 50518 & 8.38e-03 & 14.25     & 8056  & 8.38e-03 & 2.21 &  7872 & 8.38e-03 & 2.15 \\  
 1.2 & 49643 & 8.38e-03 & 13.67 & 47723 & 8.38e-03 & 13.11     & 24270 & 8.38e-03 & 6.69 & 24657 & 8.38e-03 & 6.77 \\  
 1.4 & 53015 & 8.38e-03 & 14.76 & 51237 & 8.38e-03 & 14.63     & 25345 & 8.38e-03 & 7.10 & 25391 & 8.38e-03 & 7.01 \\  
 1.6 & 60417 & 8.38e-03 & 17.07 & 57989 & 8.38e-03 & 15.73     & 31909 & 8.38e-03 & 8.94 & 31386 & 8.38e-03 & 8.62 \\  
 1.8 & 60803 & 8.38e-03 & 17.72 & 58314 & 8.38e-03 & 16.58     & 37337 & 8.38e-03 & 10.21 & 37444 & 8.38e-03 & 10.35 \\  
 2.0 & 63325 & 8.38e-03 & 17.79 & 61291 & 8.38e-03 & 16.13     & 46509 & 8.38e-03 & 12.60 & 45449 & 8.38e-03 & 12.51 \\ 
\hline
\end{tabular}} 
\end{table}

\subsection{Comparisons under different rank and noise level}\label{sec:rank_noise}

For the same ORL dataset as in the previous subsection, we fix $\lambda=1$ and vary the factorization rank $r\in\{5,50,150\}$ and the noise level $t\in\{0.001,0.01,0.1\}$. For each setting, we run both SymHALS and A-NAUM (with $\alpha=0.6$) for $10$ seconds and plot the relative objective value (smaller is better) against running time in Figure~\ref{fig:rank_noise}. The reported curves are averaged over five independent runs. Figure~\ref{fig:rank_noise} shows that A-NAUM exhibits a consistently better time-to-accuracy profile than SymHALS across all nine settings. Specifically, within the same 10-second budget, it attains a final relative objective value that is always comparable to, and often smaller than, that of SymHALS. Such an advantage is modest in the easiest regimes, such as $r=5$, where the two methods behave similarly after the initial decay. As the problem becomes more challenging with larger $r$, the gap becomes more pronounced. In particular, for $r=50$ and $r=150$, A-NAUM consistently achieves lower relative objective values. Overall, these results indicate that A-NAUM outperforms SymHALS across different factorization ranks and noise levels.

\begin{figure}[htb!] 
\centering 
\subfigure{\includegraphics[width=0.32\textwidth]{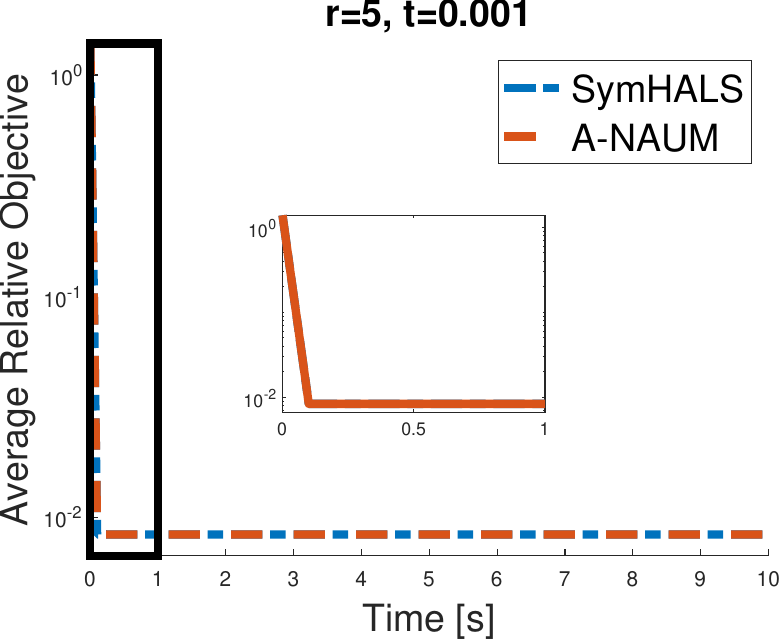} } \hfill 
\subfigure{\includegraphics[width=0.32\textwidth]{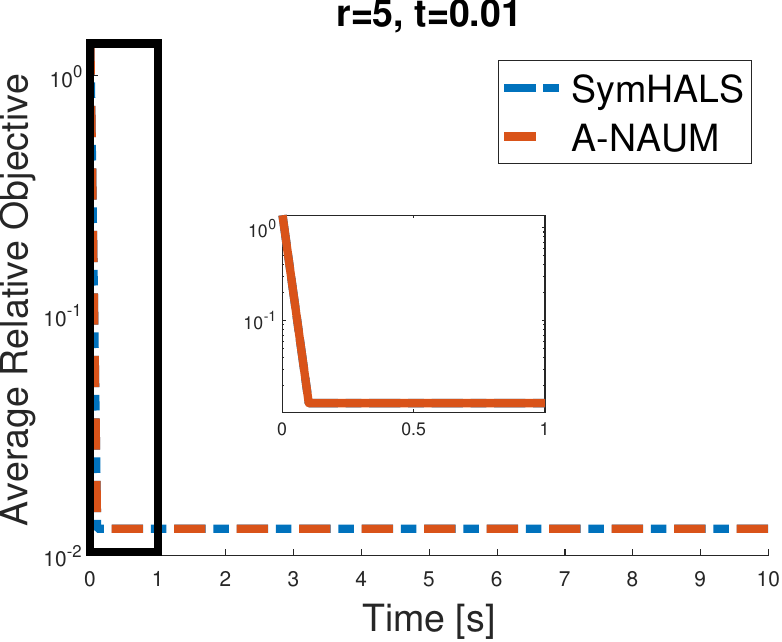} } 
\hfill 
\subfigure{\includegraphics[width=0.32\textwidth]{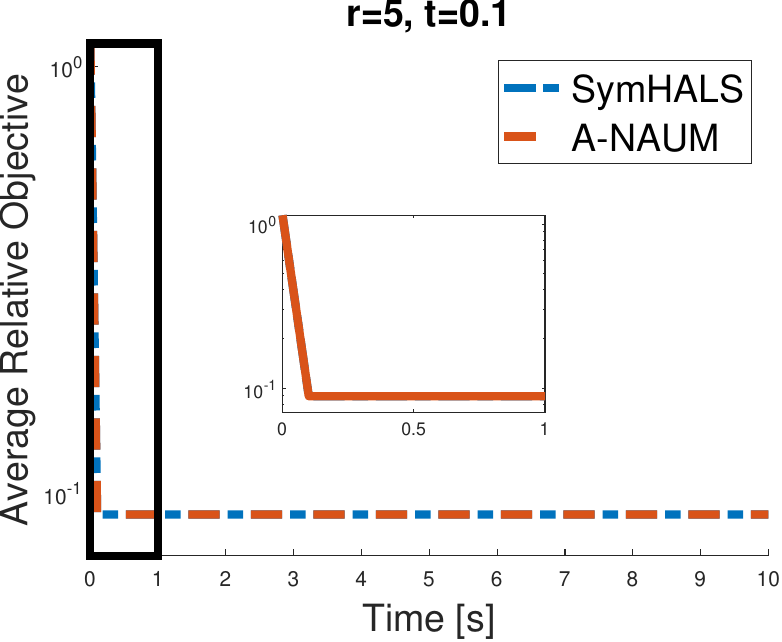} }\\ 
\subfigure{\includegraphics[width=0.32\textwidth]{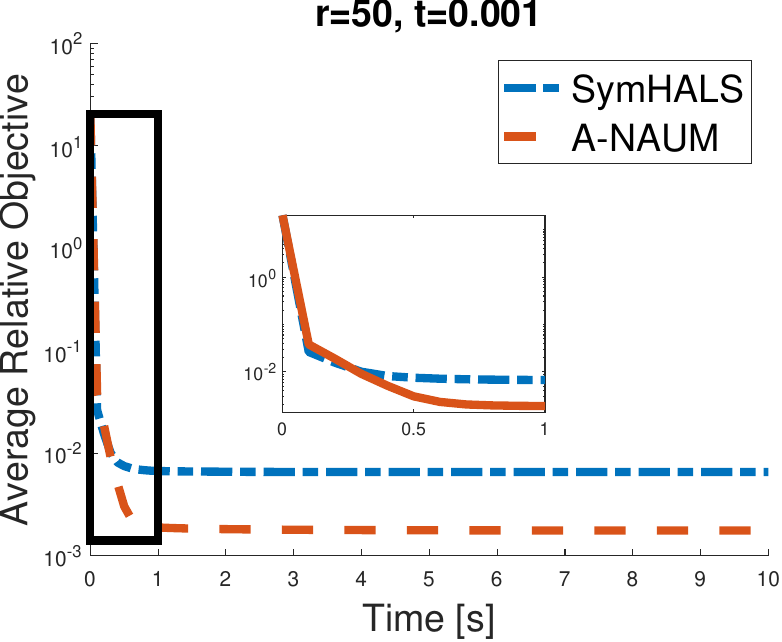}} \hfill 
\subfigure{\includegraphics[width=0.32\textwidth]{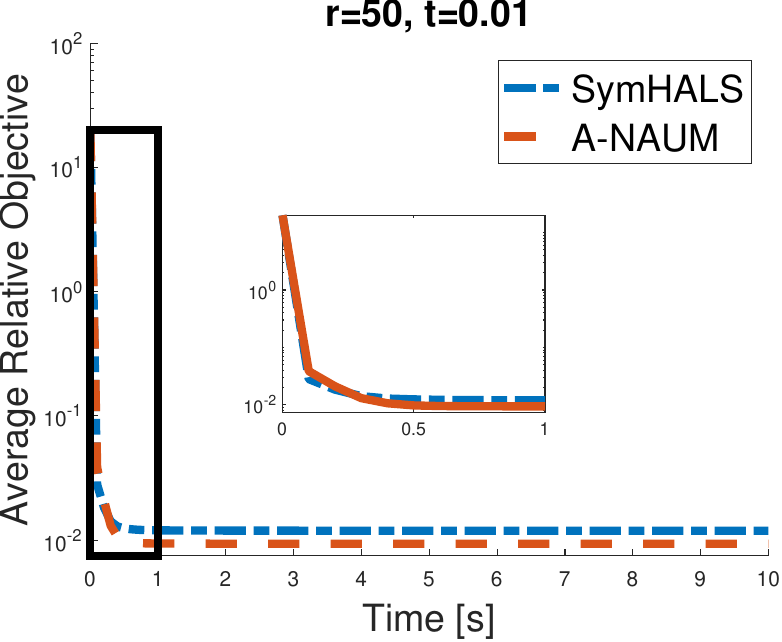} }
\hfill 
\subfigure{\includegraphics[width=0.32\textwidth]{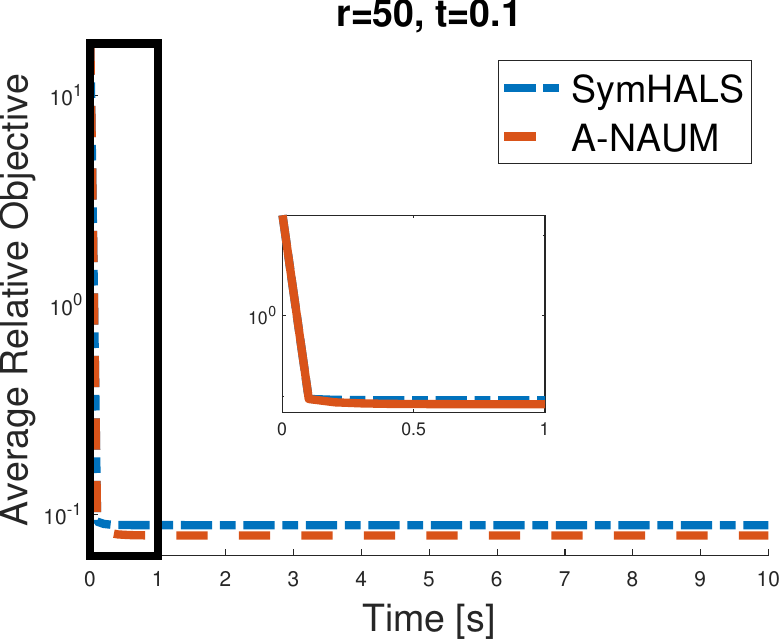} }\\ 
\subfigure{\includegraphics[width=0.32\textwidth]{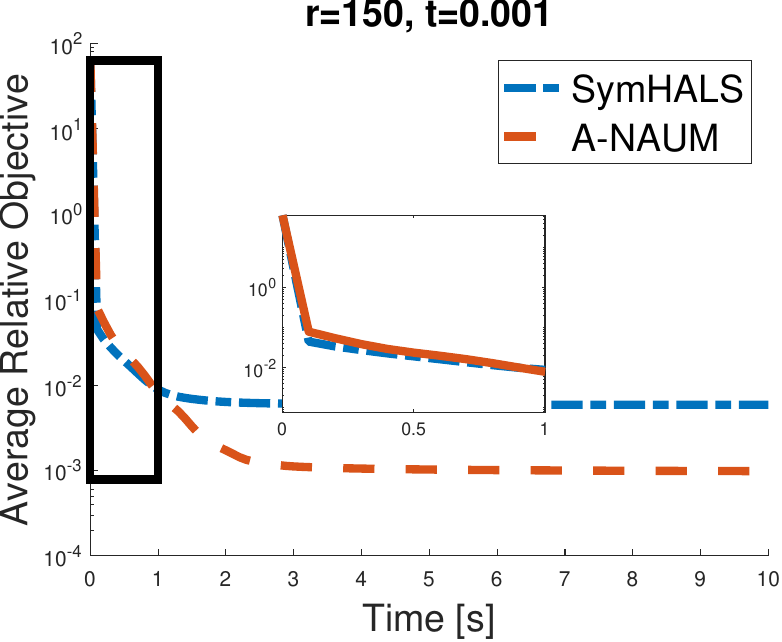}} \hfill 
\subfigure{\includegraphics[width=0.32\textwidth]{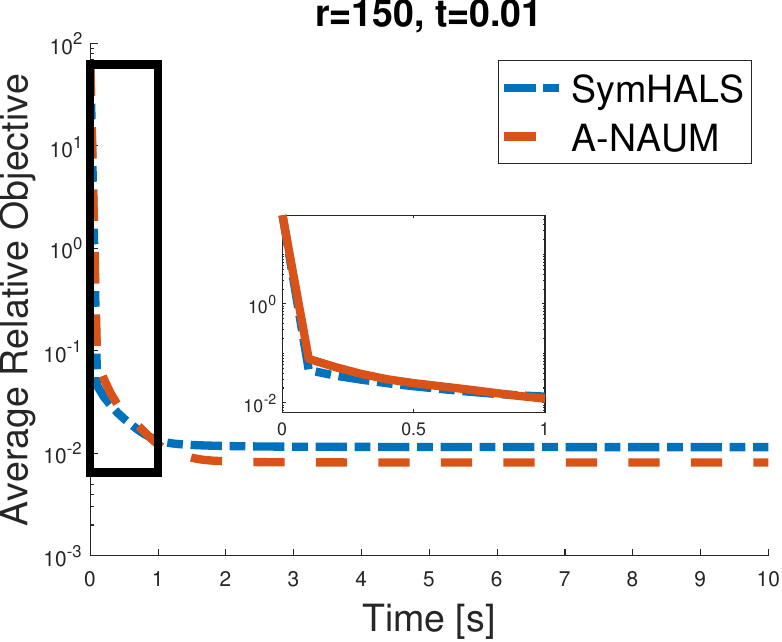} }
\hfill 
\subfigure{\includegraphics[width=0.32\textwidth]{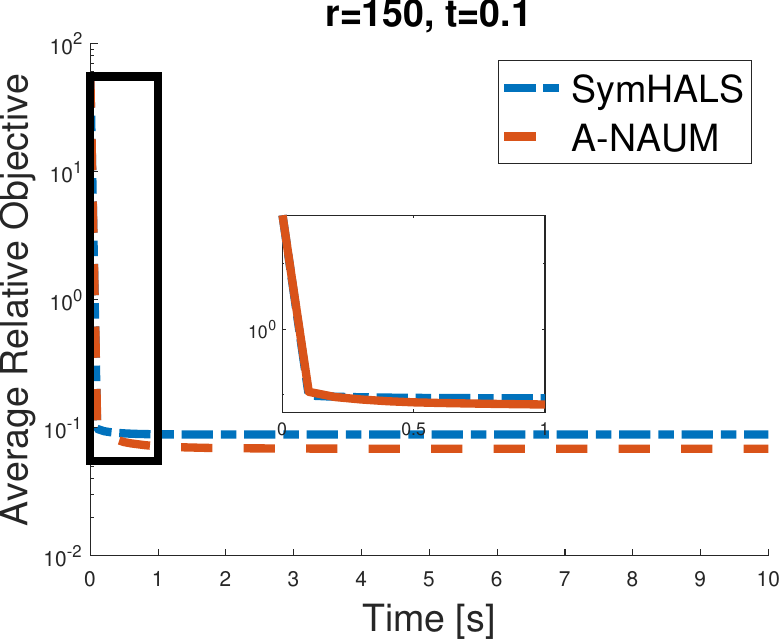} }
\caption{Average relative objective values versus running time (over 5 independent runs) for different factorization ranks $r\in\{5,50,150\}$ and noise levels $t\in\{0.001,0.01,0.1\}$ on ORL.} 
\label{fig:rank_noise} 
\end{figure}

\subsection{Comparisons under different $\lambda$}\label{sec:lambda_compare}

We further evaluate the performance of SymHALS and A-NAUM on ORL and CBCL for different penalty parameters $\lambda\in \{0.01,1,100\}$. In all cases, we fix the rank to $r=50$, the noise level to $t=0.001$, run both methods for $n/20$ seconds. The relative objective value versus running time is plotted in Figure~\ref{fig:lambda}. To assess the degree of symmetry attained by each method, the final quantity $\|X-Y\|_F^2$ is also reported in the legend of each plot. All results are averaged over 5 independent runs. 

Figure~\ref{fig:lambda} shows that A-NAUM consistently performs better than SymHALS. It exhibits a similar decay behavior but typically attains a final relative objective value that is comparable to, and often significantly smaller than, that of SymHALS. From the legends, which report the symmetry measure $\|X-Y\|_F^2$, we can also assess how close the two factors are at termination; smaller values indicate a more symmetric solution, i.e., $X\approx Y$. As expected, increasing $\lambda$ leads to noticeably more symmetric outputs for both methods. Specifically, when $\lambda=0.01$, the symmetry gaps are relatively larger, whereas for $\lambda=1$, and especially $\lambda=100$, they decrease by several orders of magnitude, indicating that the quadratic penalty term strongly enforces $X\approx Y$. Interestingly, A-NAUM achieves comparable and often smaller symmetry gaps than SymHALS for $\lambda\in\{0.01,100\}$, while SymHALS may attain smaller symmetry gaps for $\lambda=1$ despite yielding higher objective values. This observation reflects the intrinsic trade-off between data fidelity and constraint enforcement across different methods.

\begin{figure}[htb!] 
\centering 
\subfigure[ORL, $n=400$, $\lambda=0.01$]{\includegraphics[width=0.32\textwidth]{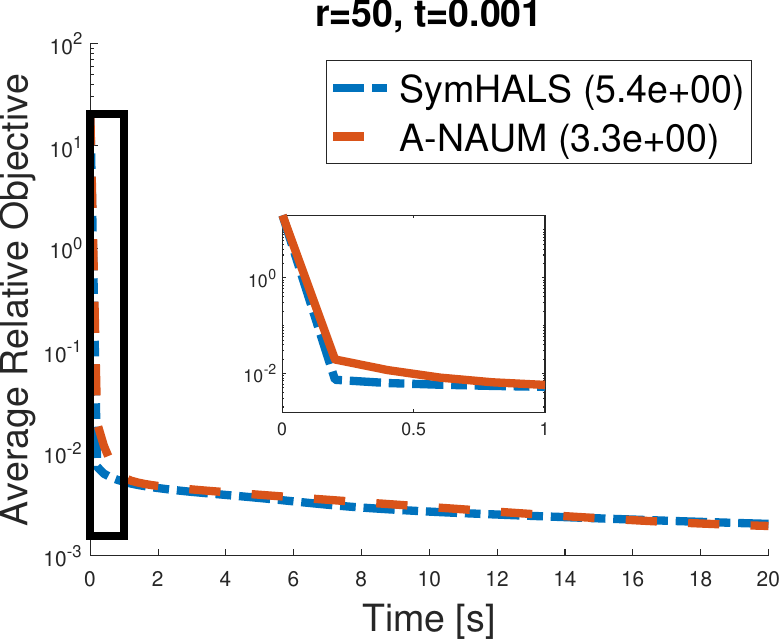}} \hfill 
\subfigure[ORL, $n=400$, $\lambda=1$]{\includegraphics[width=0.32\textwidth]{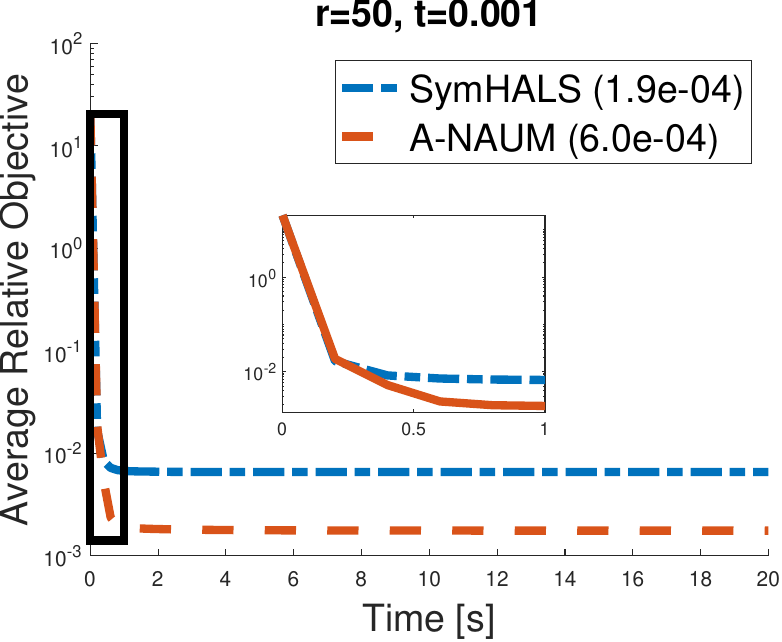}} \hfill 
\subfigure[ORL, $n=400$, $\lambda=100$]{\includegraphics[width=0.32\textwidth]{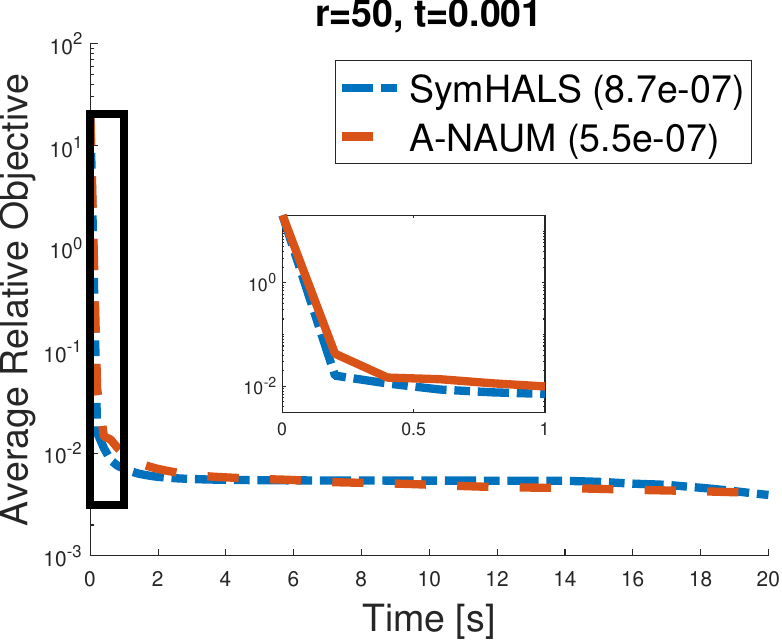}} \\ 
\subfigure[CBCL, $n=2429$, $\lambda=0.01$]{\includegraphics[width=0.32\textwidth]{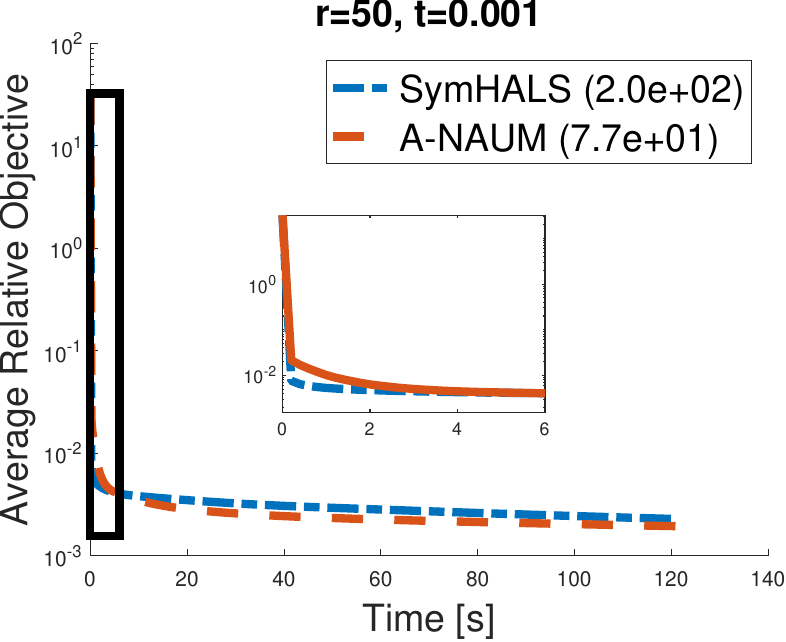}} \hfill 
\subfigure[CBCL, $n=2429$, $\lambda=1$]{\includegraphics[width=0.32\textwidth]{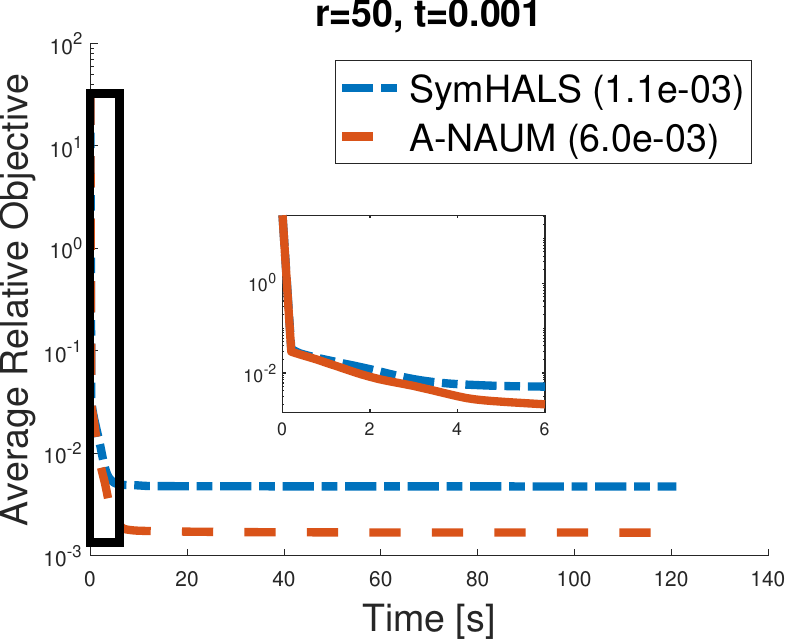}} \hfill 
\subfigure[CBCL, $n=2429$, $\lambda=100$]{\includegraphics[width=0.32\textwidth]{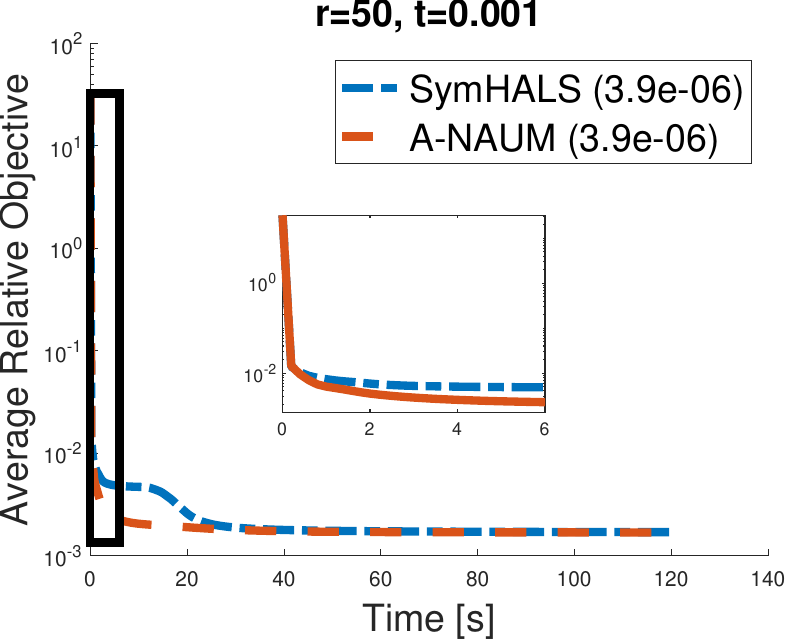}} \\ 
\caption{Average relative objective values versus running time (over 5 independent runs) for different penalty parameters $\lambda\in\{0.01,1,100\}$ on ORL and CBCL.}\label{fig:lambda} 
\end{figure}

\section{Conclusions}\label{sec-conclu}

We studied a generalized symmetric matrix factorization model that unifies a broad class of regularized matrix factorization formulations. On the modeling side, we established an exact penalty property showing that the symmetry-inducing quadratic penalty is exact for all sufficiently large but finite penalty parameters, and developed an exact relaxation framework that rigorously links stationary points of the original objective to those of a relaxed potential function. On the algorithmic side, we proposed A-NAUM, an average-type nonmonotone alternating updating method that leverages the decoupling structure to yield tractable subproblems and flexible update schemes. We proved the well-definedness of the nonmonotone line search under mild conditions, established global convergence of the entire sequence, and derived convergence-rate characterizations under the Kurdyka–{\L}ojasiewicz property and its associated exponent. Numerical experiments demonstrated the efficiency of the proposed A-NAUM.

\appendix

\section{A technical lemma} 

\begin{lemma}\label{lem:tr-inequa}
Let $A \in \mathbb{R}^{n \times n}$ be a symmetric matrix and $B \in \mathbb{R}^{n \times n}$ be a skew-symmetric matrix, i.e., $A^\top = A$ and $B^\top = -B$. The following results hold.
\begin{itemize}
\item [{\rm(i)}]$\langle A, B \rangle = \mathrm{tr}(AB) = 0$.

\item [{\rm(ii)}] (\cite[Lemma 1]{wkh1986trace}) Let $S \in \mathbb{R}^{n \times n}$ be a symmetric positive semidefinite matrix. It holds that $\rho_{\min}(A)\mathrm{tr}(S) \leq \mathrm{tr}(AS) \leq \rho_{\max}(A)\mathrm{tr}(S)$,
where $\rho_{\min}(A)$ and $\rho_{\max}(A)$ denote the smallest and largest eigenvalues of $A$, respectively.
\end{itemize}
\end{lemma}

\section{Missing Proofs in Section \ref{sec-exresult}}

\subsection{Proof of Theorem \ref{thm:2equal1}}\label{proof:2equal1}

\begin{proof}
    Since $(\widetilde{X}, \widetilde{Y})$ is a stationary point of problem \eqref{quadpenaltypro}, we have $0\in\partial\mathcal{F}_\lambda(\widetilde{X}, \widetilde{Y})$, 
which implies that there exist matrices $G \in \partial \Phi(\widetilde{X})$ and $H \in \partial \Phi(\widetilde{Y})$ such that
\begin{equation*}
\left\{\begin{aligned}
&0=G+(\mathcal{A}^*\mathcal{A}(\widetilde{X}\widetilde{Y}^\top) - \mathcal{A}^*\bm{b})\widetilde{Y} + \lambda(\widetilde{X} - \widetilde{Y}),  \\[3pt]
&0=H+(\mathcal{A}^*\mathcal{A}(\widetilde{X}\widetilde{Y}^\top) - \mathcal{A}^*\bm{b})^\top \widetilde{X} - \lambda(\widetilde{X} - \widetilde{Y}).
\end{aligned}\right.
\end{equation*}
Subtracting these two equations and using the symmetry of $\mathcal{A}^*\bm{b}$ (condition (i)) yields that
\begin{equation}\label{eq:equality}
\begin{aligned}
(2\lambda I +\mathcal{A}^*\bm{b})(\widetilde{X} - \widetilde{Y})
&=(\mathcal{A}^*\mathcal{A}(\widetilde{X}\widetilde{Y}^\top))^\top \widetilde{X}- \mathcal{A}^*\mathcal{A}(\widetilde{X}\widetilde{Y}^\top)\widetilde{Y} -(G - H).
\end{aligned}
\end{equation}
Taking the inner product with $\widetilde{X} - \widetilde{Y}$ on both sides of \eqref{eq:equality}, we have
\begin{equation}\label{eq:innerconduct}
 \langle 2\lambda I +\mathcal{A}^*\bm{b}, \,(\widetilde{X} - \widetilde{Y} )(\widetilde{X} - \widetilde{Y} )^\top\rangle 
=\langle(\mathcal{A}^*\mathcal{A}(\widetilde{X}\widetilde{Y}^\top))^\top \widetilde{X}- \mathcal{A}^*\mathcal{A}(\widetilde{X}\widetilde{Y}^\top)\widetilde{Y}-(G-H), \,\widetilde{X} - \widetilde{Y}\rangle.
\end{equation}

We next derive the lower bound for the left-hand side of \eqref{eq:innerconduct} and the upper bound for the right-hand side of \eqref{eq:innerconduct}, respectively. On the one hand, it follows from Lemma \ref{lem:tr-inequa}(ii) that
\begin{equation}\label{eq:leftbound}
\langle 2\lambda I+\mathcal{A}^*\bm{b}, \,(\widetilde{X} - \widetilde{Y} )(\widetilde{X} - \widetilde{Y} )^\top\rangle
\geq (2\lambda+\rho_{\min}(\mathcal{A}^*\bm{b}))\|\widetilde{X}-\widetilde{Y}\|_{F}^{2}.
\end{equation}
On the other hand, it follows from the $\kappa$-weak convexity of the proper closed function $\Phi$ that $\langle G-H, \,\widetilde{X} - \widetilde{Y}\rangle \geq -\kappa\|\widetilde{X}-\widetilde{Y}\|_{F}^{2}$, which further implies
\begin{equation}\label{eq:subdiffer}
\begin{aligned}
&\quad \langle(\mathcal{A}^*\mathcal{A}(\widetilde{X}\widetilde{Y}^\top))^\top \widetilde{X}- \mathcal{A}^*\mathcal{A}(\widetilde{X}\widetilde{Y}^\top)\widetilde{Y}-(G-H), \,\widetilde{X} - \widetilde{Y}\rangle   \\
&\leq\langle(\mathcal{A}^*\mathcal{A}(\widetilde{X}\widetilde{Y}^\top))^\top \widetilde{X}- \mathcal{A}^*\mathcal{A}(\widetilde{X}\widetilde{Y}^\top)\widetilde{Y}, \,\widetilde{X} - \widetilde{Y}\rangle + \kappa\|\widetilde{X}-\widetilde{Y}\|_{F}^{2}.
\end{aligned}
\end{equation}
For simplicity, we define a symmetric matrix $S$ and a skew-symmetric matrix $T$ respectively as
\begin{equation*}
S := \frac{\mathcal{A}^{*}\mathcal{A}(\widetilde{X}\widetilde{Y}^{\top}) + (\mathcal{A}^{*}\mathcal{A}(\widetilde{X}\widetilde{Y}^{\top}))^{\top}}{2},
\quad
T := \frac{\mathcal{A}^{*}\mathcal{A}(\widetilde{X}\widetilde{Y}^{\top}) - (\mathcal{A}^{*}\mathcal{A}(\widetilde{X}\widetilde{Y}^{\top}))^{\top}}{2},
\end{equation*}
which satisfy that $\mathcal{A}^*\mathcal{A}(\widetilde{X}\widetilde{Y}^\top) = S + T$ and $(\mathcal{A}^*\mathcal{A}(\widetilde{X}\widetilde{Y}^\top))^\top = S - T$.
By Lemma \ref{lem:tr-inequa}(i) and condition (ii), we further obtain
\begin{equation}\label{eq:Txy}
\begin{aligned}
&\langle T, \,\widetilde{X}\widetilde{X}^\top - \widetilde{Y}\widetilde{Y}^\top\rangle = 0, \\
&\langle T, \,\widetilde{X}\widetilde{Y}^\top - \widetilde{Y}\widetilde{X}^\top\rangle
=\langle {\textstyle\frac{\mathcal{A}^{*}\mathcal{A}(\widetilde{X}\widetilde{Y}^{\top} - \widetilde{Y}\widetilde{X}^{\top})}{2}}, \,\widetilde{X}\widetilde{Y}^\top - \widetilde{Y}\widetilde{X}^\top\rangle
={\frac{1}{2}} \| \mathcal{A}(\widetilde{X}\widetilde{Y}^\top - \widetilde{Y}\widetilde{X}^\top) \|_{F}^{2}.
\end{aligned}
\end{equation}
Using the above relations, one can check that
\begin{equation}\label{eq:rightbound1}
\begin{aligned}
&\quad \langle(\mathcal{A}^*\mathcal{A}(\widetilde{X}\widetilde{Y}^\top))^\top \widetilde{X}- \mathcal{A}^*\mathcal{A}(\widetilde{X}\widetilde{Y}^\top)\widetilde{Y},\,\widetilde{X} - \widetilde{Y}\rangle  \\
&=\langle(S - T)\widetilde{X} - (S + T)\widetilde{Y}, \,\widetilde{X} - \widetilde{Y}\rangle
=\langle S (\widetilde{X} - \widetilde{Y}),\,\widetilde{X} - \widetilde{Y}\rangle
-\langle T (\widetilde{X}+\widetilde{Y} ),\,\widetilde{X} - \widetilde{Y}\rangle \\
&=\langle S,\,(\widetilde{X} - \widetilde{Y})(\widetilde{X} - \widetilde{Y} )^\top\rangle
-\langle T, \,\widetilde{X}\widetilde{X}^\top - \widetilde{Y}\widetilde{Y}^\top\rangle
-\langle T, \,\widetilde{X}\widetilde{Y}^\top - \widetilde{Y}\widetilde{X}^\top\rangle\\
&=\langle S,\,(\widetilde{X} - \widetilde{Y})(\widetilde{X}- \widetilde{Y} )^\top\rangle
-{\frac{1}{2}} \| \mathcal{A}(\widetilde{X}\widetilde{Y}^\top - \widetilde{Y}\widetilde{X}^\top) \|_{F}^{2}  \qquad \mbox{(by \eqref{eq:Txy})}\\
&\leq \langle S,\,(\widetilde{X} - \widetilde{Y})(\widetilde{X}- \widetilde{Y} )^\top\rangle
\leq \rho_{\max}(S) \|  \widetilde{X} - \widetilde{Y} \|_{F}^{2}
\leq \| \mathcal{A}^*\mathcal{A}(\widetilde{X} \widetilde{Y}^\top) \|\|  \widetilde{X} -\widetilde{Y} \|_{F}^{2},
\end{aligned}
\end{equation}
where the second inequality follows from Lemma \ref{lem:tr-inequa}(ii) and the last inequality is due to
\begin{equation*}
\rho_{\max}(S)=\rho_{\max}\left(\textstyle\frac{\mathcal{A}^{*}\mathcal{A}(\widetilde{X}\widetilde{Y}^{\top}) + (\mathcal{A}^{*}\mathcal{A}(\widetilde{X}\widetilde{Y}^{\top}))^{\top}}{2}\right)\leq\| \mathcal{A}^*\mathcal{A}(\widetilde{X} \widetilde{Y}^\top) \|.
\end{equation*}
Now, combining \eqref{eq:innerconduct}, \eqref{eq:leftbound}, \eqref{eq:subdiffer}, \eqref{eq:rightbound1} and rearranging terms, we obtain
\begin{equation*}
\big(2\lambda + \rho_{\min}(\mathcal{A}^*\bm{b}) - \|\mathcal{A}^*\mathcal{A}(\widetilde{X}\widetilde{Y}^\top)\|-\kappa\big) \|\widetilde{X}-\widetilde{Y}\|_{F}^{2}
\leq 0.
\end{equation*}
Therefore, if $2\lambda>\| \mathcal{A}^*\mathcal{A}(\widetilde{X} \widetilde{Y}^\top) \|+\kappa-\rho_{\min}(\mathcal{A}^*\bm{b})$, it must hold that $\widetilde{X} = \widetilde{Y}$. Using this relation and condition (ii), it is straightforward to verify that $0\in \partial \Phi(\widetilde{X}) + \mathcal{A}^*(\mathcal{A}(\widetilde{X}\widetilde{X}^\top) -\bm{b})\widetilde{X}$, which implies that $\widetilde{X}$ is a stationary point of problem \eqref{SMFpro}. This completes the proof.
\end{proof}

\subsection{Proof of Proposition \ref{prop-samplingmap}} \label{proof:prop-samplingmap}

\begin{proof}
From the definition of $\mathcal{P}_\Omega$, its adjoint $\mathcal{P}_\Omega^*$ is given by
\[
\big[\mathcal{P}_\Omega^*(\bm{v})\big]_{ij}=
\begin{cases}
\bm{v}_{k}, & \text{if } (i,j)=(i_k,j_k)\in\Omega,\\
0, & \text{if } (i,j)\notin\Omega,
\end{cases}
\qquad \forall\,\bm{v}\in\mathbb{R}^{q}.
\]
Consequently, for any $U\in\mathbb{R}^{n\times n}$,
\[
\big[\mathcal{P}_\Omega^*\mathcal{P}_\Omega(U)\big]_{ij}=
\begin{cases}
U_{ij}, & (i,j)\in\Omega,\\
0, & (i,j)\notin\Omega.
\end{cases}
\]
Since $\Omega$ is symmetric, $(i,j)\in\Omega$ if and only if $(j,i)\in\Omega$, and hence
\[
\big[\mathcal{P}_\Omega^*\mathcal{P}_\Omega(U)\big]_{ji}=
\begin{cases}
U_{ji}, & (i,j)\in\Omega,\\
0, & (i,j)\notin\Omega.
\end{cases}
\]
Therefore, for any $(i,j)$,
\[
\begin{aligned}
\Big[\mathcal{P}_\Omega^*\mathcal{P}_\Omega(U)
-\big(\mathcal{P}_\Omega^*\mathcal{P}_\Omega(U)\big)^\top\Big]_{ij}
&=\big[\mathcal{P}_\Omega^*\mathcal{P}_\Omega(U)\big]_{ij}
-\big[\mathcal{P}_\Omega^*\mathcal{P}_\Omega(U)\big]_{ji}\\
&=
\begin{cases}
U_{ij}-U_{ji}, & (i,j)\in\Omega,\\
0, & (i,j)\notin\Omega,
\end{cases}
~=~\big[\mathcal{P}_\Omega^*\mathcal{P}_\Omega(U-U^\top)\big]_{ij}.
\end{aligned}
\]
This proves the desired identity.
\end{proof}

\subsection{Proof of Theorem \ref{thm:equal-min}}\label{proof:equal-min}

We begin with two auxiliary lemmas, whose verifications are straightforward and hence are omitted here for simplicity.

\begin{lemma}\label{lem:f=theta}
Suppose that $\mathcal{A}\mathcal{A}^* = \mathcal{I}_q$ and $\frac{1}{\alpha} + \frac{1}{\beta} = 1$. Then, for any $(X, \,Y, \,Z)$ satisfying
\begin{equation*}
Z = \textstyle\left(\mathcal{I} - \frac{\beta}{\alpha+\beta}\mathcal{A}^*\mathcal{A}\right)(XY^\top) + \frac{\beta}{\alpha+\beta}\mathcal{A}^*(\bm{b}),
\end{equation*}
we have $\mathcal{F}_{\lambda}(X, \,Y)=\Theta_{\alpha,\beta,\lambda}(X, \,Y, \,Z)$.
\end{lemma}

\begin{lemma}\label{prop:inverse}
Suppose that $\mathcal{A}\mathcal{A}^*=\mathcal{I}_{q}$ and $\alpha(\alpha + \beta)\neq0$. Then, $\alpha\mathcal{I}+\beta \mathcal{A}^*\mathcal{A}$ is invertible and its inverse is given by $\frac{1}{\alpha}\mathcal{I}-\frac{\beta}{\alpha(\alpha+\beta)} \mathcal{A}^*\mathcal{A}$.
\end{lemma}

\begin{proof}[Proof of Theorem \ref{thm:equal-min}]
First, it follows from $\alpha \mathcal{I} + \beta \mathcal{A}^* \mathcal{A} \succ 0$ that the function $Z \mapsto \Theta_{\alpha,\beta,\lambda}(X,\,Y,\,Z)$ is strongly convex. Thus, for any fixed $X$ and $Y$, the optimal solution $\widetilde{Z}_{XY}$ to problem $\min\limits_Z \{\Theta_{\alpha,\beta,\lambda}(X,\,Y,\,Z)\}$ uniquely exists, and can be obtained explicitly. Indeed, from the optimality condition, we have
\begin{equation*}
\alpha(\widetilde{Z}_{XY} - XY^\top) + \beta\mathcal{A}^* (\mathcal{A}(\widetilde{Z}_{XY}) - \bm{b}) = 0.
\end{equation*}
Since $\alpha \mathcal{I} + \beta \mathcal{A}^* \mathcal{A}$ is invertible (as $\alpha \mathcal{I} + \beta \mathcal{A}^* \mathcal{A} \succ 0$), it follows from Lemma \ref{prop:inverse} that
\begin{align*}
\widetilde{Z}_{XY}
&\textstyle=(\alpha \mathcal{I} + \beta \mathcal{A}^*\mathcal{A})^{-1}\left(\alpha XY^\top + \beta \mathcal{A}^*(\bm{b})\right)
=\left( \frac{1}{\alpha} \mathcal{I} - \frac{\beta}{\alpha(\alpha + \beta)} \mathcal{A}^*\mathcal{A}\right)\left( \alpha XY^\top + \beta \mathcal{A}^*(\bm{b})\right) \\
&\textstyle=\left( \mathcal{I} - \frac{\beta}{\alpha + \beta} \mathcal{A}^*\mathcal{A} \right) (XY^\top) +\left( \frac{\beta}{\alpha} \mathcal{A}^*(\bm{b}) - \frac{\beta^2}{\alpha(\alpha + \beta)} \mathcal{A}^*(\mathcal{A}\mathcal{A}^*)(\bm{b})\right) \\
&\textstyle=\left( \mathcal{I} - \frac{\beta}{\alpha + \beta} \mathcal{A}^*\mathcal{A} \right) (XY^\top) + \frac{\beta}{\alpha + \beta} \mathcal{A}^*(\bm{b}),
\end{align*}
where the last equality follows from $\mathcal{A}\mathcal{A}^*=\mathcal{I}_q$. This, together with Lemma \ref{lem:f=theta}, implies that $\mathcal{F}_{\lambda}(X,\,Y) = \Theta_{\alpha,\beta,\lambda}(X,\,Y,\,\widetilde{Z}_{XY})$. Then, we have that
\begin{equation*}
\begin{aligned}
\min_{X,\,Y,\,Z}\big\{\Theta_{\alpha,\beta,\lambda}(X,\,Y,\,Z)\big\}
&=\min_{X,\,Y}\left\{ \min_Z\left\{\Theta_{\alpha,\beta,\lambda}(X,\,Y,\,Z)\right\} \right\}
=\min_{X,\,Y}\left\{\Theta_{\alpha,\beta,\lambda}(X,\,Y,\,\widetilde{Z}_{XY}) \right\}  \\
&=\min_{X,\,Y}\big\{\mathcal{F}_{\lambda}(X,\,Y)\big\}.
\end{aligned}
\end{equation*}
This completes the proof.
\end{proof}

\subsection{Proof of Theorem \ref{thm:equal-sta}}\label{proof:equal-sta}

\begin{proof}
\textit{Statement (i).} If $(\widetilde{X},\,\widetilde{Y},\,\widetilde{Z})$  is a stationary point of $\Theta_{\alpha,\beta,\lambda}$, then we have that $0\in\partial\Theta_{\alpha,\beta,\lambda}(\widetilde{X}, \,\widetilde{Y}, \,\widetilde{Z})$, i.e.,
\begin{subnumcases}{}\label{eq:theta-opti}
0 \in \partial\Psi(\widetilde{X}) + \alpha(\widetilde{X}\widetilde{Y}^\top - \widetilde{Z})\widetilde{Y}+\lambda(\widetilde{X}-\widetilde{Y}), \label{eq:thetaopta} \\
0 \in \partial\Phi(\widetilde{Y}) + \alpha(\widetilde{X}\widetilde{Y}^\top - \widetilde{Z})^\top\widetilde{X}-\lambda(\widetilde{X}-\widetilde{Y}), \label{eq:thetaoptb} \\
0 = \alpha(\widetilde{Z} -\widetilde{X}\widetilde{Y}^\top) + \beta\mathcal{A}^*(\mathcal{A}(\widetilde{Z}) - \bm{b}). \label{eq:thetaoptc}
\end{subnumcases}
Since $\frac{1}{\alpha}+\frac{1}{\beta}=1$, we have $\alpha(\alpha+\beta)\neq0$ and hence $\alpha\mathcal{I}+\beta\mathcal{A}^*\mathcal{A}$ is invertible from Lemma \ref{prop:inverse}. Then, using the same arguments in the proof of Theorem \ref{thm:equal-min}, we see from \eqref{eq:thetaoptc} that $(\widetilde{X}, \widetilde{Y}, \widetilde{Z})$ satisfies \eqref{eq:Zbar=}. Moreover, using \eqref{eq:Zbar=} with proper caculations, we further have
\begin{eqnarray}
\widetilde{X}\widetilde{Y}^\top - \widetilde{Z}
\!\!&=&\!\! {\textstyle\frac{\beta}{\alpha+\beta}}\mathcal{A}^*(\mathcal{A}(\widetilde{X}\widetilde{Y}^\top) - \bm{b}),  \label{barXY-Z} \\
\mathcal{A}(\widetilde{Z}) - \bm{b}
\!\!&=&\!\! {\textstyle\frac{\alpha}{\alpha+\beta}}(\mathcal{A}(\widetilde{X}\widetilde{Y}^\top) - \bm{b}).  \label{barAZ-b}
\end{eqnarray}
Thus, substituting \eqref{barXY-Z} into \eqref{eq:thetaopta} and \eqref{eq:thetaoptb}, we see that
\begin{equation}\label{ffirstorderopt2}
\left\{
\begin{aligned}
0 &\in \partial \Psi(\widetilde{X}) + {\textstyle\frac{\alpha\beta}{\alpha+\beta}}\mathcal{A}^*(\mathcal{A}(\widetilde{X}\widetilde{Y}^\top) - \bm{b})\widetilde{Y} + \lambda(\widetilde{X} - \widetilde{Y}), \\
0 &\in \partial \Phi(\widetilde{Y}) + {\textstyle\frac{\alpha\beta}{\alpha+\beta}}(\mathcal{A}^*(\mathcal{A}(\widetilde{X}\widetilde{Y}^\top) - \bm{b}))^\top \widetilde{X} - \lambda(\widetilde{X} - \widetilde{Y}),
\end{aligned}
\right.
\end{equation}
which, together with $\frac{1}{\alpha}+\frac{1}{\beta}=1$ and hence $\frac{\alpha\beta}{\alpha+\beta}=1$, implies that $ (\widetilde{X}, \,\widetilde{Y}) $ is a stationary point of $ \mathcal{F}_{\lambda}$. This proves statement (i).

\textit{Statement (ii).} Since $(\widetilde{X}, \,\widetilde{Y})$ is a stationary point of $\mathcal{F}_{\lambda}$, it follows from $0\in\partial\mathcal{F}_\lambda(\widetilde{X}, \widetilde{Y})$ and $\frac{1}{\alpha}+\frac{1}{\beta}=1$ that \eqref{ffirstorderopt2} holds. By the definition of $\widetilde{Z}$ in \eqref{eq:Zbar=}, we further obtain \eqref{barXY-Z} and \eqref{barAZ-b}. Substituting \eqref{barXY-Z} into \eqref{ffirstorderopt2} yields \eqref{eq:thetaopta} and \eqref{eq:thetaoptb}. Moreover, combining \eqref{barXY-Z} and \eqref{barAZ-b} gives
$\alpha(\widetilde{Z} -\widetilde{X}\widetilde{Y}^\top) + \beta\mathcal{A}^*(\mathcal{A}(\widetilde{Z})-\bm{b})=0$.
Altogether, these relations imply that $(\widetilde{X}, \,\widetilde{Y}, \,\widetilde{Z})$ is a stationary point of $\Theta_{\alpha,\beta,\lambda}$. This proves statement (ii).
\end{proof}

\section{Missing Proofs in Section \ref{sec-convanal}}

\subsection{Proof of Lemma \ref{lem:Flambda-suffidescent}}\label{proof:lem:Flambda-suffidescent}

\begin{proof}
For any $(U, \,V)$, let
\begin{equation}\label{eq:W-Zk+1}
W := \left( \mathcal{I} - {\textstyle\frac{\beta}{\alpha+\beta}}\mathcal{A}^* \mathcal{A} \right) (UV^\top)
+ {\textstyle\frac{\beta}{\alpha+\beta}} \mathcal{A}^*(\bm{b}).
\end{equation}
It then follows from \eqref{Zkupdate} and Lemma \ref{lem:f=theta} that
\begin{equation*}
\mathcal{F}_{\lambda}(X^k, \,Y^k) = \Theta_{\alpha,\beta,\lambda}(X^k, \,Y^k, \,Z^k),
\quad
\mathcal{F}_{\lambda}(U, \,V) = \Theta_{\alpha,\beta,\lambda}(U, \,V, \,W).
\end{equation*}
Thus, to establish \eqref{eq:Flambdadescent}, we only need to consider $\Theta_{\alpha,\beta,\lambda}(U, \,V, \,W)-\Theta_{\alpha,\beta,\lambda}(X^k, \,Y^k, \,Z^k)$.

We start by noting that
\begin{equation*}
\textstyle\mathcal{A}^* \mathcal{A}(W)
=\left( \mathcal{A}^* \mathcal{A} - \frac{\beta}{\alpha+\beta}\mathcal{A}^*(\mathcal{A}\mathcal{A}^*)\mathcal{A} \right) (UV^\top)
+ \frac{\beta}{\alpha+\beta} \mathcal{A}^* (\mathcal{A} \mathcal{A}^*)(\bm{b}) 
= \frac{\alpha}{\alpha+\beta} \mathcal{A}^* \mathcal{A} (UV^\top)
+ \frac{\beta}{\alpha+\beta} \mathcal{A}^*(\bm{b}),
\end{equation*}
where the last equality follows from $\mathcal{A}\mathcal{A}^*=\mathcal{I}_q$ (Assumption \ref{assumA}(ii)). Then, it is straightforward to verify that
\begin{equation*}
\nabla_Z\Theta_{\alpha,\beta,\lambda}(U,\,V,\,W)
= \alpha (W - UV^\top) + \beta \mathcal{A}^* \mathcal{A}(W) - \beta \mathcal{A}^*(\bm{b}) = 0.
\end{equation*}
Moreover, since $\gamma$ is chosen such that $(\alpha + \gamma)\mathcal{I} + \beta \mathcal{A}^*\mathcal{A} \succeq 0$ (see \eqref{eq:definegamma}), we see that the function $Z \mapsto \Theta_{\alpha,\beta,\lambda}(U,\,V,\,Z) + \frac{\gamma}{2}\|Z - Z^k\|_F^2$ is convex and hence
\begin{equation*}
\begin{aligned}
&\Theta_{\alpha,\beta,\lambda}(U,\,V,\,Z^k) + \underbrace{ \frac{\gamma}{2}\|Z^k - Z^k\|_F^2}_{=\,0}\\
\geq &~\Theta_{\alpha,\beta,\lambda}(U,\,V,\,W)+  \frac{\gamma}{2}\|W - Z^k\|_F^2
+ \langle \underbrace{\nabla_Z \Theta_{\alpha,\beta,\lambda}(U,\,V,\,W)}_{=\,0} + \gamma(W - Z^k), \,Z^k - W \rangle,
\end{aligned}
\end{equation*}
which implies that
\begin{equation}\label{eq:thetadiffer}
\Theta_{\alpha,\beta,\lambda}(U,\,V,\,W)
-\Theta_{\alpha,\beta,\lambda}(U,\,V,\,Z^k)
\leq\frac{\gamma}{2}\|W-Z^k\|_F^2.
\end{equation}
Then, substituting \eqref{Zkupdate} and \eqref{eq:W-Zk+1} into \eqref{eq:thetadiffer}, we obtain
\begin{equation}\label{eq:thetadiffer1}
\begin{aligned}
&~~\Theta_{\alpha,\beta,\lambda}(U,\,V,\,W)-\Theta_{\alpha,\beta,\lambda}(U,\,V,\,Z^k) \\
&\textstyle \leq\frac{\gamma}{2}\left\|\left(\mathcal{I}-{\textstyle\frac{\beta}{\alpha+\beta}}\mathcal{A}^*\mathcal{A}\right)(U V^\top\!-\!X^k(Y^k)^\top)\right\|_F^2
\leq\frac{\gamma}{2}\left\|\mathcal{I}-{\textstyle\frac{\beta}{\alpha+\beta}}\mathcal{A}^*\mathcal{A}\right\|^2\cdot\left\|UV^\top\!-\!X^k(Y^k)^\top\right\|_F^2 \\
&\textstyle = \frac{\gamma\rho}{2}\big\|U(V\!-\!Y^k)^\top+(U\!-\!X^k)(Y^k)^\top\big\|_F^2
\leq\frac{\gamma\rho}{2}\left(\big\|U(V\!-\!Y^k)^\top\big\|_F+\big\|(U\!-\!X^k)(Y^k)^\top\big\|_F\right)^2\\
&\textstyle \stackrel{\rm(i)}{\leq} \frac{\gamma\rho}{2}\left(\|U\|\|V\!-\!Y^k\|_F
+\|Y^k\|\|U\!-\!X^k\|_F\right)^2
\stackrel{\rm(ii)}{\leq}\gamma\rho\left(\|U\|^2\|V\!-\!Y^k\|_F^2+\|Y^k\|^2\|U\!-\!X^k\|_F^2\right),
\end{aligned}
\end{equation}
where the equality follows from the definition of $\rho$ in \eqref{eq:definerho}, (i) follows from the relation $\|A B\|_{F}\leq\|A\|\|B\|_{F}$ and (ii) follows from the relation $(a+b)^{2}\leq 2a^{2}+2b^{2}$.

Next, we claim that
\begin{eqnarray}
\Theta_{\alpha,\beta,\lambda}(U, \,V, \,Z^k)
- \Theta_{\alpha,\beta,\lambda}(U, \,Y^k, \,Z^k)
\!\!&\leq&\!\! \frac{\alpha\|U\|^2-\sigma_k}{2}\,\|V-Y^k\|_F^2, \label{eq:thetadiffer2} \\[3pt]
\Theta_{\alpha,\beta,\lambda}(U, \,Y^k, \,Z^k)
- \Theta_{\alpha,\beta,\lambda}(X^k, \,Y^k, \,Z^k )
\!\!&\leq&\!\! \frac{\alpha\|Y^k\|^2-\mu_k}{2}\,\|U-X^k\|_F^2.  \label{eq:thetadiffer3}
\end{eqnarray}
Below, we will only prove \eqref{eq:thetadiffer2}, since the proof for \eqref{eq:thetadiffer3} can be done in a similar way. To this end, we consider the following three cases.
\begin{itemize}[leftmargin=0.5cm]
\item \textbf{Proximal:} In this case, we have
\begin{align*}
&\quad \Theta_{\alpha,\beta,\lambda}(U,\,V,\,Z^k)
-\Theta_{\alpha,\beta,\lambda}(U,\,Y^k,\,Z^k)  \\
&=\Phi(V) + \mathcal{H}_\alpha(U,\,V,\,Z^k) + \frac{\lambda}{2}\|U-V\|_F^2
-\Phi(Y^k) - \mathcal{H}_\alpha(U,\,Y^k,\,Z^k) - \frac{\lambda}{2}\|U-Y^k\|_F^2 \\
&=-\frac{\sigma_k}{2}\|V-Y^k\|_F^2
+ \left[\Phi(V)+\mathcal{H}_\alpha(U,\,V,\,Z^k)+\frac{\lambda}{2}\|U-V\|_F^2
+\frac{\sigma_k}{2}\|V-Y^k\|_F^2 \right] \\
&\qquad -\left[\Phi(Y^k) + \mathcal{H}_\alpha(U,\,Y^k,\,Z^k )
+\frac{\lambda}{2}\|U-Y^k\|_F^2 + \frac{\sigma_k}{2}\|Y^k-Y^k\|_F^2\right] \\
&\leq -\frac{\sigma_k}{2}\|V-Y^k\|_F^2,
\end{align*}
where the inequality follows from the definition of $V$ as a minimizer of \eqref{supro:V-prox}.

\item \textbf{Prox-linear:} In this case, we have
\begin{align*}
&\quad \Theta_{\alpha,\beta,\lambda}(U,\,V,\,Z^k )-\Theta_{\alpha,\beta,\lambda} (U,\,Y^k,\,Z^k )\\
&=\Phi(V)+\mathcal{H}_\alpha(U,\,V,\,Z^k)+\frac{\lambda}{2}\|U-V\|_F^2
-\Phi(Y^k)-\mathcal{H}_\alpha(U,\,Y^k,\,Z^k)-\frac{\lambda}{2}\|U-Y^k\|_F^2 \\
&\leq \Phi(V)+\frac{\lambda}{2}\|U-V\|_F^2-\Phi(Y^k)-\frac{\lambda}{2}\|U-Y^k\|_F^2
+ \langle\nabla_Y\mathcal{H}_\alpha(U,\,Y^k,\,Z^k),\,V-Y^k \rangle  \\
&\qquad + \frac{\alpha\|U\|^2}{2}\|V-Y^k\|_F^2 \\
&=\left[\Phi(V)+ \langle\nabla_Y\mathcal{H}_\alpha(U,\,Y^k,\,Z^k),\,V-Y^k\rangle
+\frac{\lambda}{2}\|U-V\|_F^2+\frac{\sigma_k}{2}\|V-Y^k\|_F^2 \right]\\
&\qquad -\left[\Phi(Y^k)+\langle\nabla_Y\mathcal{H}_\alpha(U,\,Y^k,\,Z^k),\,Y^k-Y^k \rangle + \frac{\lambda}{2}\|U-Y^k\|_F^2+\frac{\sigma_k}{2}\|Y^k-Y^k\|_F^2  \right] \\
&\qquad +\frac{\alpha\|U\|^2-\sigma_k}{2}\|V-Y^k\|_F^2  \\
&\leq \frac{\alpha\|U\|^2-\sigma_k}{2}\|V-Y^k\|_F^2,
\end{align*}
where the first inequality holds because $\nabla_{Y} \mathcal{H}_{\alpha}(X, \cdot, Z)$ is Lipschitz continuous with modulus $\alpha \|X\|^2$ and 
the last inequality follows from the definition of $V$ as a minimizer of \eqref{supro:V-proxlinear}.

\item \textbf{Hierarchical-prox:} In this case, for any $1\leq i\leq r$, we have
\begin{align*}
&\quad \Theta_{\alpha,\beta,\lambda}(U,\,\bm{v}_{j<i},\,\bm{v}_i,\,\bm{y}_{j>i}^k,\,Z^k)
-\Theta_{\alpha,\beta,\lambda}(U,\,\bm{v}_{j<i},\,\bm{y}_i^k,\,\bm{y}_{j>i}^k,\,Z^k) \\
&= \phi_i(\bm{v}_i)+\mathcal{H}_\alpha(U,\,\bm{v}_{j<i},\,\bm{v}_i,\,\bm{y}_{j>i}^k,\,Z^k)
+\frac{\lambda}{2}\|\bm{u}_i-\bm{v}_i\|^2 \\
&\qquad  -\phi_i(\bm{y}_i^k)-\mathcal{H}_\alpha(U,\,\bm{v}_{j<i},\,\bm{y}_i^k,\,\bm{y}_{j>i}^k,Z^k)
-\frac{\lambda}{2}\|\bm{u}_i-\bm{y}_i^k\|^2 \\
&=-\frac{\sigma_k}{2}\|\bm{v}_i-\bm{y}_i^k\|^2+\left[\phi_i(\bm{v}_i)
+\mathcal{H}_\alpha(U,\,\bm{v}_{j<i},\,\bm{v}_i,\,\bm{y}_{j>i}^k,\,Z^k)
+\frac{\lambda}{2}\|\bm{u}_i-\bm{v}_i\|^2
+\frac{\sigma_k}{2}\|\bm{v}_i-\bm{y}_i^k\|^2 \right]\\
&\qquad -\left[\phi_i(\bm{y}_i^k)+\mathcal{H}_\alpha(U,\bm{v}_{j<i},\bm{y}_i^k,\bm{y}_{j>i}^k,Z^k)
+\frac{\lambda}{2}\|\bm{u}_i\!-\!\bm{y}_i^k\|^2
+\frac{\sigma_k}{2}\|\bm{y}_i^k\!-\!\bm{y}_i^k\|^2\right]  \\
&\leq -\frac{\sigma_k}{2}\|\bm{v}_i\!-\!\bm{y}_i^k\|^2,
\end{align*}
where the inequality follows from the definition of $\bm{v}_{i}$ as a minimizer of  \eqref{supro:V-hier}. Summing the above relation from $i=r$ to $i=1$ and simplifying the resulting inequality, we obtain \eqref{eq:thetadiffer2}.
\end{itemize}

The inequality \eqref{eq:thetadiffer3} can be obtained via a similar argument. Then, summing \eqref{eq:thetadiffer1}, \eqref{eq:thetadiffer2}, \eqref{eq:thetadiffer3}, and using $\mathcal{F}_\lambda(U,\,V)=\Theta_{\alpha,\beta,\lambda}(U,\,V,\,W)$ and $\mathcal{F}_\lambda(X^{k},\,Y^{k})=\Theta_{\alpha,\beta,\lambda}(X^{k},\,Y^{k},\,Z^{k})$, we obtain \eqref{eq:Flambdadescent}. This completes the proof.
\end{proof}

\subsection{Proof of Lemma \ref{lem:A-welldefined}}\label{proof:lem:A-welldefined}

\begin{proof}
We prove this lemma by contradiction. Assume that there exists some $k\geq0$ such that the line search criterion \eqref{eq:A-linesearch} cannot be satisfied after finitely many inner iterations. Then, for this fixed $k$, we first claim that
\begin{equation}\label{eq:omega>flambda}
\mathcal{R}_k\geq\mathcal{F}_\lambda(X^k, \,Y^k).
\end{equation}
Indeed, if $k=0$, \eqref{eq:omega>flambda} holds trivially since $\mathcal{R}_0=\mathcal{F}_\lambda(X^0, \,Y^0)$. If $k\geq1$, by the definition of ${\mathcal{R}_k}$ in \eqref{eq:omegadefinition}, we have
\begin{equation*}
\begin{aligned}
\mathcal{R}_k
&=(1-p_k)\mathcal{R}_{k-1} + p_k\mathcal{F}_{\lambda}(X^k,\,Y^k) \\
&\textstyle\geq (1 - p_k)\left(\mathcal{F}_\lambda(X^k, \,Y^k)
+ \frac{c}{2}\left(\|X^k-X^{k-1}\|_F^2 + \|Y^k-Y^{k-1}\|_F^2\right)\right)
+ p_k\mathcal{F}_{\lambda}(X^k,\,Y^k) \\
&\geq \mathcal{F}_\lambda(X^k, \,Y^k),
\end{aligned}
\end{equation*}
where the first inequality follows from the fact that the line search condition \eqref{eq:A-linesearch} is satisfied at the $(k-1)$th iteration.

Next, from Steps (2\ref{algo-AcomputeU}) and (2\ref{algo-Abacktrack}) of Algorithm \ref{algo-A-NAUM}, we know that $\mu_{k}\leq\mu_{k}^{\max}=(\alpha+2\gamma\rho) \|Y^{k} \|^{2}+c$ and hence $\mu_{k}=\mu_{k}^{\max}$ must be attained after finitely many inner iterations. Let $n_{k}$ denote the number of inner iterations at which $\mu_{k}=\mu_{k}^{\max}$ is reached for the \textit{first} time. If $\mu_{k}^{0}\geq\mu_{k}^{\max}$, then $n_{k}=1$; otherwise, 
\begin{equation*}
\mu^{\min}\tau^{n_{k}-2}\leq\mu_{k}^{0}\tau^{n_{k}-2}<\mu_{k}^{\max},
\end{equation*}
which implies that
\begin{equation}\label{eq:etak-bound}
n_k\leq\left\lfloor \frac{\log\left(\mu_k^{\max}\right)-\log\left(\mu^{\min}\right)}{\log\tau}+2\right\rfloor,
\end{equation}
where $\lfloor a\rfloor$ denotes the largest integer smaller than or equal to $a$.
Then, by Step (2\ref{algo-Abacktrack}) of Algorithm \ref{algo-A-NAUM}, we have $U\equiv U_{\mu_{k}^{\max}}$ and $\sigma_{k}^{\max}=(\alpha+2\gamma\rho) \|U_{\mu_{k}^{\max}} \|^{2}+c$ after at most $n_{k}+1$ inner iterations, where $U_{\mu_{k}^{\max}}$ is computed by \eqref{supro:U-prox}, \eqref{supro:U-proxlinear} or \eqref{supro:U-hier} with $\mu_{k}=\mu_{k}^{\max}$. Moreover, $\sigma_{k}=\sigma_{k}^{\max}$ must also be attained after finitely many inner iterations. Let $\widehat{n}_{k}$ denote the number of inner iterations at which $\sigma_{k}=\sigma_{k}^{\max}$ is reached for the \textit{first} time. If $\sigma_{k}^{0}>\sigma_{k}^{\max}$, then $\widehat{n}_{k}=n_{k}$; if $\sigma_{k}^{0}=\sigma_{k}^{\max}$, then $\widehat{n}_{k}=0$; otherwise, we have
\begin{equation*}
 \sigma^{\min}\tau^{\widehat{n}_{k}-1} \leq \sigma_{k}^{0}\tau^{\widehat{n}_{k}-1} < \sigma_{k}^{\max},
\end{equation*}
which implies that
\begin{equation*}
\widehat{n}_{k} \leq \left\lfloor\frac{\log(\sigma_{k}^{\max}) - \log(\sigma^{\min})}{\log\tau} + 1  \right\rfloor.
\end{equation*}
Thus, after at most $\max\{n_{k}, \widehat{n}_{k}\} + 1$ inner iterations, we must have $V \equiv V_{\sigma_{k}^{\max}}$, where $V_{\sigma_{k}^{\max}}$ is computed by \eqref{supro:V-prox}, \eqref{supro:V-proxlinear} or \eqref{supro:V-hier} with $\sigma_{k} = \sigma_{k}^{\max}$. Therefore, after at most $\max\{n_{k}, \widehat{n}_{k}\} + 1$ inner iterations, we have from Lemma \ref{lem:Flambda-suffidescent} that
\begin{equation*}
\begin{aligned}
&\quad \mathcal{F}_\lambda(U_{\mu_{k}^{\max}}, \,V_{\sigma_{k}^{\max}})
- \mathcal{F}_\lambda (X^{k}, \,Y^{k} ) \\
&\leq -{\textstyle\frac{\mu_{k}^{\max} \,-\, (\alpha \,+\, 2\gamma\rho)\|Y^{k}\|^{2}}{2}}  \|U_{\mu_{k}^{\max}}-X^{k}\|_{F}^{2} - {\textstyle\frac{\sigma_{k}^{\max}\,-\,(\alpha \,+\, 2\gamma\rho)\|U_{\mu_{k}^{\max}}\|^{2}}{2}}\|V_{\sigma_{k}^{\max}}-Y^{k}\|_{F}^{2} \\
&= -{\textstyle\frac{c}{2}}\left(\|U_{\mu_{k}^{\max}}-X^{k}\|_{F}^{2}
+ \|V_{\sigma_{k}^{\max}}-Y^{k}\|_{F}^{2}\right),
\end{aligned}
\end{equation*}
where the equality follows from $\mu_k^{\max} = (\alpha + 2\gamma\rho)\|Y^k\|^2 +c$ and $\sigma_k^{\max} = (\alpha + 2\gamma\rho)\|U_{\mu_k^{\max}}\|^2 + c$.
This, together with \eqref{eq:omega>flambda}, further implies that
\begin{align*}
&\quad \mathcal{F}_\lambda (U_{\mu_{k}^{\max}}, \,V_{\sigma_{k}^{\max}}) - \mathcal{R}_k \\
&\leq\mathcal{F}_\lambda (U_{\mu_{k}^{\max}}, \,V_{\sigma_{k}^{\max}}) - \mathcal{F}_\lambda(X^{k}, \,Y^{k})
\leq - \frac{c}{2}\big(\|U_{\mu_{k}^{\max}}-X^{k}\|_{F}^{2}
+ \|V_{\sigma_{k}^{\max}}-Y^{k}\|_{F}^{2}\big),
\end{align*}
which means that for this $k$, the line search criterion \eqref{eq:A-linesearch} must be satisfied after at most $\max\{n_k,\widehat{n}_k\}+1$ inner iterations. This leads to a contradiction and completes the proof.
\end{proof}

\subsection{Proof of Lemma \ref{lem:dist-inequa}}\label{proof:lem:dist-inequa}

\begin{proof}
First, from the updating rule of $Z^k$ in \eqref{Zkupdate} and $\frac{1}{\alpha}+\frac{1}{\beta}=1$, we have
\begin{equation}\label{eq:A*A-b}
\textstyle
\mathcal{A}^*(\mathcal{A}(X^k (Y^k)^\top) -\bm{b}) = \frac{\alpha + \beta}{\beta} \left(X^k(Y^k)^\top-Z^k\right) = \alpha \left(X^k (Y^k)^\top - Z^k\right).
\end{equation}
We now characterize the subdifferential $\partial\mathcal{F}_\lambda(X^k, \,Y^k)$. For the partial subdifferential with respect to $X$ or $\bm{x}_i$, we consider the following three cases:
\begin{itemize}[leftmargin=0.5cm]
\item \textbf{Proximal:} In this case, we have
\begin{align*}
&\partial_X \mathcal{F}_\lambda(X^k,\,Y^k)
= \partial\Psi(X^k)
+ \mathcal{A}^*(\mathcal{A}(X^k (Y^k)^\top)-\bm{b})Y^k
+ \lambda(X^k-Y^k) \\
&= \partial\Psi(X^k) + \alpha(X^k(Y^k)^\top-Z^k)Y^k + \lambda(X^k-Y^k)\\
&= \partial\Psi(X^k) + \alpha(X^k(Y^{k-1})^\top\!-\!Z^{k-1})Y^{k-1}
+ \lambda(X^k\!-\!Y^{k-1}) + \bar{\mu}_{k-1}(X^k\!-\!X^{k-1})  
- \bar{\mu}_{k-1}(X^k\!-\!X^{k-1}) \\
&\quad - \lambda(Y^k-Y^{k-1})
+ \alpha(X^k(Y^k)^\top Y^k-X^k(Y^{k-1})^\top Y^{k-1})
- \alpha (Z^kY^k-Z^{k-1}Y^{k-1}) \\
&\ni -\bar{\mu}_{k-1}(X^k\!-\!X^{k-1})
- \lambda(Y^k\!-\!Y^{k-1})
+ \alpha X^k(Y^k\!-\!Y^{k-1})^\top Y^k
+ \alpha X^k(Y^{k-1})^\top (Y^k\!-\!Y^{k-1}) \\
&\quad - \alpha (Z^k-Z^{k-1})Y^k - \alpha Z^{k-1}(Y^k-Y^{k-1}),
\end{align*}
where the second equality follows from \eqref{eq:A*A-b} and the inclusion follows from the first-order optimality condition \eqref{supro:Uopt-prox} with $k=k-1$, $U=X^k$ and $\mu_{k-1}=\bar{\mu}_{k-1}$.

\item \textbf{Prox-linear:} In this case, we have
\begin{align*}
&\partial_X \mathcal{F}_\lambda(X^k,\,Y^k)
= \partial \Psi(X^k) + \mathcal{A}^*(\mathcal{A}(X^k (Y^k)^\top) - \bm{b}) Y^k+\lambda(X^k-Y^k) \\
&= \partial\Psi(X^k) + \alpha(X^k (Y^k)^\top - Z^k) Y^k + \lambda(X^k-Y^k)\\
&= \partial \Psi(X^k) + \alpha(X^{k-1}(Y^{k-1})^\top\!-\!Z^{k-1}) Y^{k-1}+\lambda(X^k\!-\!Y^{k-1}) + \bar{\mu}_{k-1}(X^k\!-\!X^{k-1}) 
- \bar{\mu}_{k-1}(X^k\!-\!X^{k-1}) \\
&\quad -\lambda(Y^k-Y^{k-1}) + \alpha (X^k (Y^k)^\top Y^k - X^{k-1} (Y^{k-1})^\top Y^{k-1}) - \alpha (Z^k Y^k - Z^{k-1} Y^{k-1}) \\
&\ni -\bar{\mu}_{k-1} (X^k\!-\!X^{k-1})
- \lambda(Y^k\!-\!Y^{k-1})
+ \alpha X^k(Y^k\!-\!Y^{k-1})^\top Y^k
+ \alpha X^{k-1} (Y^{k-1})^\top (Y^k\!-\!Y^{k-1})   \\
&\quad  + \alpha (X^k - X^{k-1}) (Y^{k-1})^\top Y^k
- \alpha (Z^k - Z^{k-1}) Y^k - \alpha Z^{k-1} (Y^k - Y^{k-1}),
\end{align*}
where the second equality follows from \eqref{eq:A*A-b} and the inclusion follows from the first-order optimality condition \eqref{supro:Uopt-proxlinear} with $k=k-1$, $U=X^k$ and $\mu_{k-1}=\bar{\mu}_{k-1}$.

\item \textbf{Hierarchical-prox:} In this case, for any $i = 1, 2, \ldots, r$, we have
\begin{align*}
&\partial_{\bm{x}_i} \mathcal{F}_\lambda(X^k, \,Y^k)
= \partial \psi_i(\bm{x}_i^k) + \mathcal{A}^*(\mathcal{A}(X^k (Y^k)^\top) - \bm{b}) \bm{y}_i^k+\lambda(\bm{x}_i^k-\bm{y}_i^k) \\
&= \partial \psi_i(\bm{x}_i^k) + \alpha (X^k (Y^k)^\top - Z^k) \bm{y}_i^k+\lambda(\bm{x}_i^k-\bm{y}_i^k) \\
&= \partial \psi_i(\bm{x}_i^k) + \alpha \left( {\textstyle\sum_{j=1}^i}\bm{x}_j^k (\bm{y}_j^{k-1})^\top + {\textstyle\sum_{j=i+1}^r}\bm{x}_j^{k-1} (\bm{y}_j^{k-1})^\top - Z^{k-1}\right) \bm{y}_i^{k-1} + \lambda(\bm{x}_i^k - \bm{y}_i^{k-1}) \\
&\quad  + \bar{\mu}_{k-1} (\bm{x}_i^k - \bm{x}_i^{k-1}) + \alpha {\textstyle\sum_{j=1}^i} (\bm{x}_j^{k-1} - \bm{x}_j^k) (\bm{y}_j^{k-1})^\top \bm{y}_i^{k-1}
+ \alpha (X^k (Y^k)^\top \bm{y}_i^k - X^{k-1} (Y^{k-1})^\top \bm{y}_i^{k-1}) \\
&\quad - \alpha (Z^k \bm{y}_i^k - Z^{k-1} \bm{y}_i^{k-1}) 
- \bar{\mu}_{k-1} (\bm{x}_i^k - \bm{x}_i^{k-1}) - \lambda(\bm{y}_i^k - \bm{y}_i^{k-1}) \\
&\ni \alpha{\textstyle\sum_{j=1}^i}(\bm{x}_j^{k-1}-\bm{x}_j^k) (\bm{y}_j^{k-1})^\top \bm{y}_i^{k-1}
+ \alpha (X^k - X^{k-1}) (Y^k)^\top \bm{y}_i^k 
+ \alpha X^{k-1} (Y^k - Y^{k-1})^\top \bm{y}_i^k \\
&\quad + \alpha X^{k-1} (Y^{k-1})^\top (\bm{y}_i^k - \bm{y}_i^{k-1}) 
- \alpha (Z^k-Z^{k-1})\bm{y}_i^k
- \alpha Z^{k-1}(\bm{y}_i^k-\bm{y}_i^{k-1}) \\
&\quad - \bar{\mu}_{k-1}(\bm{x}_i^k-\bm{x}_i^{k-1})
- \lambda(\bm{y}_i^k-\bm{y}_i^{k-1}),
\end{align*}
where the second equality follows from \eqref{eq:A*A-b} and the inclusion follows from the first-order optimality condition \eqref{supro:Uopt-hier}  with $k=k-1$, $\bm{u}_i=\bm{x}_i^k$ and $\mu_{k-1}=\bar{\mu}_{k-1}$.
\end{itemize}

Similarly, for the partial subdifferential with respect to $Y$ or $\bm{y}_i$, the following holds.
\begin{itemize}
\item \textbf{Proximal:} In this case, we have
\begin{equation*}
\begin{aligned}
\partial_Y \mathcal{F}_\lambda(X^k,\,Y^k)
&= \partial \Phi(Y^k) + (\mathcal{A}^*(\mathcal{A}(X^k (Y^k)^\top) - \bm{b}))^\top X^k-\lambda(X^k-Y^k) \\
&\ni -\bar{\sigma}_{k-1} (Y^k - Y^{k-1}) - \alpha (Z^k - Z^{k-1})^\top X^k.
\end{aligned}
\end{equation*}

\vspace{-2mm}
\item \textbf{Prox-linear:} In this case, we have
\begin{equation*}
\begin{aligned}
\partial_Y \mathcal{F}_\lambda(X^k, Y^k) 
&= \partial \Phi(Y^k) + (\mathcal{A}^*(\mathcal{A}(X^k (Y^k)^\top) - \bm{b}))^\top X^k -\lambda(X^k-Y^k) \\
&\ni -\bar{\sigma}_{k-1} (Y^k - Y^{k-1}) + \alpha (Y^k - Y^{k-1}) (X^k)^\top X^k - \alpha (Z^k - Z^{k-1})^\top X^k.
\end{aligned}
\end{equation*}

\vspace{-2mm}
\item \textbf{Hierarchical-prox:} In this case, for any $i = 1, 2, \ldots, r$, we have
\begin{equation*}
\begin{aligned}
&\partial_{\bm{y}_i} \mathcal{F}_\lambda(X^k, Y^k)= \partial \phi_i(\bm{y}_i^k)
+ (\mathcal{A}^*(\mathcal{A}(X^k (Y^k)^\top) - \bm{b}))^\top \bm{x}_i^k -\lambda(\bm{x}_i^k-\bm{y}_i^k) \\
&\ni \alpha {\textstyle\sum_{j=1}^i} (\bm{y}_j^{k-1}\!-\!\bm{y}_j^k) (\bm{x}_j^k)^\top \bm{x}_i^k + \alpha (Y^k\!-\!Y^{k-1}) (X^k)^\top \bm{x}_i^k 
- \alpha (Z^k\!-\!Z^{k-1})^\top \bm{x}_i^k - \bar{\sigma}_{k-1} (\bm{y}_i^k\!-\!\bm{y}_i^{k-1}).
\end{aligned}
\end{equation*}
\end{itemize}

Finally, note that for each $i$, $\|\bm{x}_i^k - \bm{x}_i^{k-1}\| \leq \|X^k - X^{k-1}\|_F$ and $\|\bm{y}_i^k - \bm{y}_i^{k-1}\| \leq \|Y^k - Y^{k-1}\|_F$. Moreover, from the updating rule of $Z^k$ in \eqref{Zkupdate},
\begin{align*}
\|Z^k - Z^{k-1}\|_F
&\leq \left\|{\textstyle\left(\mathcal{I} - \frac{\beta}{\alpha + \beta} \mathcal{A}^* \mathcal{A}\right)} \left(X^k (Y^k)^\top - X^{k-1} (Y^{k-1})^\top\right)\right\|_F \\
&\leq\sqrt{\rho}\,\|X^k\| \|Y^k - Y^{k-1}\|_F
+ \sqrt{\rho}\,\|Y^{k-1}\| \|X^k - X^{k-1}\|_F,
\end{align*}
where $\rho$ is defined in \eqref{eq:definerho}. Combining the above estimates with the boundedness of $\{X^k\}$, $\{Y^k\}$, $\{\bar{\mu}_k\}$ and $\{\bar{\sigma}_k\}$ (see Proposition \ref{prop:Omegak-descentbehavior}(v)), we obtain \eqref{eq:dist-inequa}. This completes the proof.
\end{proof}

\subsection{Proof of Theorem \ref{thm:fun_rate}}\label{proof:thm:fun_rate}

\begin{proof}
Let $\Delta^k_\mathcal{R}:=\mathcal{R}_k-\zeta$ for all $k\geq0$, it follows from Proposition \ref{prop:Omegak-descentbehavior}(ii) that $\{\Delta^k_\mathcal{R}\}$ is non-increasing and $\Delta^k_\mathcal{R}\geq0$. Then, for any $k\geq1$, we have that
\begin{equation}\label{Fxk-zeta}
\begin{aligned}
|\mathcal{F}_\lambda(X^k,\,Y^k)-\zeta|
&=\left|\mathcal{R}_{k-1}+{\textstyle\frac{1}{p_k}}(\mathcal{R}_{k}-\mathcal{R}_{k-1})-\zeta\right|
\leq\Delta^{k-1}_\mathcal{R}
+ {\textstyle\frac{1}{p_k}}(\mathcal{R}_{k-1}-\mathcal{R}_{k})\\
&=\Delta^{k-1}_\mathcal{R}
+ {\textstyle\frac{1}{p_k}}(\Delta^{k-1}_\mathcal{R}-\Delta^{k}_\mathcal{R})
\leq(1+\textstyle\frac{1}{p_{\min}})\Delta^{k-1}_\mathcal{R}
= d_1\Delta_\mathcal{R}^{k-1},
\end{aligned}
\end{equation}
where $d_1:=1+\textstyle\frac{1}{p_{\min}}$, the first equality follows from $\mathcal{F}_\lambda(X^k,\,Y^k)=\mathcal{R}_{k-1}+{\textstyle\frac{1}{p_k}}(\mathcal{R}_{k}-\mathcal{R}_{k-1})$ by the updating rule of $\mathcal{R}_{k}$, the last inequality follows from $\Delta^k_\mathcal{R}\geq0$ and $p_k\geq p_{\min}>0$ for all $k\geq 0$.

With \eqref{Fxk-zeta} in hand, we can characterize the convergence rate of $\{|\mathcal{F}_\lambda(X^k,\,Y^k)-\zeta|\}$ by analyzing the rate of $\{\Delta_\mathcal{R}^{k}\}$. To this end, we first consider the case where $\Delta_\mathcal{R}^{K_0}=0$ for some $K_0\geq0$. Since $\{\Delta_\mathcal{R}^{k}\}$ is non-increasing, it follows that $\Delta_\mathcal{R}^{k}=0$ for all $k \geq K_0$. This, together with \eqref{Fxk-zeta}, immediately proves all statements. For now on, we consider $\Delta_\mathcal{R}^{k}>0$ for all $k\geq0$.

First, we see from Theorem \ref{thm:fullconverge} that the whole sequence $\{(X^k, \,Y^k)\}$ is convergent. Let $(\widetilde{X},\,\widetilde{Y})$ be the limit point, we can conclude from Theorem \ref{thm:A-convergence} that $\mathcal{F}_\lambda(\widetilde{X},\,\widetilde{Y})=\zeta$. This, together with the assumption that $\mathcal{F}_\lambda$ is a KL function with an exponent $\theta$, implies that, there exist $\nu>0$, a neighborhood $\mathcal{V}$  of $(\widetilde{X},\,\widetilde{Y})$ and $\varphi\in\Phi_{\nu}$ such that 
\begin{equation}\label{phi-dist}
\varphi'\big(\mathcal{F}_\lambda(X,\,Y)-\zeta\big)
\cdot\mathrm{dist}\big(0,\,\partial \mathcal{F}_\lambda(X,\,Y)\big)\geq1,
~~\text{with}~~\varphi(s)=\tilde{a}s^{1-\theta}~\text{for some}~\tilde{a}>0
\end{equation}
for all $(X,\,Y)$ satisfying $(X,\,Y)\in \mathcal{V}$ and ${\zeta}<\mathcal{F}_\lambda(X,\,Y)<\zeta+\nu$. Next, we recall from Lemma \ref{lem:dist-inequa} that there exists $d>0$ such that
\begin{equation}\label{eq:dist-inequa1}
\operatorname{dist}\big(0, \,\partial\mathcal{F}_{\lambda}(X^k,\,Y^k)\big) 
\leq d\big( \|X^k - X^{k-1}\|_F + \|Y^k - Y^{k-1}\|_F \big),
\quad \forall\,k\geq 0.
\end{equation}
We also see from \eqref{omegaupdate} and $\mathcal{F}_\lambda(X^{k+1},\,Y^{k+1})\leq\mathcal{R}_{k+1}$ (by Proposition \ref{prop:Omegak-descentbehavior}(i)) that
\begin{equation}\label{Rkdiff-xkdiff-square}    
d_2\big(\|X^{k+1}-X^{k}\|_F^2 
+ \|Y^{k+1}-Y^{k}\|_F^2\big)
\leq\mathcal{R}_k - \mathcal{R}_{k+1}
=\Delta^k_\mathcal{R} - \Delta^{k+1}_\mathcal{R}
\leq \mathcal{R}_k-\mathcal{F}_\lambda(X^{k+1},\,Y^{k+1})
\end{equation}
for any $k\geq0$, where $d_2:=\frac{cp_{\min}}{2}$. Moreover, since $(X^k,\,Y^k)$ converges to $(\widetilde{X},\,\widetilde{Y})$ and the sequence $\{\mathcal{R}_k\}$ converges monotonically to $\zeta$ (by Proposition \ref{prop:Omegak-descentbehavior}(ii)), there exists an integer $K_1>0$ such that $\zeta<\mathcal{R}_k<\zeta+\min\left\{\nu,\, 1,\, \frac{d_2}{2\tilde{a}^2d^2}\right\}$ and $(X^k,\,Y^k)\in \mathcal{V}$ whenever $k\geq K_1$. In the following, for notational simplicity, let $\Delta^{k}_{\mathcal{F}_\lambda}:=\mathcal{F}_\lambda(X^k,\,Y^k)-\zeta$. Then, we see that $\Delta^{k}_{\mathcal{F}_\lambda}\leq\Delta^k_{\mathcal{R}}<1$ and $\Delta_\mathcal{R}^{k-1}<\frac{d_2}{2\tilde{a}^2d^2}$ hold for all $k\geq K_1+1$. With these preparations, we next proceed to prove the desired results and divide the proof into three steps.

\textit{\textbf{Step 1}}. First, we claim that for any $k\geq K_1$, the following statements hold:
\begin{itemize}
\item [(1a)] If $\mathcal{F}_\lambda(X^k,\,Y^k)\leq\zeta$, there exists $\rho_1\in(0,1)$ such that $\Delta_{\mathcal{R}}^{k}\leq\rho_1\Delta_{\mathcal{R}}^{k-1}$;
\item[(1b)] If $\mathcal{F}_\lambda(X^k,\,Y^k)>\zeta$, there exists $a_1>0$ such that $\big(\Delta_{\mathcal{F}_\lambda}^k\big)^{2\theta}\leq a_1\big(\Delta_\mathcal{R}^{k-1}-\Delta_\mathcal{R}^k\big)$.
\end{itemize}

 We first consider $\mathcal{F}_\lambda(X^k,\,Y^k)\leq\zeta$. In this case, together with the updating rule of $\mathcal{R}_k$, $0<p_{\min}\leq p_k\leq1$ and the fact that $\{\Delta^k_{\mathcal{R}}\}$ is non-increasing, we have that
\begin{align*}
\Delta_{\mathcal{R}}^k
&=p_k\mathcal{F}_\lambda(X^k,\,Y^k)
+ (1-p_k)\mathcal{R}_{k-1}-\zeta=
p_k\big(\mathcal{F}_\lambda(X^k,\,Y^k)-\zeta\big)
+ (1-p_k)\big(\mathcal{R}_{k-1}-\zeta\big)\\
&\leq(1-p_k)\Delta_{\mathcal{R}}^{k-1} 
\leq(1-p_{\min})\Delta_{\mathcal{R}}^{k-1}
=\rho_1\Delta^{k-1}_{\mathcal{R}},
\end{align*}
where $\rho_1:=1-p_{\min}\in(0,1)$. This shows that statement (1a) holds.

We next consider $\mathcal{F}_\lambda(X^k,\,Y^k)>\zeta$. In this case, it follows from Proposition \ref{prop:Omegak-descentbehavior}(i) and $k\geq K_1$ that $\zeta<\mathcal{F}_\lambda(X^k,\,Y^k)\leq\mathcal{R}_k<\zeta+\nu$ and $(X^k,\,Y^k)\in \mathcal{V}$. Thus, by \eqref{phi-dist}, we have
\begin{equation*}
\varphi'\big(\mathcal{F}_\lambda(X^k,\,Y^k)-\zeta\big)
\cdot\operatorname{dist}\big(0, \partial \mathcal{F}_\lambda(X^k,\,Y^k)\big) \geq 1.
\end{equation*}
Moreover, we see that
\begin{equation*}
\begin{aligned}
1&\leq\varphi'\big(\mathcal{F}_\lambda(X^k,\,Y^k)-\zeta\big)
\cdot\operatorname{dist}\big(0, \partial \mathcal{F}_\lambda(X^k,\,Y^k)\big)\\
&\leq \tilde{a}(1-\theta)\cdot(\Delta_{\mathcal{F}_\lambda}^k)^{-\theta}
\cdot d\big( \|X^k - X^{k-1}\|_F + \|Y^k - Y^{k-1}\|_F \big) \\
&\leq\tilde{a}d(1-\theta)\cdot(\Delta_{\mathcal{F}_\lambda}^k)^{-\theta}\cdot\sqrt{2\big( \|X^k - X^{k-1}\|_F^2 + \|Y^k - Y^{k-1}\|_F^2 \big)} \\
&\leq \sqrt{2/d_2}\tilde{a}d(1-\theta)\cdot(\Delta_{\mathcal{F}_\lambda}^k)^{-\theta}\cdot\sqrt{\mathcal{R}_{k-1}-\mathcal{R}_k} \\
&=\sqrt{2/d_2}\tilde{a}d(1-\theta)\cdot(\Delta_{\mathcal{F}_\lambda}^k)^{-\theta}\cdot\sqrt{\Delta_\mathcal{R}^{k-1}-\Delta_\mathcal{R}^k},
\end{aligned}
\end{equation*}
where the second inequality follows from \eqref{eq:dist-inequa1} and the last inequality follows from \eqref{Rkdiff-xkdiff-square}. This inequality further yields that
\begin{equation*}
\big(\Delta_{\mathcal{F}_\lambda}^k\big)^{2\theta}
\leq (2/d_2)\tilde{a}^2d^2(1-\theta)^2
\big(\Delta_\mathcal{R}^{k-1}-\Delta_\mathcal{R}^k\big)
= a_1\big(\Delta_\mathcal{R}^{k-1}-\Delta_\mathcal{R}^k\big),
\end{equation*}
where $a_1:=(2/d_2)\tilde{a}^2d^2(1-\theta)^2>0$. This shows that statement (1b).

\textit{\textbf{Step 2}}. Now, we claim that for any $k\geq K_1+1$, the following statements hold:
\begin{itemize}
\item [(2a)] If $\theta=0$, $\mathcal{F}_\lambda(X^k,\,Y^k)\leq\zeta$ and there exists $\rho_2\in(0,1)$ such that $\Delta_{\mathcal{R}}^{k}\leq\rho_2\Delta_{\mathcal{R}}^{k-1}$;
\item[(2b)] If $\theta\in(0,\frac{1}{2}]$, there exists $\rho_3\in(0,1)$ such that $\Delta_{\mathcal{R}}^{k}\leq\rho_3\Delta_{\mathcal{R}}^{k-1}$;
\item[(2c)] If $\theta\in(\frac{1}{2},1)$, there exists $a_2>0$ such that $(\Delta_{\mathcal{R}}^{k})^{1-2\theta}-(\Delta_{\mathcal{R}}^{k-1})^{1-2\theta}\geq a_2$.
\end{itemize}

\textit{Statement (2a)}. Suppose that $\theta=0$, we consider the following two cases.

We first consider $\mathcal{F}_\lambda(X^k,\,Y^k)\leq\zeta$. In this case, for any $k\geq K_1+1$, it follows from statement (1a) in \textit{\textbf{Step 1}} that there exists $\rho_1\in(0,1)$ such that $\Delta_{\mathcal{R}}^{k}\leq\rho_1\Delta_{\mathcal{R}}^{k-1}$. 

We next consider $\mathcal{F}_\lambda(X^k,\,Y^k)>\zeta$. In this case, we see from statement (1b) with $\theta=0$ in \textit{\textbf{Step 1}} and $\Delta_{\mathcal{R}}^{k}\geq0$ that, for any $k\geq K_1+1$,
\begin{equation*}
\textstyle \Delta_\mathcal{R}^{k-1}\geq\Delta_\mathcal{R}^{k-1}-\Delta_\mathcal{R}^k
\geq\frac{1}{a_1}=\frac{d_2}{2\tilde{a}^2d^2},
\end{equation*}
which contradicts to the fact that $\Delta_\mathcal{R}^{k-1}<\frac{d_2}{2\tilde{a}^2d^2}$ holds whenever $k\geq K_1+1$. Thus, this case cannot happen. Combining with these two cases, we prove statement (2a).

\textit{Statement (2b)}. Suppose that $\theta\in(0,\frac{1}{2}]$. We consider the following two cases. 

We first consider $\mathcal{F}_\lambda(X^k,\,Y^k)\leq\zeta$. In this case, for any $k\geq K_1+1$, it follows from the statement (1a) in \textit{\textbf{Step 1}} that there exists $\rho_1\in(0,1)$ such that $\Delta_{\mathcal{R}}^{k}\leq\rho_1\Delta_{\mathcal{R}}^{k-1}$, which gives the desired result. 

We next consider $\mathcal{F}_\lambda(X^k,\,Y^k)>\zeta$. In this case, for any $k\geq K_1+1$, it follows from $0<\Delta^{k}_{\mathcal{F}_\lambda}\leq\Delta^k_{\mathcal{R}}<1$, $2\theta\in(0,1]$ and statement (1b) in \textit{\textbf{Step 1}} that
\begin{equation*}
\textstyle \Delta_{\mathcal{F}_\lambda}^k\leq(\Delta_{\mathcal{F}_\lambda}^k)^{2\theta} \leq a_1(\Delta_\mathcal{R}^{k-1}-\Delta_\mathcal{R}^k)\leq a_1(\Delta_\mathcal{R}^{k-1}-\Delta_{\mathcal{F}_\lambda}^k)
\quad\Longrightarrow\quad
\Delta_{\mathcal{F}_\lambda}^k\leq\frac{a_1}{1+a_1}\Delta_\mathcal{R}^{k-1}.
\end{equation*}
This, along with the updating rule of $\mathcal{R}_k$ and $p_k\in[p_{\min},1]$ yields that
\begin{equation*}
\begin{aligned}
\Delta_\mathcal{R}^k
&= p_k\mathcal{F}_\lambda(X^k,\,Y^k)+(1-p_k)\mathcal{R}_{k-1}-\zeta
= p_k\Delta_{\mathcal{F}_\lambda}^k+(1-p_k)\Delta_\mathcal{R}^{k-1}\\
&\leq {\textstyle\left(\frac{a_1}{1+a_1}p_k+1-p_k\right)}\Delta_\mathcal{R}^{k-1}
= {\textstyle\left(1-\frac{p_k}{1+a_1}\right)}\Delta_\mathcal{R}^{k-1} 
\leq {\textstyle\left(1-\frac{p_{\min}}{1+a_1}\right)}\Delta_\mathcal{R}^{k-1}.
\end{aligned}
\end{equation*}
Combining with the above two cases, we can conclude that $\Delta_{\mathcal{R}}^{k}\leq\rho_3\Delta_{\mathcal{R}}^{k-1}$, where $\rho_3:=\max\left\{\rho_1,1-\frac{p_{\min}}{1+a_1}\right\}\in(0,1)$. This shows that statement (2b) holds.

\textit{Statement (2c)}. Suppose that $\theta\in(\frac{1}{2},1)$, we consider the following two cases.

We first consider $\mathcal{F}_\lambda(X^k,\,Y^k)\leq\zeta$. In this case, for any $k\geq K_1+1$, it follows from statement (1a) in \textit{\textbf{Step 1}} that there exists $\rho_1\in(0,1)$ such that $\Delta_{\mathcal{R}}^{k}\leq\rho_1\Delta_{\mathcal{R}}^{k-1}$. Since $1-2\theta<0$ and $\Delta_{\mathcal{R}}^{k-1},\,\Delta_{\mathcal{R}}^{k}>0$, we further have that
\begin{equation*}
\big(\Delta_{\mathcal{R}}^{k}\big)^{1-2\theta}
\geq \rho_1^{1-2\theta}\big(\Delta_{\mathcal{R}}^{k-1}\big)^{1-2\theta},
\end{equation*}
which implies that
\begin{equation*}
\big(\Delta_{\mathcal{R}}^{k}\big)^{1-2\theta}
- \big(\Delta_{\mathcal{R}}^{k-1}\big)^{1-2\theta}
\geq \big({\rho_1^{1-2\theta}}-1\big)\big(\Delta_{\mathcal{R}}^{k-1}\big)^{1-2\theta}
\geq\big({\rho_1^{1-2\theta}}-1\big)\big(\Delta_{\mathcal{R}}^{K_1}\big)^{1-2\theta}
>0,
\end{equation*}
where the second inequaliy follows from the facts that $\{\Delta_{\mathcal{R}}^k\}$  is non-increasing, $k\geq K_1+1$, $\rho_1\in(0,1)$, $1-2\theta<0$, and ${\rho_1^{1-2\theta}}-1>0$. This gives the desired result.

We next consider $\mathcal{F}_\lambda(X^k,\,Y^k)>\zeta$. In this case, for any $k\geq K_1+1$, $\Delta_{\mathcal{F}_\lambda}^k>0$ and it follows from statement (1b) in \textit{\textbf{Step 1}} that
\begin{equation}\label{frac-1-a1-leq}
a_1^{-1}\leq\big(\Delta_{\mathcal{F}_\lambda}^k\big)^{-2\theta} \big(\Delta_{\mathcal{R}}^{k-1}-\Delta_{\mathcal{R}}^{k}\big).
\end{equation}
Next, we define $g(s):=s^{-2 \theta}$ for $s \in(0, \infty)$. It is easy to see that $g$ is non-increasing. Then, for any $k \geq K_1+1$, we further consider the following two cases.
\begin{itemize}
\item If $g(\Delta_{\mathcal{F}_\lambda}^k) \leq 2 g(\Delta_{\mathcal{R}}^{k-1})$, it follows from \eqref{frac-1-a1-leq} that
\begin{equation*}
\begin{aligned}
a_1^{-1}
&\leq g(\Delta_{\mathcal{F}_\lambda}^k)\big(\Delta_{\mathcal{R}}^{k-1}
-\Delta_{\mathcal{R}}^{k}\big)
\leq 2g(\Delta_{\mathcal{R}}^{k-1}) \big(\Delta_{\mathcal{R}}^{k-1}-\Delta_{\mathcal{R}}^{k}\big) \\
&\textstyle \leq 2 \int_{\Delta_{\mathcal{R}}^{k}}^{\Delta_{\mathcal{R}}^{k-1}} g(s)\,\mathrm{d}s
=\frac{2(\Delta_{\mathcal{R}}^{k-1})^{1-2\theta}\,-\,2(\Delta_{\mathcal{R}}^{k})^{1-2\theta}}{1-2 \theta},
\end{aligned}
\end{equation*}
which, together with $1-2\theta<0$, implies that
\begin{equation*}
(\Delta_{\mathcal{R}}^{k})^{1-2\theta}
-(\Delta_{\mathcal{R}}^{k-1})^{1-2\theta}
\geq (2\theta-1)/(2a_1).
\end{equation*}

\item If $g(\Delta_{\mathcal{F}_\lambda}^k) > 2 g(\Delta_{\mathcal{R}}^{k-1})$, it follows that $\Delta_{\mathcal{F}_\lambda}^k<2^{-\frac{1}{2\theta}}\Delta_{\mathcal{R}}^{k-1}$.
This, along with the updating rule of $\mathcal{R}_k$ and $p_k\in[p_{\min},1]$, yields that
\begin{equation*}
\begin{aligned}
\Delta_\mathcal{R}^k
&= p_k\mathcal{F}_\lambda(X^k,\,Y^k)+(1-p_k)\mathcal{R}_{k-1}-\zeta
= p_k\Delta_{\mathcal{F}_\lambda}^k+(1-p_k)\Delta_\mathcal{R}^{k-1}\\
&\leq \big[1-\big(1-2^{-\frac{1}{2\theta}}\big)p_k\big]\Delta_\mathcal{R}^{k-1}
\leq \big[1-\big(1-2^{-\frac{1}{2\theta}}\big)p_{\min}\big]\Delta_\mathcal{R}^{k-1}
=d_3\Delta_\mathcal{R}^{k-1},
\end{aligned}
\end{equation*}
where $d_3:=1-\big(1-2^{-\frac{1}{2\theta}}\big)p_{\min}\in(0,1)$. This, together with $\Delta_{\mathcal{R}}^{k-1}$, $\Delta_{\mathcal{R}}^{k}>0$, implies that
\begin{equation*}
(\Delta_{\mathcal{R}}^{k})^{1-2\theta}-(\Delta_{\mathcal{R}}^{k-1})^{1-2\theta}
\geq(d_3^{1-2\theta}-1)(\Delta_{\mathcal{R}}^{k-1})^{1-2\theta}
\geq(d_3^{1-2\theta}-1)(\Delta_{\mathcal{R}}^{K_1})^{1-2\theta}>0,
\end{equation*}
where the second inequality follows from the facts that $\{\Delta_{\mathcal{R}}^k\}$ is non-increasing, $k\geq K_1+1$, $d_3\in(0,1)$, $1-2\theta<0$ and $d_3^{1-2\theta}-1>0$.
\end{itemize}

In view of the above, we have that
\begin{equation*}
(\Delta_{\mathcal{R}}^{k})^{1-2\theta}-(\Delta_{\mathcal{R}}^{k-1})^{1-2\theta}
\geq a_2:=\min\left\{\big({\rho_1^{1-2\theta}}-1\big)(\Delta_{\mathcal{R}}^{K_1})^{1-2\theta},
\,\textstyle\frac{2\theta-1}{2a_1},
\,\big({d_3^{1-2\theta}}-1\big)(\Delta_{\mathcal{R}}^{K_1})^{1-2\theta}\right\}.
\end{equation*}
This shows that statement (2c) holds.

\textit{\textbf{Step 3}}. We are now ready to prove our final results.

\textit{Statement (i)}. Suppose that $\theta=0$. For any $k\geq K_1+2$, combining statement (2a) in \textit{\textbf{Step 2}}, \eqref{Fxk-zeta} and the fact that $\{\Delta_{\mathcal{R}}^k\}$ is non-increasing, we have that $\mathcal{F}_\lambda(X^k,\,Y^k)\leq\zeta$ and
\begin{equation*}
\begin{aligned}
\max\left\{|\mathcal{F}_\lambda(X^k,\,Y^k)-\zeta|,\mathcal{R}_k-\zeta\right\}
&\leq \max\{d_1,1\}\Delta_\mathcal{R}^{k-1}
\leq \max\{d_1,1\}\rho_2^{k-K_1-1} \Delta_\mathcal{R}^{K_1}=c_1\eta_1^k,
\end{aligned}
\end{equation*}
where $c_1:=\max\{d_1,1\}\rho_2^{-K_1-1} \Delta_\mathcal{R}^{K_1}$ and $\eta_1:=\rho_2\in(0,1)$. This also implies that $\zeta-c_1\eta_1^k\leq \mathcal{F}_\lambda(X^k,\,Y^k)\leq \zeta$, and thus proves statement (i).

\textit{Statement (ii)}. Suppose that $\theta\in(0,\frac{1}{2}]$. Using similar arguments as in the above case, we can obtain the desired result in statement (ii).

\textit{Statement (iii)}. Suppose that $\theta\in(\frac{1}{2},1)$. For any $k\geq K_1+1$, combining statement (2c) in \textit{\textbf{Step 2}} and the nonnegativity of $\{\Delta_{\mathcal{R}}^k\}$, we have
\begin{equation*}
\begin{aligned}
(\Delta_{\mathcal{R}}^k)^{1-2 \theta} 
&\textstyle \geq(\Delta_{\mathcal{R}}^k)^{1-2 \theta}-(\Delta_{\mathcal{R}}^{K_1})^{1-2 \theta}
= \sum_{j=1}^{k-K_1}\left((\Delta_{\mathcal{R}}^{K_1+j})^{1-2 \theta}-(\Delta_{\mathcal{R}}^{K_1+j-1})^{1-2\theta}\right) \\
&\textstyle \geq (k-K_1) a_2 \geq \frac{a_2}{2} k,
\end{aligned}
\end{equation*}
where the last inequality holds whenever $k \geq 2K_1$. Finally, using this relation, \eqref{Fxk-zeta} and the fact that $\{\Delta_{\mathcal{R}}^k\}$ is non-increasing, we see that, for all $k \geq 2K_1+2$,
\begin{equation*}
\begin{aligned}
&\max\left\{|\mathcal{F}_\lambda(X^k,\,Y^k)-\zeta|,\mathcal{R}_k-\zeta\right\}\\
\leq& \max\{d_1,1\} \Delta_{\mathcal{R}}^{k-1}
\leq \max\{d_1,1\}(2/a_2)^{\frac{1}{2\theta-1}}(k-1)^{-\frac{1}{2\theta-1}}
\\
\leq& \max\{d_1,1\}(2/a_2)^{\frac{1}{2\theta-1}}[k/(k-1)]^{\frac{1}{2\theta-1}}\cdot k^{-\frac{1}{2 \theta-1}}
\leq \max\{d_1,1\}(4/a_2)^{\frac{1}{2\theta-1}}k^{-\frac{1}{2\theta-1}}=c_3k^{-\frac{1}{2 \theta-1}},
\end{aligned}
\end{equation*}
where $c_3:=\max\{d_1,1\}(4/a_2)^{\frac{1}{2\theta-1}}$, and the last inequality follows from $\frac{k}{k-1}\leq2$ and $\frac{1}{2\theta-1}\geq0$. This proves statement (iii).    
\end{proof}

\subsection{Proof of Theorem \ref{thm:iter_rate}}\label{proof:thm:iter_rate}

\begin{proof}
First, recall the notation used in the previous analysis that
\begin{equation*}
\begin{aligned}
M&:=\left \lceil\textstyle\frac{1+\sqrt{1-p_{\min}}}{1-\sqrt{1-p_{\min}}}  \right \rceil ^2, \quad\ell(k):=k+M-1, \quad\Xi_{k}:=\sqrt{\mathcal{R}_{k}-\mathcal{R}_{k+1}},\\
\Delta_{i,j}^{\varphi}&:=\varphi(\mathcal{R}_i-\zeta)-\varphi(\mathcal{R}_j-\zeta), \quad\pi:= \sqrt{cp_{\min}}/2,
\quad \Delta^k_{\mathcal{R}}:=\mathcal{R}_k-\zeta, 
\end{aligned}
\end{equation*}
where $\lceil a \rceil$ denotes the smallest integer greater than or equal to $a$. From Theorem \ref{thm:fullconverge} and the proof proceeding it, we have that the whole sequence $\{(X^k, \,Y^k)\}$ is convergent (let $(\widetilde{X},\,\widetilde{Y})$ be the limit point),
\begin{equation}\label{xk-xkinequal1}
\textstyle\|X^{k+1}-X^{k}\|_F+\|Y^{k+1}-Y^{k}\|_F\leq\frac{\Xi_k}{\pi},
\end{equation}
\begin{equation}\label{M-relation-App}
\big(1-\sqrt{1-p_{\min}}\big)\sqrt{M}
- \big(\textstyle\frac{1}{2}+\sqrt{1-p_{\min}}\big)
\geq \textstyle\frac{1}{2},
\end{equation}
and there exists a sufficiently large integer $K_1$ such that
\begin{equation}\label{ineq-sum-Xi}
\textstyle\frac{1-\sqrt{1-p_{\min }}}{\sqrt{M}}\sum_{i=k}^{\ell(k)}\Xi_i
\leq \left(\frac{1}{2}+\sqrt{1-p_{\min}}\right)
\Xi_{k-1}
+ \frac{d}{2\pi}\Delta^{\varphi}_{k,k+M},
\quad \forall\,k\geq K_1.
\end{equation}
Moreover, since $\{\mathcal{R}_k\}$ converges non-increasingly to $\zeta$ (by Proposition \ref{prop:Omegak-descentbehavior}(ii)) and Theorem \ref{thm:fun_rate}(i)\&(ii)\&(iii) hold for all sufficiently large $k$, there exists an integer $K_2$ such that $\zeta\leq\mathcal{R}_k\leq\zeta+1$ (i.e., $0\leq\Delta_\mathcal{R}^k\leq1$) and Theorem \ref{thm:fun_rate}(i)\&(ii)\&(iii) hold whenever $k\geq K_2$.

Next, we claim that there exist $t_1>0$ and $t_2>0$ such that the following inequality holds: 
\begin{equation}\label{xk-xtild}
\|X^k-\widetilde{X}\|_F+\|Y^k-\widetilde{Y}\|_F\leq t_1(\Delta_\mathcal{R}^{k-1})^{\frac{1}{2}}
+t_2(\Delta_\mathcal{R}^{k-1})^{1-\theta},
\quad \forall\,k\geq K_1.
\end{equation}
Indeed, for any $k \geq K_1$, we see that
\begin{equation}\label{ineqadd-tildek}
\begin{aligned}
&\textstyle \quad \big(1-\sqrt{1-p_{\min}})\sqrt{M} \sum_{i=\ell(k)}^{\tilde{k}}\Xi_i
\leq \frac{1-\sqrt{1-p_{\min}}}{\sqrt{M}}
\sum_{i=k}^{\tilde{k}}\sum_{t=i}^{\ell(i)}\Xi_t  \\
&\textstyle \leq \left(\frac{1}{2}+\sqrt{1-p_{\min}}\right) \sum_{i=k}^{\tilde{k}}\Xi_{i-1}
+ \frac{d}{2\pi}\sum_{i=k}^{\tilde{k}}
\Delta^{\varphi}_{i,i+M}  \\
&\textstyle\leq \left(\frac{1}{2}+\sqrt{1-p_{\min}}\right)\sum_{i=k-1}^{\tilde{k}-1}\Xi_i
+ \frac{d\tilde{a}}{2\pi}\sum_{i=k}^{\ell(k)} \big(\mathcal{R}_i-\zeta\big)^{1-\theta}\\
&\textstyle =\left(\frac{1}{2}+\sqrt{1-p_{\min}}\right) \left(\sum_{i=k-1}^{\ell(k)-1}\Xi_i+\sum_{i=\ell(k)}^{\tilde{k}}\Xi_i\right)
+ \frac{d\tilde{a}}{2\pi}\sum_{i=k}^{\ell(k)} \big(\mathcal{R}_i-\zeta\big)^{1-\theta}
\end{aligned}
\end{equation}
holds for all $\tilde{k}\geq\ell(k)$,
where the first inequality follows from the nonnegativity of $\Xi_k$ and the fact that the term $\Xi_i$ with $i\in\{\ell(k), \cdots, \tilde{k}\}$ occurs $M$ times in the double sum; the second inequality is obtained by applying \eqref{ineq-sum-Xi} to each $k\in\{k,k+1,\cdots,\tilde{k}\}$; the third inequality holds because
\begin{equation*}
\begin{aligned}
&\textstyle\sum_{i=k}^{\tilde{k}}\Delta^{\varphi}_{i,i+M}
\textstyle=\sum_{i=k}^{\tilde{k}}\varphi(\mathcal{R}_i-\zeta)-\sum_{i=k}^{\tilde{k}}\varphi(\mathcal{R}_{i+M}-\zeta)
=\sum_{i=k}^{\tilde{k}}\varphi(\mathcal{R}_i-\zeta)-\sum_{i=k+M}^{\tilde{k}+M}\varphi(\mathcal{R}_i-\zeta) \\
&~~\textstyle=\sum_{i=k}^{{k}+M-1}\varphi(\mathcal{R}_i-\zeta)-\sum_{i=\tilde{k}+1}^{\tilde{k}+M}\varphi(\mathcal{R}_i-\zeta)
\leq\sum_{i=k}^{\ell({k})}\varphi(\mathcal{R}_i-\zeta)=\sum_{i=k}^{\ell({k})}\tilde{a}(\mathcal{R}_i-\zeta)^{1-\theta},
\end{aligned}
\end{equation*}
where the last equality follows from $\varphi(s)=\tilde{a}s^{1-\theta}$.
Using \eqref{ineqadd-tildek}, together with \eqref{M-relation-App}, we obtain
\begin{equation*}
\textstyle\sum_{i=\ell(k)}^{\tilde{k}}\Xi_i
\leq \big(1+2\sqrt{1-p_{\min}}\big)\sum_{i=k-1}^{\ell(k)-1}\Xi_i
+ \frac{d\tilde{a}}{\pi}\sum_{i=k}^{\ell({k})}(\mathcal{R}_i-\zeta)^{1-\theta},
\end{equation*}
which further implies that
\begin{equation}\label{Ei-k-tildek}
\textstyle\sum_{i=k}^{\tilde{k}}\Xi_i= \sum_{i=k}^{\ell(k)-1}\Xi_i
+ \sum_{i=\ell(k)}^{\tilde{k}}\Xi_i
\leq \big(2+2\sqrt{1-p_{\min}}\big) \sum_{i=k-1}^{\ell(k)-1}\Xi_i
+ \frac{d\tilde{a}}{\pi}\sum_{i=k}^{\ell(k)} \big(\mathcal{R}_i-\zeta\big)^{1-\theta}.
\end{equation}
Moreover, it follows from  $\ell(k)=k+M-1$ and the fact that $\{\mathcal{R}_k\}$ converges non-increasingly to $\zeta$ (by Proposition \ref{prop:Omegak-descentbehavior}(ii)) that
\begin{equation*}
\textstyle\sum_{i=k-1}^{\ell(k)-1}\Xi_i=\sum_{i=k-1}^{\ell(k)-1}\sqrt{\mathcal{R}_i-\mathcal{R}_{i+1}}\leq\sum_{i=k-1}^{\ell(k)-1}\sqrt{\mathcal{R}_i-\zeta}\leq M \sqrt{\mathcal{R}_{k-1}-\zeta}
=M (\Delta_\mathcal{R}^{k-1})^{\frac{1}{2}}
\end{equation*}
and
\begin{equation*}
\textstyle\sum_{i=k}^{\ell({k})}(\mathcal{R}_i-\zeta)^{1-\theta}\leq M(\mathcal{R}_k-\zeta)^{1-\theta}\leq M(\mathcal{R}_{k-1}-\zeta)^{1-\theta}= M(\Delta_{\mathcal{R}}^{k-1})^{1-\theta}.
\end{equation*}
Using the two above relations and passing to the limit $\tilde{k}\to\infty$ in \eqref{Ei-k-tildek}, we further have
\begin{equation*}
\begin{aligned}
\textstyle\sum_{i=k}^{\infty}\Xi_i
&\textstyle\leq \big(2+2\sqrt{1-p_{\min}}\big)\sum_{i=k-1}^{\ell(k)-1}\Xi_i
+ \frac{d\tilde{a}}{\pi}\sum_{i=k}^{\ell(k)} \big(\mathcal{R}_i-\zeta\big)^{1-\theta}\\
&\textstyle\leq\big(2+2\sqrt{1-p_{\min}}\big)M(\Delta_\mathcal{R}^{k-1})^{\frac{1}{2}}
+ \frac{d\tilde{a}M}{\pi} (\Delta_\mathcal{R}^{k-1})^{1-\theta}.
\end{aligned}
\end{equation*}
This, together with the triangle inequality and \eqref{xk-xkinequal1}, yields
\begin{equation*}
\begin{aligned}
\|X^k-\widetilde{X}\|_F+\|Y^k-\widetilde{Y}\|_F 
&\textstyle\leq\sum_{i=k}^{\infty}\left(\|X^{i+1}-X^i\|_F+\|Y^{i+1}-Y^i\|_F\right)
\leq\frac{1}{\pi}\sum_{i=k}^{\infty}\Xi_i \\
&\textstyle\leq\frac{\big(2+2\sqrt{1-p_{\min}}\big)M}{\pi}(\Delta_\mathcal{R}^{k-1})^{\frac{1}{2}}
+ \frac{d\tilde{a}M}{\pi^2} (\Delta_\mathcal{R}^{k-1})^{1-\theta}.
\end{aligned}
\end{equation*}
This proves \eqref{xk-xtild} with $t_1:=\pi^{-1}\big(2+2\sqrt{1-p_{\min}}\big)M$ and $t_2:=\pi^{-2}d\tilde{a}M$.

With the above inequality in hand, we are now ready to prove our final results.

\textit{Statement (i)}. Suppose that $\theta\in[0,\frac{1}{2}]$. In this case, for any $k\geq\max\{K_1,K_2\}+1$, we have that $1-\theta\geq\frac{1}{2}$, $\Delta_{\mathcal{R}}^{k-1}\leq1$, Theorem \ref{thm:fun_rate}(i)\&(ii) and the inequality \eqref{xk-xtild} hold. Then,
\begin{equation*}
\begin{aligned}
\|X^k-\widetilde{X}\|_F+\|Y^k-\widetilde{Y}\|_F
\leq t_1(\Delta_\mathcal{R}^{k-1})^{\frac{1}{2}}
+t_2(\Delta_\mathcal{R}^{k-1})^{1-\theta}
\leq (t_1+t_2)(\Delta_\mathcal{R}^{k-1})^{\frac{1}{2}}.
\end{aligned}
\end{equation*}
Combining this with Theorem \ref{thm:fun_rate}(i)\&(ii), we obtain
\begin{equation*}
\|X^k-\widetilde{X}\|_F+\|Y^k-\widetilde{Y}\|_F
\leq(t_1+t_2)\max\{\sqrt{c_1},\sqrt{c_2}\}{\max\{\sqrt{\eta_1},\sqrt{\eta_2}\}}^{k-1}=d_1\varrho^k,
\end{equation*}
where $d_1:=(t_1+t_2)\max\{\sqrt{c_1},\sqrt{c_2}\}{\max\{\sqrt{\eta_1},\sqrt{\eta_2}\}}^{-1}>0$ and $\varrho:=\max\{\sqrt{\eta_1},\sqrt{\eta_2}\}\in(0,1)$.
This proves statement (i).

\textit{Statement (ii)}. Suppose that $\theta\in(\frac{1}{2},1)$. In this case, for any $k\geq\max\{K_1,K_2\}+2$, we have that $1-\theta<\frac{1}{2}$, $\Delta_{\mathcal{R}}^{k-1}\leq1$, Theorem \ref{thm:fun_rate}(iii) and the inequality \eqref{xk-xtild} hold. Then, 
\begin{equation*}
\|X^k-\widetilde{X}\|_F+\|Y^k-\widetilde{Y}\|_F
\leq t_1(\Delta_\mathcal{R}^{k-1})^{\frac{1}{2}}
+ t_2 (\Delta_\mathcal{R}^{k-1})^{1-\theta}\leq (t_1+t_2)(\Delta_\mathcal{R}^{k-1})^{1-\theta}.
\end{equation*}
Combining this with Theorem \ref{thm:fun_rate}(iii), we obtain
\begin{equation*}
\begin{aligned}
\|X^k-\widetilde{X}\|_F+\|Y^k-\widetilde{Y}\|_F
&\leq(t_1+t_2)c_3^{1-\theta}{(k-1)}^{-\frac{1-\theta}{2\theta-1}}\leq(t_1+t_2)c_3^{1-\theta}{[k/(k-1)]}^{\frac{1-\theta}{2\theta-1}}\cdot{k}^{-\frac{1-\theta}{2\theta-1}}\\
&\leq 2^{\frac{1-\theta}{2\theta-1}}(t_1+t_2)c_3^{1-\theta}{k}^{-\frac{1-\theta}{2\theta-1}}= d_2{k}^{-\frac{1-\theta}{2\theta-1}},
\end{aligned}
\end{equation*}
where $d_2:=2^{\frac{1-\theta}{2\theta-1}}(t_1+t_2)c_3^{1-\theta}>0$, and the last inequality follows from $\frac{k}{k-1}\leq2$ and $\frac{1-\theta}{2\theta-1}>0$. This proves statement (ii).    
\end{proof}

\bibliographystyle{plain}
\bibliography{Ref_GSymMF}

\end{document}